\documentclass[11pt,american,english]{amsart}
\usepackage[T1]{fontenc}
\usepackage[latin9]{inputenc}
\usepackage[a4paper]{geometry}
\usepackage{babel}
\usepackage{tipa}
\usepackage{amsmath}
\usepackage{amsthm}
\usepackage{amssymb}
\usepackage{setspace}
\usepackage{cases}
\usepackage[unicode=true,pdfusetitle,
 bookmarks=true,bookmarksnumbered=true,bookmarksopen=true,bookmarksopenlevel=3,
 breaklinks=false,pdfborder={0 0 1},backref=false,colorlinks=false]
 {hyperref}
\usepackage{xcolor}
\usepackage{bbm}

\usepackage{etoolbox}
\newcommand{\zerodisplayskips}{%
  \setlength{\abovedisplayskip}{1pt}%
  \setlength{\belowdisplayskip}{1pt}%
  \setlength{\abovedisplayshortskip}{1pt}%
  \setlength{\belowdisplayshortskip}{1pt}}
\appto{\normalsize}{\zerodisplayskips}
\appto{\small}{\zerodisplayskips}
\appto{\footnotesize}{\zerodisplayskips}

\makeatletter
\newtheorem{theorem}{Theorem}[section]

\newtheorem{lemma}[theorem]{Lemma}
\newtheorem{proposition}[theorem]{Proposition}
\theoremstyle{definition}

\newtheorem{remark}[theorem]{Remark}

\numberwithin{equation}{section}

\usepackage{multirow}

\usepackage{arydshln}

\usepackage{lmodern}
\newcommand{\1}{\mbox{1\hspace{-1mm}I}}
\numberwithin{equation}{section}

\usepackage{calc}
\makeatother

\addto\captionsamerican{}
\addto\captionsamerican{}
\addto\captionsamerican{}
\addto\captionsamerican{}
\addto\captionsenglish{}
\addto\captionsenglish{}
\addto\captionsenglish{}
\addto\captionsenglish{}

\begin{document}
\selectlanguage{american}%
\global\long\def\1{\mbox{1\hspace{-1mm}I}}%

\title{CONTROL ON HILBERT SPACES AND MEAN FIELD CONTROL: THE COMMON NOISE
CASE }
\author{Alain Bensoussan\\
International Center for Decision and Risk Analysis\\
Naveen Jindal School of Management, University of Texas at Dallas\\
\lowercase{axb046100@utdallas.edu}\\ \\
P. Jameson Graber\\
Department of Mathematics, Baylor University\\
J\lowercase{ameson}\_G\lowercase{raber@baylor.edu}\\ \\
Sheung Chi Phillip Yam\\
Department of Statistics, The Chinese University of Hong Kong\\
\lowercase{scpyam@sta.cuhk.edu.hk}}
\thanks{Alain Bensoussan gratefully acknowledges the support of the National Science Foundation through grant DMS-2204795. Jameson Graber gratefully acknowledge the support of the National Science Foundation through grant DMS-2045027. Phillip Yam acknowledges the financial support from HKGRF-14301321 with the project title 
	``General Theory for Infinite Dimensional Stochastic Control: Mean Field and Some Classical Problems'', and 
	HKGRF-14300123 with the project title 
	``Well-posedness of Some Poisson-driven Mean Field Learning Models and their Applications''. 
	He also thanks the University of Texas at Dallas for the kind invitation to be a Visiting Professor in the Naveen Jindal School of Management during his sabbatical leave. 
	The work described in this article was supported by a grant from the Germany/Hong Kong Joint Research Scheme sponsored by the Research Grants Council of Hong Kong and the German Academic Exchange Service of Germany (Reference No. G-CUHK411/23).}

\maketitle
\begin{abstract}
The objective of this paper is to provide an equivalent of the theory developed in P.~Cardaliaguet, F.~Delarue, J.M.~Lasry, P.L.~Lions \cite{CDLL},
following the approach of control on Hilbert spaces introduced by
the authors in \cite{BGY-2}. We include the common noise in this
paper, so the alternative is now complete. Since we consider a control
problem, our theory applies only to Mean field control and not to
mean field games. The assumptions are adapted to guarantee a unique
optimal control, so they insure that the cost functional is strictly
convex and coercive. 
\end{abstract}

\section{INTRODUCTION }

A mean field control problem is a control problem for a dynamic system
whose state is a probability measure on $\mathbb{R}^{n}.$ The evolution is
described by a McKean-Vlasov equation. If one does not assume that
the probability measures have densities, the natural functional space
for the state of the system is the Wasserstein space of probability
measures on $\mathbb{R}^{n}$. Since it is a metric space and not a vector space, the classical methods of control theory are difficult to apply. One way to circumvent this difficulty, is to use the lifting idea of P.L.~Lions, introduced in his lectures at College de France, see
R.~Carmona, F.~Delarue \cite{RCD} and P.~Cardialaguet et al. \cite{CDLL}
for full details. To a probability measure, one associates a random
variable whose probability is the initial probability measure. Assuming
that the random variable is square integrable, we have a convenient
Hilbert space associated to the Wasserstein space.

One way to implement
this approach in mean field control is to derive first an optimality
principle of Dynamic Progamming in the Wasserstein space, to translate
it into the Hilbert space of square-integrable random variables and
use a viscosity theory argument, to finally obtain a Bellman equation
in the Hilbert space. Viscosity theory applies indeed to Bellman equations
in Hilbert spaces. This is the approach used in Pham and Wei \cite{PHW}.
{The theory of viscosity solutions for Bellman equations on the Wasserstein space has been a subject of active research in recent years. Without giving an exhaustive list of works, we refer the reader to \cite{B,CGKPR,CHH,DS,DJS} and works cited therein. The present article deals mainly with smooth, classical solutions rather than viscosity solutions.}

Our approach in \cite{ABY}, for two of the authors, then in \cite{BGY}
is different. We reformulate the control problem in the Wasserstein
space into a control problem in the Hilbert space and solve it up
to the Bellman equation, without using the Wasserstein space. When
sufficient smoothness is available, viscosity theory is not needed.

In \cite{ABY}, we assume there is no local noise, nor common noise,
which means that randomness of the dynamics originates only from the
initial condition. In this case, the control problem reduces to a
deterministic control problem in the Hilbert space. In fact, we can
take a general Hilbert space. The approach works remarkably well.
We have obtained the Bellman equation and the Master equation and
solved them completely. The fact that the Hilbert space is a space
of square integrable random variables plays only a role at the level
of interpretation, when we want to check that we have solved the original
problem.

In \cite{BGY}, we have a local noise. So there are two types
of noises, the initial condition, and the Wiener process which models
the local noise. This leads to a difficulty. Even though the initial
randomness is independent of the Wiener process, it is not true at
any positive time. Indeed, the state of the system depends on both
the initial condition, and of the local noise. We have still used
the Hilbert space of random variables (and not a general Hilbert
space), keeping track of the two noises. The approach works, but
we find difficulties at the level of interpretation. Indeed, there
is not a full equivalence of concepts of derivatives (at second order) in the Hilbert space of square-integrable random variables and in
the space of probability measures.

In \cite{BGY-2}, we have introduced
a different approach, which overcomes the difficulty. We extend the
initial condition by taking a random variable depending on a parameter.
To this parameter we associate the initial probability measure, but
we do not consider the parameter as a random variable. The payoff
to optimize extends the original payoff, by incorporating the parameter
dependent random variable and reduces to it when the random variable
is simply the identity, which associates to the parameter, the parameter
itself. The big advantage of the extended control problem is that
it is a control problem in a Hilbert space. The Hilbert space is not
the Hilbert space of square integrable random variables, but control
theory in Hilbert spaces can be fully applied. Since it is an extension
of the original control problem, we recover it as a particular case.
The master equation is the gradient of the classical master equation,
but we can recover the classical one, without much difficulties. The
control problem we solve is new. It is not just a reformulation of
the initial problem.

In this paper, we extend our previous work, \cite{BGY-2}
to cover the situation of the book \cite{CDLL}, so we can see that
our approach, by itself does not carry any specific limitation. We
have also included the common noise, so the alternative is now complete. 
Whereas in \cite{CDLL} the master equation is studied via a system of partial differential equations, our approach uses entirely the control theory for stochastic McKean-Vlasov type equations.
The monotonicity condition we require--see Equation \eqref{eq:3-6}--is no longer of Lasry-Lions type as in \cite{CDLL}, but it is more closely related to displacement monotonicity, cf.~\cite{GMMZ,GM}.
In contrast to \cite{GMMZ,GM}, in this work we build solutions to the master equation by developing a complete theory of classical solutions to a Hamilton-Jacobi equation on the Hilbert space introduced below.
The results are comparable, but the technique is quite distinct.
Although the results presented here apply only to mean field type control problems (or to potential mean field games), one can use similar techniques to obtain solutions to master equations for mean field games that are not potential \cite{BTWY}.

\section{\label{sec:FORMALISM}FORMALISM }

\subsection{WASSERSTEIN SPACE}

We consider the space of probability measures on $\mathbb{R}^{n}$, with second
moment, namely $\int_{\mathbb{R}^{n}}|x|^{2}dm(x)<+\infty$, if $m$ is a probability
measure on $\mathbb{R}^{n}.$ We call $\mathcal{P}_{2}(\mathbb{R}^{n})$ this space.
A metric can be defined on $\mathcal{P}_{2}(\mathbb{R}^{n}).$ A convenient
way to define the metric is to associate to $m$ a random variable
$X_{m}$ in $L^{2}(\Omega,\mathcal{A},\mathbb{P};\mathbb{R}^{n})$, where $(\Omega,\mathcal{A},\mathbb{P})$
is an atomless probability space, whose probability law $\mathcal{L}X_{m}=m$.
Then the metric is defined by 
\begin{equation}
W_{2}^{2}(m,m') = \inf \left\{\mathbb{E} \left(|X_{m}-X_{m'}|^{2}\right) : \mathcal{L}X_m = m, \ \mathcal{L}X_{m'} = m' \right\}.\label{eq:1-1}
\end{equation}
The infimum is attained, so we can find $\hat{X}_{m},\hat{X}_{m'}$
such that 
\begin{equation}
W_{2}^{2}(m,m')=\mathbb{E}\left(|\hat{X}_{m}-\hat{X}_{m'}|^{2}\right).\label{eq:1-2}
\end{equation}
A family $m_{k}$ converges to $m$ in $\mathcal{P}_{2}(\mathbb{R}^{n})$ if
and only if it converges in the sense of the weak convergence and 
\begin{equation}
\int_{\mathbb{R}^{n}}|x|^{2}dm_{k}(x)\rightarrow\int_{\mathbb{R}^{n}}|x|^{2}dm(x).\label{eq:1010}
\end{equation}
We refer to Carmona-Delarue \cite{RCD} for details. 

\subsection{FUNCTIONALS}

Consider a functional $F(m)$ on $\mathcal{P}_{2}(\mathbb{R}^{n}).$ Continuity
is clearly defined by the metric. For the concept of derivative in
$\mathcal{P}_{2}(\mathbb{R}^{n}),$ we use the concept of functional derivative.
The functional derivative of $F(m)$ at $m$ is a function on $\mathcal{P}_{2}(\mathbb{R}^{n})\times \mathbb{R}^{n},$
$m,x\mapsto \dfrac{dF}{d\nu}(m)(x)$ such that $m\times x\mapsto \dfrac{dF}{d\nu}(m)(x)$
continuous, satisfying 
\begin{equation}
\int_{\mathbb{R}^{n}}\left|\dfrac{dF}{d\nu}(m)(x)\right|^{2}dm(x)\leq c(m),\label{eq:2-104}
\end{equation}
 and 
\begin{equation}
\dfrac{F(m+\epsilon(m'-m))-F(m)}{\epsilon} \rightarrow \int\dfrac{dF}{d\nu}(m)(x)(dm'(x)-dm(x)),\ \text{as} \ \epsilon\rightarrow 0,  \label{eq:2-200}
\end{equation}
for any $m'$$\in\mathcal{P}_{2}(\mathbb{R}^{n}).$ Note that the definition
(\ref{eq:2-200}) implies that
\begin{equation}
\dfrac{d}{d\theta}F(m+\theta(m'-m))=\int_{\mathbb{R}^{n}}\dfrac{dF}{d\nu}(m+\theta(m'-m))(x)(dm'(x)-dm(x)),\label{eq:2-210}
\end{equation}
 and 
\begin{equation}
F(m')-F(m)=\int_{0}^{1}\int_{\mathbb{R}^{n}}\dfrac{dF}{d\nu}(m+\theta(m'-m))(x)(dm'(x)-dm(x))d\theta.\label{eq:2-211}
\end{equation}
Of course $\dfrac{dF}{d\nu}(m)(x)$ is just a notation. We have not
written $\dfrac{dF}{dm}(m)(x)$ to make the difference between the
notation $\nu$ and the argument $m.$ Also we prefer the notation
$\dfrac{dF}{d\nu}(m)(x)$ to $\dfrac{\delta F}{\delta m}(m)(x)$ used
in R.~Carmona and F.~Delarue \cite{RCD}, because there is no risk of confusion
and it works pretty much like an ordinary G\^ateaux derivative. We turn
to the second derivative. If $\dfrac{d}{d\theta}F(m+\theta(m'-m))$
is continuously differentiable in $\theta$, with the formula 
\begin{equation}
\dfrac{d^{2}}{d\theta^{2}}F(m+\theta(m'-m))=\int_{\mathbb{R}^{n}}\int_{\mathbb{R}^{n}}\dfrac{d^{2}F}{d\nu^{2}}(m+\theta(m'-m))(x,x^{1})(dm'(x)-dm(x))(dm'(x^{1})-dm(x^{1})),\label{eq:2-213}
\end{equation}
where $m,x,x^{1}$ $\mapsto\dfrac{d^{2}F}{d\nu^{2}}(m)(x,x^{1})$
is continuous and satisfies 
\begin{equation}
\int_{\mathbb{R}^{n}}\int_{\mathbb{R}^{n}}\left|\dfrac{d^{2}F}{d\nu^{2}}(m)(x,x^{1})\right|^{2}dm(x)dm(x^{1})\leq c(m),\label{eq:2-212}
\end{equation}
then the function $\dfrac{d^{2}F}{d\nu^{2}}(m)(x,x^{1})$ is called
the second-order functional derivative. Moreover, we have the formula 
\begin{align}
  F(m')-F(m)&=\int_{\mathbb{R}^{n}}\dfrac{dF}{d\nu}(m)(x)(dm'(x)-dm(x))\nonumber\\
  &\qquad +\int_{0}^{1}\int_{0}^{1}\theta\dfrac{d^{2}F}{d\nu^{2}}(m+\lambda\theta(m'-m))(x,x^{1})(dm'(x)-dm(x))(dm'(x^{1})-dm(x^{1}))d\lambda d\theta.
 \end{align}
Formulas (\ref{eq:2-200}), (\ref{eq:2-210}) do not change if we
had a constant to $\dfrac{dF}{d\nu}(m)(x)$. It is customary to assume
the normalisation $\displaystyle \int_{\mathbb{R}^{n}}\dfrac{dF}{d\nu}(m)(x)dm(x)=0.$ Similarly,
in formulas (\ref{eq:2-213}), (\ref{eq:2-212}), we can replace $\dfrac{d^{2}F}{d\nu^{2}}(m)(x,x^{1})$
with $\dfrac{1}{2}\left(\dfrac{d^{2}F}{d\nu^{2}}(m)(x,x^{1})+\dfrac{d^{2}F}{d\nu^{2}}(m)(x^{1},x)\right)$
with no change of value. So we may assume that the function $(x,x^{1})\mapsto\dfrac{d^{2}F}{d\nu^{2}}(m)(x,x^{1})$
is symmetric. Also we can add to a symmetric $\dfrac{d^{2}F}{d\nu^{2}}(m)(x,x^{1})$
a function of the form $\varphi(x)+\psi(x^{1})$. We can neglect such
a contribution. We also have the property 
\begin{equation}
\dfrac{\dfrac{d}{d\nu}F(m+\epsilon(m'-m))(x)-\dfrac{dF}{d\nu}(m)(x)}{\epsilon}\rightarrow\int_{\mathbb{R}^{n}}\dfrac{d^{2}F}{d\nu^{2}}(m)(x,x^{1})(dm'(x^{1})-dm(x^{1})),\label{eq:2-2100}
\end{equation}
for any $x$. If the limit function $\dfrac{d^{2}F}{d\nu^{2}}(m)(x,x^{1})$
is not symmetric, we keep it but the real one is the symmetric expression
$\dfrac{1}{2}\left(\dfrac{d^{2}F}{d\nu^{2}}(m)(x,x^{1})+\dfrac{d^{2}F}{d\nu^{2}}(m)(x^{1},x)\right).$
This is for convenience and will be used below. In the sequel, we
shall make use of the formulas 
\begin{numcases}{}
    \dfrac{dF}{d\nu}(m)(x)=\dfrac{dF}{d\nu}(m)(0)+\int_{0}^{1}D\dfrac{dF}{d\nu}(m)(\theta x)\cdot x\,d\theta,\nonumber\\
    \dfrac{d^{2}F}{d\nu^{2}}(m)(x,x^{1})=\dfrac{d^{2}F}{d\nu^{2}}(m)(x,0)+\dfrac{d^{2}F}{d\nu^{2}}(m)(0,x^{1})-\dfrac{d^{2}F}{d\nu^{2}}(m)(0,0)\label{eq:2-2101}\\
    \qquad\qquad\qquad\qquad+\int_{0}^{1}\int_{0}^{1}DD_{1}\dfrac{d^{2}F}{d\nu^{2}}(m)(\lambda x,\mu x^{1})x^{1}\cdot x\,d\lambda d\mu,\nonumber
\end{numcases}
which allows to recover the first- and second-order functional derivatives,
knowing the first and second gradient. If the matrix $DD_{1}\dfrac{d^{2}F}{d\nu^{2}}(m)(x,x^{1})$
is symmetric in $x,x^{1}$, then this formula gives a symmetric function in $x,x^{1}$. 

\subsection{\label{subsec:HILBERT-SPACE}HILBERT SPACE }

We consider an atomless probability space, $(\Omega,\mathcal{A},\mathbb{P})$. Later on, we shall need this space to be sufficiently large. Our
working Hilbert space will be $\mathcal{H}_{m}=L^{2}(\Omega,\mathcal{A},\mathbb{P};L_{m}^{2}(\mathbb{R}^{n};\mathbb{R}^{n})).$
An element of $\mathcal{H}_{m}$ will be denoted by $Z_{x}$, in
which for each $x$ we have an element of $L^{2}(\Omega,\mathcal{A},\mathbb{P};\mathbb{R}^{n})$.
However, we shall also omit $x$, as we omit $\omega$ for random variables.
We denote the norm 
\begin{equation}
||Z||^{2}=||Z||_{\mathcal{H}_{m}}^{2} :=\mathbb{E} \left(\int_{\mathbb{R}^{n}}|Z_{x}|^{2}dm(x)\right),\label{eq:2-215}
\end{equation}
and the corresponding inner product $\langle Z, Y \rangle$.

We then consider the pushforward probability on $\mathbb{R}^{n},$ $Z\#\mathbb{P}\otimes m\in\mathcal{P}_{2}(\mathbb{R}^{n})$.
Since there is no interest in keeping $\mathbb{P}$ in the notation, and since
we shall use $Z$ and $m$ as arguments, we prefer the notation $Z_{\cdot}\otimes m$. So we write 
\begin{equation}
\int_{\mathbb{R}^{n}}\varphi(\xi)dZ_{\cdot}\otimes m(\xi)=\mathbb{E}\left(\int_{\mathbb{R}^{n}}\varphi(Z_{x})dm(x)\right).\label{eq:2-216}
\end{equation}
If $Z_{x}=z_{x}$ a deterministic function, then $z_{\cdot}\in L_{m}^{2}(\mathbb{R}^{n};\mathbb{R}^{n}).$
Since $m$ belongs to $\mathcal{P}_{2}(\mathbb{R}^{n})$, the identity $J_{x}=x$
belongs to $\mathcal{H}_{m}$ and $J_{\cdot}\otimes m=m.$ Let $m\mapsto F(m)$
be a functional on $\mathcal{P}_{2}(\mathbb{R}^{n})$, the map $Z_{\cdot}\mapsto F(Z_{\cdot}\otimes m$)
becomes a functional on $\mathcal{H}_{m}.$ It is continuous as soon
as $F$ is continuous on $\mathcal{P}_{2}(\mathbb{R}^{n}).$ It is G\textroundcap{a}teaux
differentiable if 
\begin{equation}
\dfrac{F((Z_{\cdot}+\epsilon Y_{\cdot})\otimes m)-F(Z_{\cdot}\otimes m)}{\epsilon}\rightarrow\langle D_{X}F(Z_{\cdot}\otimes m),Y\rangle\ \text{as}\ \epsilon\rightarrow0,\forall Y\in\mathcal{H}_{m},\label{eq:2-217}
\end{equation}
and $D_{X}F(Z_{\cdot}\otimes m)$ is $\sigma(Z_{\cdot})$-measurable. If $F(m)$
has a functional derivative $\dfrac{d}{d\nu}F(m)(x)$ which is continuous
in both arguments and satisfies a growth condition:
\begin{equation}
\left|D\dfrac{d}{d\nu}F(m)(x)\right|\leq c(m)(1+|x|),\label{eq:2-218}
\end{equation}
with $c(m)$ is bounded on bounded sets then one can check the formula: 
\begin{equation}
D_{X}F(Z_{\cdot}\otimes m)=D\dfrac{d}{d\nu}F(Z_{\cdot}\otimes m)(Z_{x}).\label{eq:2-2000}
\end{equation}
Note that 
\begin{equation}
\langle D_{X}F(Z_{\cdot}\otimes m),Y_{\cdot} \rangle =0,\text{ if }Y_{x}\ \text{is independent of }Z_{\cdot}\ \text{and }\mathbb{E}(Y_{x})=0.\label{eq:2-3000}
\end{equation}
In particular taking $Z_{\cdot}=J_{\cdot}$, $D_{X}F(m)$ is simply a deterministic
function of $x$, $(D_{X}F(m))_{x}$ with values in $\mathbb{R}^{n}.$ If we
can show that this function is the gradient of a continuous function
of $m$, then we necessarily have:
\begin{equation}
(D_{X}F(m))_{x}=D\dfrac{d}{d\nu}F(m)(x),\label{eq:2-3500}
\end{equation}
and we can recover $\dfrac{d}{d\nu}F(m)(x)$ by the first relation
(\ref{eq:2-2101}). 

The second-order G\textroundcap{a}teaux differential of $F(Z_{\cdot}\otimes m)$
in $\mathcal{H}_{m}$ is a linear operator from $\mathcal{H}_{m}$
to $\mathcal{H}_{m}$ denoted by $D_{X}^{2}F(Z_{\cdot}\otimes m)$, such that
$D_{X}^{2}F(Z_{\cdot}\otimes m)(Y_{\cdot})$ is $\sigma(Z_{\cdot},Y_{\cdot})$-measurable for all
$Y_{\cdot}\in\mathcal{H}_{m}$ and 
\begin{equation}
\dfrac{\langle D_{X}F((Z_{\cdot}+\epsilon Y_{\cdot})\otimes m)-D_{X}F(Z_{\cdot}\otimes m),Y_{\cdot} \rangle}{\epsilon}\rightarrow\langle D_{X}^{2}F(Z_{\cdot}\otimes m)(Y_{\cdot}),Y_{\cdot} \rangle,\text{ as }\epsilon\rightarrow0.\label{eq:2-2005}
\end{equation}
The linear operator $D_{X}^{2}F(Z_{\cdot}\otimes m)$ is necessarily self-adjoint. As a consequence of (\ref{eq:2-3000}), we have:
\begin{equation}
\Big\langle D_{X}^{2}F(Z_{\cdot}\otimes m)(Y_{\cdot}),Y_{\cdot} \Big\rangle = \left\langle D^{2}\dfrac{dF}{d\nu}(Z_{\cdot}\otimes m)(Z_{\cdot})Y_{\cdot},Y_{\cdot} \right\rangle,\label{eq:2-3001}
\end{equation}
if $Y_{\cdot}(\omega)$ is independent of $Z_{\cdot}(\omega)$ and $\mathbb{E}(Y_{\cdot})=0$. Indeed,
\begin{align}
\dfrac{\langle D_{X}F((Z_{\cdot}+\epsilon Y_{\cdot})\otimes m)-D_{X}F(Z_{\cdot}\otimes m),Y_{\cdot} \rangle}{\epsilon}=&\dfrac{\left\langle D\dfrac{dF}{d\nu}((Z_{\cdot}+\epsilon Y_{\cdot})\otimes m)(Z_{\cdot}+\epsilon Y_{\cdot})-D\dfrac{dF}{d\nu}((Z_{\cdot}+\epsilon Y_{\cdot})\otimes m)(Z_{\cdot}),Y_{\cdot} \right\rangle}{\epsilon}\nonumber\\
&+ \left\langle \dfrac{F((Z_{\cdot}+\epsilon Y_{\cdot})\otimes m)-F(Z_{\cdot}\otimes m)}{\epsilon}(Z_{\cdot}),Y_{\cdot} \right\rangle.\nonumber
\end{align}

The second term tends to $0,$ according to (\ref{eq:2-3000}). The
first term is equal to:
\[
\left\langle \int_{0}^{1}D^{2}\dfrac{dF}{d\nu}((Z_{\cdot}+\epsilon Y_{\cdot})\otimes m)(Z_{\cdot}+\theta\epsilon Y_{\cdot})Y_{\cdot},Y_{\cdot} d\theta \right\rangle \rightarrow \left\langle D^{2}\dfrac{dF}{d\nu}(Z_{\cdot}\otimes m)(Z_{\cdot})Y_{\cdot},Y_{\cdot} \right\rangle.
\]
In general, we have the formula:
\begin{equation}
D_{X}^{2}F(Z_{\cdot}\otimes m)(Y_{\cdot})=D^{2}\dfrac{dF}{d\nu}(Z_{\cdot}\otimes m)(Z_{x})Y_{x}+\mathbb{E}^{1} \left(\int_{\mathbb{R}^{n}}DD_{1}\dfrac{d^{2}F}{d\nu^{2}}(Z_{\cdot}\otimes m)(Z_{x},Z_{x^{1}}^{1})Y_{x^{1}}^{1}dm(x^{1})\right),\label{eq:2-3002}
\end{equation}
where $Z_{x^{1}}^{1},Y_{x^{1}}^{1}$ are independent copies of $Z_{x},Y_{x}$,
and the expectation $\mathbb{E}^{1}$ affects only this independent copy. We
can recover the second-order functional derivative $\dfrac{d^{2}F}{d\nu^{2}}(m)(x,x^{1})$
from the second-order G\textroundcap{a}teaux differential $D_{X}^{2}F(Z_{\cdot}\otimes m)(Y_{\cdot})$
by taking $Z_{\cdot}=J_{\cdot}$ and $Y_{\cdot}$ as a deterministic function $y_{x}$
in $L_{m}^{2}(\mathbb{R}^{n},\mathbb{R}^{n}).$ Then (\ref{eq:2-3002}) becomes 
\begin{equation}
(D_{X}^{2}F(m)(y_{\cdot}))_{x}=D^{2}\dfrac{dF}{d\nu}(m)(x)y_{x}+\int_{\mathbb{R}^{n}}DD_{1}\dfrac{d^{2}F}{d\nu^{2}}(m)(x,x^{1})y_{x^{1}}dm(x^{1}),\label{eq:2-3003}
\end{equation}
which defines $DD_{1}\dfrac{d^{2}F}{d\nu^{2}}(m)(x,x^{1})$ as the
kernel of a linear operator in $L_{m}^{2}(\mathbb{R}^{n};\mathbb{R}^{n}).$ The regularity
of the kernel depends on the regularity of the function $(D_{X}^{2}F(m)(y_{\cdot}))_{x}.$
Once we have the function $DD_{1}\dfrac{d^{2}F}{d\nu^{2}}(m)(x,x^{1}),$
formula (\ref{eq:2-2101}) allows us to obtain the second-order functional
derivative. 

The second-order G\textroundcap{a}teaux differential allows to write the Taylor expansion: 
\begin{equation}
F((Z_{\cdot}+Y_{\cdot})\otimes m)=F(Z_{\cdot}\otimes m)+ \langle D_{X}F(Z_{\cdot}\otimes m),Y_{\cdot} \rangle +\int_{0}^{1}\int_{0}^{1}\lambda \langle D_{X}^{2}F((Z_{\cdot}+\lambda\mu Y_{\cdot})\otimes m)(Y_{\cdot}),Y_{\cdot} \rangle d\lambda d\mu.\label{eq:2-3004}
\end{equation}

\subsection{CONDITIONAL PROBABILITY MEASURE}

As an element of $\mathcal{P}_{2}(\mathbb{R}^{n})$, $Z_{\cdot}\otimes m$ is
deterministic. We are now going to introduce random probability measures
by considering a sub $\sigma$-algebra $\mathcal{B}$ of $\mathcal{A}$
and introducing $\mathbb{P}^{\mathcal{B}}$, the conditional probability on
$\Omega$ given $\mathcal{B}$. We define the random probability
measure $(Z_{\cdot}\otimes m)^{\mathcal{B}}$ by the formula 
\begin{equation}
\int_{\mathbb{R}^{n}}\varphi(\xi)d(Z_{\cdot}\otimes m)^{\mathcal{B}}(\xi)=\mathbb{E}^{\mathcal{B}} \left(\int_{\mathbb{R}^{n}}\varphi(Z_{x})dm(x)\right).\label{eq:2-2160}
\end{equation}
We can see that $(Z_{\cdot}\otimes m)^{\mathcal{B}}$ is the pushforward
of $\mathbb{P}^{\mathcal{B}}\otimes m$ by the map $(\omega,x)\mapsto Z_{x}(\omega).$
$(Z_{\cdot}\otimes m)^{\mathcal{B}}$ is a random variable with values
in $\mathcal{P}_{2}(\mathbb{R}^{n})$ which is $\mathcal{B}$-measurable. We introduce
the possibility of composition. Consider two random fields $Y_{x}(\omega)$
and $Z_{z}(\omega)$ and two sub $\sigma$-algebras $\mathcal{C}$
and $\mathcal{B}$. We have first the conditional probability measure
$(Y_{\cdot}\otimes m)^{\mathcal{C}}$. We can next consider $(Z_{\cdot}\otimes(Y_{\cdot}\otimes m)^{\mathcal{C}})^{\mathcal{B}}.$
We then have 
\begin{equation} \label{eq:2-2161}
    \int_{\mathbb{R}^{n}}\varphi(\xi)d\left(Z_{\cdot}\otimes(Y_{\cdot}\otimes m)^{\mathcal{C}}\right)^{\mathcal{B}}(\xi) =\mathbb{E}^{\mathcal{B}} \left(\int_{\mathbb{R}^{n}}\varphi(Z_{z})d(Y_{\cdot}\otimes m)^{\mathcal{C}}(z)\right)=\mathbb{E}^{\mathcal{B}}\left(\widetilde{\mathbb{E}}^{\mathcal{C}} \left(\int_{\mathbb{R}^{n}}\varphi(Z_{\tilde{Y}_{\tilde{x}}})dm(\tilde{x})\right)\right),
\end{equation}
where $\widetilde{\mathbb{E}}^{\mathcal{C}}$ operates only on $\tilde{Y}_{\tilde{x}}$. 
\begin{remark}
\label{rem2-2} We have 
\begin{equation}
\mathbb{E} \left((Z_{\cdot}\otimes m)^{\mathcal{B}}\right) =Z_{\cdot}\otimes m.\label{eq:2-2260}
\end{equation}
\end{remark}

\subsection{FUNCTIONALS OF CONDITIONAL PROBABILITY MEASURES}

Let $m\mapsto F(m)$ be a functional on $\mathcal{P}_{2}(\mathbb{R}^{n}).$
We assume the following properties: 
\begin{equation}
m\mapsto  F(m),\ \text{continuous},\ |F(m)|\leq C\left(1+\int_{\mathbb{R}^{n}}|x|^{2}dm(x)\right).\label{eq:2-2162}
\end{equation}
The functional $F(m)$ has a functional derivative $\dfrac{dF}{d\nu}(m)(x)$
satisfying:
\begin{equation}
\left\{\begin{aligned}
    &(x,m)\mapsto\dfrac{dF}{d\nu}(m)(x),D\dfrac{dF}{d\nu}(m)(x)\ \text{continuous,}\\\label{eq:2-2163}
    &\left|\dfrac{dF}{d\nu}(m)(x)\right|\leq C(1+|x|^{2}),\  \left\lVert D\dfrac{dF}{d\nu}(m)(x) \right\rVert \leq C(1+|x|^{2})^{\frac{1}{2}},\\
    &\left\lVert D\dfrac{dF}{d\nu}(m)(x_{1})-D\dfrac{dF}{d\nu}(m)(x_{2})\right\rVert\leq C|x_{1}-x_{2}|.
\end{aligned}\right.
\end{equation}
Let $\mathcal{B}$ be a sub $\sigma$-algebra of $\mathcal{A}$. We
define a functional on $\mathcal{H}_{m}$ as follows:
\begin{equation}
Z_{\cdot}\mapsto \mathbb{E}\left(F\left((Z_{\cdot}\otimes m)^{\mathcal{B}}\right)\right).\label{eq:2-2164}
\end{equation}
It generalizes the functional $Z_{\cdot}\mapsto F(Z_{\cdot}\otimes m)$
introduced in our previous work \cite{BGY-2}. We can state the following proposition.
\begin{proposition}
\label{prop2-1}With the assumptions (\ref{eq:2-2162}), (\ref{eq:2-2163})
the functional $Z_{\cdot}\mapsto \mathbb{E}\left(F((Z_{\cdot}\otimes m)^{\mathcal{B}})\right)$
is continuously G\textroundcap{a}teaux differentiable on $\mathcal{H}_{m},$
with the formula 
\begin{equation}
\begin{aligned}
&D_{X}\mathbb{E}\left(F((Z_{\cdot}\otimes m)^{\mathcal{B}})\right)=D\dfrac{dF}{d\nu}((Z_{\cdot}\otimes m)^{\mathcal{B}})(Z_{\cdot})\in\mathcal{H}_{m},\label{eq:2-2165}\\
&\left\lVert D_{X}\mathbb{E}\left(F((Z_{\cdot}\otimes m)^{\mathcal{B}})\right)\right\rVert\leq C(1+||Z||).\nonumber
\end{aligned}
\end{equation}
\end{proposition}

\begin{proof}
We have 
\begin{align}
    &\ \ \ \ \dfrac{1}{\epsilon} \left[F\left(((Z_{\cdot}+\epsilon Y_{\cdot})\otimes m)^{\mathcal{B}} \right) -F \left((Z_{\cdot}\otimes m)^{\mathcal{B}} \right) \right]\\\nonumber
    &=\dfrac{1}{\epsilon}\int_{0}^{1} \bigg[\dfrac{dF}{d\nu} \left((Z_{\cdot}\otimes m)^{\mathcal{B}}+\theta\left(((Z_{\cdot}+\epsilon Y_{\cdot})\otimes m)^{\mathcal{B}}-(Z_{\cdot}\otimes m)^{\mathcal{B}}\right) \right)(x)\\\nonumber
    &\qquad\qquad \left(d \left((Z_{\cdot}+\epsilon Y_{\cdot})\otimes m\right)^{\mathcal{B}}(x)-d(Z_{\cdot}\otimes m)^{\mathcal{B}}(x) \right) \bigg] d\theta\\\nonumber
    &=\dfrac{1}{\epsilon}\int_{0}^{1}\mathbb{E}^{\mathcal{B}}\bigg[\int_{\mathbb{R}^{n}}\dfrac{dF}{d\nu}\left((Z_{\cdot}\otimes m)^{\mathcal{B}}+\theta \Big(((Z_{\cdot}+\epsilon Y_{\cdot})\otimes m)^{\mathcal{B}}-(Z_{\cdot}\otimes m)^{\mathcal{B}}\Big)\right)(Z_{x}+\epsilon Y_{x})dm(x)\\\nonumber
    &\qquad\qquad\qquad-\int_{\mathbb{R}^{n}}\dfrac{dF}{d\nu} \left((Z_{\cdot}\otimes m)^{\mathcal{B}}+\theta\Big(((Z_{\cdot}+\epsilon Y_{\cdot})\otimes m)^{\mathcal{B}}-(Z_{\cdot}\otimes m)^{\mathcal{B}}\Big)\right)(Z_{x})dm(x)\bigg] d\theta\\\nonumber
    &\longrightarrow
    \mathbb{E}^{\mathcal{B}}\left(\int_{\mathbb{R}^{n}}D\dfrac{dF}{d\nu}((Z_{\cdot}\otimes m)^{\mathcal{B}})(Z_{x}) \cdot Y_{x}dm(x)\right), \text{ a.s.. }
\end{align}



We then conclude that 
\[
\dfrac{1}{\epsilon} \Bigg[\mathbb{E}\left(F\left(((Z_{\cdot}+\epsilon Y_{\cdot})\otimes m)^{\mathcal{B}}\right)\right)-\mathbb{E}\left(F\left((Z_{\cdot}\otimes m)^{\mathcal{B}}\right)\right)\Bigg]\rightarrow \mathbb{E}\left(\int_{\mathbb{R}^{n}}D\dfrac{dF}{d\nu}\left((Z_{\cdot}\otimes m)^{\mathcal{B}}\right)(Z_{x})\cdot Y_{x}dm(x)\right),
\]
which proves formula (\ref{eq:2-2165}). We have also 
\begin{equation} \hspace{-1cm}
    \left\lVert D_{X}\mathbb{E} \left(F\left((Z_{\cdot}\otimes m)^{\mathcal{B}}\right)\right)\right\rVert =\sqrt{\mathbb{E} \left(\int_{\mathbb{R}^{n}} \left|D\dfrac{dF}{d\nu}((Z_{\cdot}\otimes m)^{\mathcal{B}})(Z_{x})\right|^{2}dm(x)\right)} \leq C\sqrt{\mathbb{E}\left(\int_{\mathbb{R}^{n}}(1+|Z_{x}|^{2})dm(x)\right)}=C(1+||Z||).
\end{equation}

The function $Z_{\cdot}\mapsto D_{X}\mathbb{E} \left(F((Z_{\cdot}\otimes m)^{\mathcal{B}})\right)$
is continuous. This completes the proof. 
\end{proof}
\begin{remark}
\label{rem2-1} $D_{X}\mathbb{E}\left(F((Z_{\cdot}\otimes m)^{\mathcal{B}})\right)$ is a notation
for the G\textroundcap{a}teaux derivative of $\mathbb{E}\left(F((Z_{\cdot}\otimes m)^{\mathcal{B}})\right).$
There is no interchange of the symbols $D_{X}$ and expectation $\mathbb{E}$.
Formula (\ref{eq:2-2165}) shows that $D_{X}\mathbb{E}\left(F((Z_{\cdot}\otimes m)^{\mathcal{B}})\right)$
is a random variable $\mathcal{B\ \cup\,\sigma}(X_{\cdot})$ measurable.
\end{remark}

If $Y_{\cdot}$ is independent of $Z_{\cdot}$ and $\mathcal{B}$ and $\mathbb{E}(Y_{x})=0$,
we have 
\begin{equation}
\left\langle D_{X}\mathbb{E} \left(F((Z_{\cdot}\otimes m)^{\mathcal{B}})\right),Y_{\cdot} \right\rangle =0.\label{eq:2-3100}
\end{equation}
As mentioned earlier, the fact that $Z_{\cdot}\in\mathcal{H}_{m}$ is
fine, when $m$ does not change and is just a parameter (a reference
probability). However it complicates the study of the map $m\mapsto F(Z_{\cdot}\otimes m)$
since $Z_{\cdot}$ is coupled with $m$. To decouple, we may assume the
following stronger assumption on $Z_{\cdot}$, namely $\dfrac{Z_{x}}{(1+|x|^{2})^{\frac{1}{2}}}\in L^{\infty}(\mathbb{R}^{n};L^{2}(\Omega,\mathcal{A},\mathbb{P};\mathbb{R}^{n}))$,
which means 
\begin{equation}
\mathbb{E}(|Z_{x}|^{2})\leq c(1+|x|^{2}),\label{eq:2-219}
\end{equation}
where $c$ is a constant. Then $Z_{\cdot}\in\mathcal{H}_{m}$ for any
$m,$ and $Z$ and $m$ can be considered as two separate arguments
in $\mathbb{E}\left(F((Z_{\cdot}\otimes m)^{\mathcal{B}})\right)$. So we can consider the
functional $m\mapsto \mathbb{E}\left(F((Z_{\cdot}\otimes m)^{\mathcal{B}})\right)$ and
look at its functional derivative. We denote it $\dfrac{\partial}{\partial m}\mathbb{E}\left(F((Z_{\cdot}\otimes m)^{\mathcal{B}})\right)$. It is indeed a partial derivative. In fact, we see easily that 
\begin{equation}
\dfrac{\partial}{\partial m}\mathbb{E}\left(F\left((Z_{\cdot}\otimes m)^{\mathcal{B}}\right)\right)=\mathbb{E}\left(\dfrac{dF}{d\nu}\left((Z_{\cdot}\otimes m)^{\mathcal{B}}\right)(Z_{x})\right).\label{eq:2-220}
\end{equation}

\subsection{SECOND-ORDER G\^{A}TEAUX DERIVATIVE IN THE HILBERT SPACE}

We next make further assumptions: 
\begin{align}
    &(m,x)\mapsto  D^{2}\dfrac{dF}{d\nu}(m)(x)\ \text{continuous},\label{eq:2-221} \  \Bigg|D^{2}\dfrac{dF}{d\nu}(m)(x)\Bigg|\leq C,\\
    &(m,x,x^{1})\mapsto DD_{1}\dfrac{d^{2}}{d\nu^{2}}F(m)(x,x^{1})\ \text{continuous},\label{eq:2-222} \   \Bigg|DD_{1}\dfrac{d^{2}}{d\nu^{2}}F(m)(x,x^{1})\Bigg|\leq C,
\end{align}
where $D$ is the gradient with respect to the first argument $x$
and $D_{1}$ is the gradient with respect to the second argument $x^{1}.$

We say that the functional $Z_{\cdot}\mapsto \mathbb{E}\left(F((Z_{\cdot}\otimes m)^{\mathcal{B}})\right)$
has a second-order G\^ateaux derivative in $\mathcal{H}_{m}$, denoted
$D_{X}^{2}\mathbb{E}\left(F((Z_{\cdot}\otimes m)^{\mathcal{B}})\right)\in\mathcal{L}(\mathcal{H}_{m};\mathcal{H}_{m})$,
if 
\begin{equation}
\dfrac{\left\langle D_{X}\mathbb{E}\left(F(((Z_{\cdot}+\epsilon Y_{\cdot})\otimes m)^{\mathcal{B}})\right)-D_{X}\mathbb{E}\left(F((Z_{\cdot}\otimes m)^{\mathcal{B}})\right),Y_{\cdot} \right\rangle }{\epsilon}\rightarrow \bigg\langle D_{X}^{2}\mathbb{E}\left(F\left((Z_{\cdot}\otimes m)^{\mathcal{B}}\right)\right)(Y_{\cdot}),Y_{\cdot} \bigg\rangle ,\label{eq:2-227}
\end{equation}
as $\epsilon\rightarrow 0$, for all $Y_\cdot \in \mathcal{H}_m$. We then state the following result.
\begin{proposition}
\label{prop2-2} We make the assumptions of Proposition \ref{prop2-1}
and (\ref{eq:2-221}), (\ref{eq:2-222}). Then the functional $Z_{\cdot}\mapsto \mathbb{E}\left(F((Z_{\cdot}\otimes m)^{\mathcal{B}})\right)$
has a second-order G\^ateaux derivative in $\mathcal{H}_{m}$, given
by the formula:
\begin{align}
D_{X}^{2}\mathbb{E}\left(F \left((Z_{\cdot}\otimes m)^{\mathcal{B}}\right)\right)(Y_{\cdot})&=D^{2}\dfrac{dF}{d\nu}\left((Z_{\cdot}\otimes m)^{\mathcal{B}}\right)(Z_{\cdot})Y_{\cdot}\label{eq:2-100}\\
&\quad +\mathbb{E}^{1\mathcal{B}}\left(\int_{\mathbb{R}^{n}}DD_{1}\dfrac{d^{2}}{d\nu^{2}}F\left((Z_{\cdot}\otimes m)^{\mathcal{B}}\right)(Z_{\cdot},Z_{x^{1}}^{1})Y_{x^{1}}^{1}dm(x^{1})\right),\nonumber
\end{align}
where for $x,x^{1}$ given $Z_{x},Y_{x}$ and $Z_{x^{1}}^{1},Y_{x^{1}}^{1}$
are conditionally to $\mathcal{B}$ independent copies. 
\end{proposition}

\begin{proof}
We consider 
\begin{align}
    & \ \ \ \ \dfrac{1}{\epsilon}\mathbb{E}\left(\int_{\mathbb{R}^{n}} \left[D\dfrac{dF}{d\nu}\left(((Z_{\cdot}+\epsilon Y_{\cdot})\otimes m)^{\mathcal{B}}\right)(Z_{x}+\epsilon Y_{x})-D\dfrac{dF}{d\nu}((Z_{\cdot}\otimes m)^{\mathcal{B}})(Z_{x}) \right]\cdot Y_{x}dm(x)\right)\\\nonumber
&=\dfrac{1}{\epsilon}\mathbb{E}\left(\int_{\mathbb{R}^{n}} \left[D\dfrac{dF}{d\nu}\left(((Z_{\cdot}+\epsilon Y_{\cdot})\otimes m)^{\mathcal{B}}\right)(Z_{x}+\epsilon Y_{x})-D\dfrac{dF}{d\nu}\left(((Z_{\cdot}+\epsilon Y_{\cdot})\otimes m)^{\mathcal{B}}\right)(Z_{x}) \right]\cdot Y_{x}dm(x)\right)\\\nonumber
&\qquad \quad +\dfrac{1}{\epsilon}\mathbb{E} \left(\ \mathbb{E}^{\mathcal{B}}\left(\int_{\mathbb{R}^{n}} \left[D\dfrac{dF}{d\nu} \left(((Z_{\cdot}+\epsilon Y_{\cdot})\otimes m)^{\mathcal{B}}\right)(Z_{x})-D\dfrac{dF}{d\nu}((Z_{\cdot}\otimes m)^{\mathcal{B}})(Z_{x}) \right]\cdot Y_{x}dm(x)\right)\right).\nonumber
\end{align}
 Thanks to the assumptions, we can obtain the following limit as $\epsilon\rightarrow0$: 
\begin{align}
&\mathbb{E}\left(\int_{\mathbb{R}^{n}} D^{2}\dfrac{dF}{d\nu}((Z_{\cdot}\otimes m)^{\mathcal{B}})(Z_{x})Y_{x} \cdot Y_{x}dm(x)\right)\nonumber\\
&+\mathbb{E} \left(\ \mathbb{E}^{\mathcal{B}} \left(\int_{\mathbb{R}^{n}}\mathbb{E}^{1\mathcal{B}} \left(\int_{\mathbb{R}^{n}}DD_{1}\dfrac{d^{2}}{d\nu^{2}}F((Z_{\cdot}\otimes m)^{\mathcal{B}})(Z_{x},Z_{x^{1}}^{1})Y_{x^{1}}^{1} \cdot Y_{x}dm(x^{1})\right)dm(x)\right)\right),\nonumber
\end{align}
and from the definition (\ref{eq:2-100}), this expression is equal
to $\left\langle D_{X}^{2}\mathbb{E}\left(F((Z_{\cdot}\otimes m)^{\mathcal{B}})\right)(Y_{\cdot}),Y_{\cdot} \right\rangle.$
Using the symmetry 
\begin{equation}
DD_{1}\dfrac{d^{2}}{d\nu^{2}}F(m)(x,x^{1})=D_{1}D\dfrac{d^{2}}{d\nu^{2}}F(m)(x,x^{1})=D_{1}D\dfrac{d^{2}}{d\nu^{2}}F(m)(x^{1},x),\label{eq:2-101}
\end{equation}
the operator $D_{X}^{2}\mathbb{E}\left(F((Z_{\cdot}\otimes m)^{\mathcal{B}})\right)(Y_{\cdot})$
is self-adjoint. The matrix $D^{2}\dfrac{dF}{d\nu}((Z_{\cdot}\otimes m)^{\mathcal{B}})(Z_{\cdot})$
and the second term $\mathbb{E}^{1\mathcal{B}}\left(\int_{\mathbb{R}^{n}}DD_{1}\dfrac{d^{2}}{d\nu^{2}}F((Z_{\cdot}\otimes m)^{\mathcal{B}})(Z_{\cdot},Z_{x^{1}}^{1})Y_{x^{1}}^{1}dm(x^{1})\right)$
are measurable with respect to $\mathcal{B\ \cup\,\sigma}(Z_{\cdot})$.
\end{proof}
\begin{remark}
\label{rem3-10} If $Y_{\cdot}$ and $Z_{\cdot}$ are independent, conditionally
to $\mathcal{B}$ and $\mathbb{E}(Y_{x})=0,$ then we have 
\begin{equation}
D_{X}^{2}\mathbb{E}\left(F((Z_{\cdot}\otimes m)^{\mathcal{B}})\right)(Y_{\cdot})=D^{2}\dfrac{dF}{d\nu}((Z_{\cdot}\otimes m)^{\mathcal{B}})(Z_{\cdot})Y_{\cdot}.\label{eq:3-2000}
\end{equation}
\end{remark}

\begin{proposition}
\label{prop2-3} With the assumptions of Proposition \ref{prop2-2}
we have the estimate 
\begin{equation}
\left\lVert D_{X}^{2}\mathbb{E}\left(F((Z_{\cdot}\otimes m)^{\mathcal{B}})\right)(Y_{\cdot})\right\rVert\leq C||Y_{\cdot}||.\label{eq:2-102}
\end{equation}
The function $Z_{\cdot}\mapsto D_{X}^{2}\mathbb{E}\left(F((Z_{\cdot}\otimes m)^{\mathcal{B}})\right)(Y)$
is continuous from $\mathcal{H}_{m}$ to $\mathcal{H}_{m}$, for fixed
$Y_{\cdot}.$ We have also the continuity property:
\begin{equation}
\left\langle D_{X}^{2}\mathbb{E}\left(F((Z_{\cdot}^{k}\otimes m)^{\mathcal{B}})\right)(Y^{k}),Y^{k} \right\rangle - \left\langle D_{X}^{2}\mathbb{E}\left(F((Z_{\cdot}\otimes m)^{\mathcal{B}})\right)(Y_{\cdot}^{k}),Y^{k} \right\rangle \rightarrow0,\label{eq:2-103}
\end{equation}
if $Z_{\cdot}^{k}\rightarrow Z_{\cdot}$ in $\mathcal{H}_{m}$ and $||Y_{\cdot}^{k}\mathbbm{1}_{|Y^{k}|\geq M}||\leq c(M)$,
with $c(M)\rightarrow0,$ as $M\rightarrow+\infty$. 
\end{proposition}

\begin{proof}
We use 
\begin{align}
    &\ \left\lVert D_{X}^{2}\mathbb{E}\left(F((Z_{\cdot}^{k}\otimes m)^{\mathcal{B}})\right)(Y_{\cdot})-D_{X}^{2}\mathbb{E}\left(F((Z_{\cdot}\otimes m)^{\mathcal{B}})\right)(Y_{\cdot})\right\rVert \nonumber\\
    \leq &\ \left\lVert \left(D^{2}\dfrac{dF}{d\nu}((Z_{\cdot}^{k}\otimes m)^{\mathcal{B}})(Z_{\cdot}^{k})-D^{2}\dfrac{dF}{d\nu}((Z_{\cdot}\otimes m)^{\mathcal{B}})(Z_{\cdot}) \right)Y_{\cdot} \right\rVert\label{eq:2-1040}\\
    &\quad + \left\lVert\mathbb{E}^{1\mathcal{B}}\left(\int_{\mathbb{R}^{n}} \left(DD_{1}\dfrac{d^{2}}{d\nu^{2}}F((Z_{\cdot}^{k}\otimes m)^{\mathcal{B}})(Z_{\cdot}^{k},Z_{x^{1}}^{1k})-DD_{1}\dfrac{d^{2}}{d\nu^{2}}F((Z_{\cdot}\otimes m)^{\mathcal{B}})(Z_{\cdot},Z_{x^{1}}^{1}) \right)Y_{x^{1}}^{1}dm(x^{1})\right) \right\rVert.\nonumber
\end{align}
Then 
\begin{align}
    &\quad \left\lVert \left(D^{2}\dfrac{dF}{d\nu}((Z_{\cdot}^{k}\otimes m)^{\mathcal{B}})(Z_{\cdot}^{k})-D^{2}\dfrac{dF}{d\nu}((Z_{\cdot}\otimes m)^{\mathcal{B}})(Z_{\cdot}) \right)Y_{\cdot}\right\rVert^{2}\nonumber\\
    &=\mathbb{E}\left(\int_{\mathbb{R}^{n}} \left|\left(D^{2}\dfrac{dF}{d\nu}((Z_{\cdot}^{k}\otimes m)^{\mathcal{B}})(Z_{x}^{k})-D^{2}\dfrac{dF}{d\nu}((Z_{\cdot}\otimes m)^{\mathcal{B}})(Z_{x}) \right)Y_{x} \right|^{2}dm(x)\right)\ \rightarrow0,\nonumber
\end{align}
from the assumptions and standard Lebesgue integration theory. 
Similarly, 
\begin{equation*} \hspace{-1cm}
\begin{aligned} 
    &\quad \left\lVert\mathbb{E}^{1\mathcal{B}}\left(\int_{\mathbb{R}^{n}} \left(DD_{1}\dfrac{d^{2}}{d\nu^{2}}F((Z_{\cdot}^{k}\otimes m)^{\mathcal{B}})(Z_{\cdot}^{k},Z_{x^{1}}^{1k})-DD_{1}\dfrac{d^{2}}{d\nu^{2}}F((Z_{\cdot}\otimes m)^{\mathcal{B}})(Z_{\cdot},Z_{x^{1}}^{1}) \right)Y_{x^{1}}^{1}dm(x^{1})\right)\right\rVert^{2}\nonumber\\
    &=\mathbb{E}\left(\int_{\mathbb{R}^{n}} \left|\mathbb{E}^{1\mathcal{B}} \left(\int_{\mathbb{R}^{n}} \left(DD_{1}\dfrac{d^{2}}{d\nu^{2}}F((Z_{\cdot}^{k}\otimes m)^{\mathcal{B}})(Z_{\cdot}^{k},Z_{x^{1}}^{1k})-DD_{1}\dfrac{d^{2}}{d\nu^{2}}F((Z_{\cdot}\otimes m)^{\mathcal{B}})(Z_{\cdot},Z_{x^{1}}^{1}) \right)Y_{x^{1}}^{1}dm(x^{1})\right) \right|^{2}dm(x)\right)\nonumber\\
    &\leq \mathbb{E} \left(\ \mathbb{E}^{\mathcal{B}} \left(\int_{\mathbb{R}^{n}}\mathbb{E}^{1\mathcal{B}} \left(\int_{\mathbb{R}^{n}} \left|DD_{1}\dfrac{d^{2}}{d\nu^{2}}F((Z_{\cdot}^{k}\otimes m)^{\mathcal{B}})(Z_{\cdot}^{k},Z_{x^{1}}^{1k})-DD_{1}\dfrac{d^{2}}{d\nu^{2}}F((Z_{\cdot}\otimes m)^{\mathcal{B}})(Z_{\cdot},Z_{x^{1}}^{1}) \right|^{2}|Y_{x^{1}}^{1}|^{2}dm(x^{1})\right)dm(x)\right)\right)\nonumber\\&
    \rightarrow0,\nonumber
\end{aligned}
\end{equation*}
from the assumptions and Lebesgue integration theory. Also, (\ref{eq:2-103})
will follow from 
\begin{equation*} \hspace{-1cm}
\begin{aligned} 
    &\mathbb{E}\left(\int_{\mathbb{R}^{n}} \left|D^{2}\dfrac{dF}{d\nu}((Z_{\cdot}^{k}\otimes m)^{\mathcal{B}})(Z_{x}^{k})-D^{2}\dfrac{dF}{d\nu}((Z_{\cdot}\otimes m)^{\mathcal{B}})(Z_{x}) \right|^{2}dm(x)\right)\rightarrow0,\nonumber\\
&\mathbb{E}\left(\ \mathbb{E}^{\mathcal{B}}\left(\int_{\mathbb{R}^{n}}\mathbb{E}^{1\mathcal{B}} \left(\int_{\mathbb{R}^{n}} \left|DD_{1}\dfrac{d^{2}}{d\nu^{2}}F((Z_{\cdot}^{k}\otimes m)^{\mathcal{B}})(Z_{x}^{k},Z_{x^{1}}^{1k})-DD_{1}\dfrac{d^{2}}{d\nu^{2}}F((Z_{\cdot}\otimes m)^{\mathcal{B}})(Z_{x},Z_{x^{1}}^{1}) \right|^{2}dm(x^{1})\right)dm(x)\right)\right)\rightarrow0,\nonumber
\end{aligned}
\end{equation*}
as well as the fact that we can replace $Y^{k}$ with $Y_{\cdot}^{k}\mathbbm{1}_{|Y^{k}|<M}$,
which is uniformly bounded. Finally, (\ref{eq:2-102}) is an easy consequence
of the second parts of (\ref{eq:2-221}), (\ref{eq:2-222}). This
concludes the proof. 
\end{proof}
By collecting results, we can write the second-order Taylor formula:
\begin{align}
    \mathbb{E} \left(F(((Z_{\cdot}+Y_{\cdot})\otimes m)^{\mathcal{B}})\right)-\mathbb{E} \left(F((Z_{\cdot}\otimes m)^{\mathcal{B}})\right) &= \left\langle D_{X}\mathbb{E}\left(F((Z_{\cdot}\otimes m)^{\mathcal{B}})\right),Y_{\cdot} \right\rangle \label{eq:2-105}\\
    &\quad \, +\int_{0}^{1}\int_{0}^{1}\lambda \left\langle D_{X}^{2}\mathbb{E}\left( F(((Z_{\cdot}+\lambda\mu Y_{\cdot})\otimes m)^{\mathcal{B}})\right)(Y_{\cdot}),Y_{\cdot} \right\rangle d\lambda d\mu.\nonumber
\end{align}

\section{STOCHASTIC CALCULUS }

\subsection{PRELIMINARIES}

As said before the probability space $(\Omega,\mathcal{A},\mathbb{P})$ will
be sufficiently large to contain two independent standard Wiener processes
in $\mathbb{R}^{n}$, denoted by $w(t)$ and $b(t)$, respectively, and for
any $t$, additional random variables independent of the filtration
$\mathcal{F}_{t},$ which is the family of $\sigma$-algebras $\mathcal{F}_{t}^{s}=\sigma(w(\tau)-w(t),b(\tau)-b(t)\ ,t\leq\tau\leq s).$
We also set $\mathcal{B}_{t}^{s}=\sigma(b(\tau)-b(t),\ t\leq\tau\leq s).$ Let $X_{x}$ be in $\mathcal{H}_{m}$, independent of $\mathcal{F}_{t}$.
We define $\mathcal{F}_{X_{\cdot}t}^{s} :=\sigma(X_{\cdot})\cup\mathcal{F}_{t}^{s}$
and $\mathcal{F}_{X_{\cdot}t}$ the filtration generated by the $\sigma$-algebras
$\mathcal{F}_{X_{\cdot}t}^{s}$. We denote by $L_{\mathcal{F}_{X_{\cdot}t}}^{2}(t,T;\mathcal{H}_{m})$
the subspace of $L^{2}(t,T;\mathcal{H}_{m})$ adapted to the filtration
$\mathcal{F}_{X_{\cdot}t}$. An element of $L_{\mathcal{F}_{X_{\cdot}t}}^{2}(t,T;\mathcal{H}_{m})$
can be written as $X_{\xi t}(s)\Big|_{\xi=X_{x}}$, where for any $\xi$
in $\mathbb{R}^{n},$ the process $X_{\xi t}(\cdot)$ is adapted to $\mathcal{F}_{t}$,
and 
\[
\int_{t}^{T}\mathbb{E}\left(\int_{\mathbb{R}^{n}}\left(\mathbb{E}\left(|X_{\xi t}(s)|^{2}\right)\bigg|_{\xi=X_{x}}\right)dm(x)\right)ds=\int_{t}^{T}\mathbb{E}\left(\int_{\mathbb{R}^{n}}|X_{X_{x}t}(s)|^{2}dm(x)\right)ds,
\]
because of the independence of $X_{x}$ and $\mathcal{F}_{t}$. Similarly,
we have: 
\begin{equation}
(X_{X_{\cdot}t}(s)\otimes m)^{\mathcal{B}_{t}^{s}}=(X_{\cdot t}(s)\otimes(X_{\cdot}\otimes m))^{\mathcal{B}_{t}^{s}}.\label{eq:3-200}
\end{equation}
This follows from the relation 
\begin{equation}
\mathbb{E}^{\mathcal{B}_{t}^{s}} \left(\int_{\mathbb{R}^{n}}\varphi(X_{X_{x}t}(s))dm(x)\right)=\mathbb{E}^{\mathcal{B}_{t}^{s}}\left(\int_{\mathbb{R}^{n}}\varphi(X_{\eta t}(s))d(X_{\cdot}\otimes m)(\eta)\right),\label{eq:3-201}
\end{equation}
 for any test function $\varphi$ continuous and bounded on $\mathbb{R}^{n}.$ The
relation (\ref{eq:3-201}) is a consequence of the independence of
$X_{x}$ and $\mathcal{B}_{t}^{s}.$ 
\begin{remark}
\label{rem3-1}The random variable $X_{X_{x}t}(s)$ is independent
of $\mathcal{B}_{s}^{T},$ if $s<T.$ Therefore,
\begin{equation}
(X_{X_{\cdot}t}(s)\otimes m)^{\mathcal{B}_{t}^{s}}=(X_{X_{\cdot}t}(s)\otimes m)^{\mathcal{B}_{t}^{T}}.\label{eq:3-202}
\end{equation}
\end{remark}

\subsection{IT\^O PROCESS}

Let $u_{X_{\cdot}t}(\cdot)\in L_{\mathcal{F}_{X_{\cdot}t}}^{2}(t,T;\mathcal{H}_{m}).$
To simplify, we assume that 
\begin{equation}
u_{X_{\cdot}t}(\cdot)\in C_{\mathcal{F}_{X_{\cdot}t}}^{0}(t,T;\mathcal{H}_{m}),\ ||u_{X_{\cdot}t}(s)||\leq C(1+||X||).\label{eq:3-1000}
\end{equation}
We then focus on processes $X_{X_{\cdot}t}(s)$ of the form 
\begin{equation}
X_{X_{x}t}(s)=X_{x}+\int_{t}^{s}u_{X_{x}t}(\tau)d\tau+\sigma(w(s)-w(t))+\beta(b(s)-b(t)),\label{eq:3-101}
\end{equation}
where $\sigma$ is a matrix $n\times n$ and $\beta$ is simply a
constant. We call $X_{X_{x}t}(s)$ an It\^o process. For the differential
calculus we are going to use, we could consider much more general It\^o processes, but the form (\ref{eq:3-101}) will be sufficient for
the control problem we are interested in. We consider a functional
$F(m,s)$ on $\mathcal{P}_{2}(\mathbb{R}^{n})\times(0,T).$ As a function of
$m$, it will satisfy all the assumptions (\ref{eq:2-2162}), (\ref{eq:2-2163}),
(\ref{eq:2-221}), (\ref{eq:2-222}), but because of the argument
$s$, we have to restate them: 
\begin{numcases}{}
    (m,s)\mapsto  F(m,s),\ \text{continuous},\ |F(m,s)|\leq C\left(1+\int_{\mathbb{R}^{n}}|x|^{2}dm(x)\right);\label{eq:3-150}\\
    s\mapsto F(m,s)\ \text{is differentiable a.e., }\left|\tfrac{\partial F}{\partial s}(m,s)\right|\leq C\left(1+\int_{\mathbb{R}^{n}}|x|^{2}dm(x)\right);\label{eq:3-151}\\
    (x,m,s)\mapsto\dfrac{dF}{d\nu}(m,s)(x),D\dfrac{dF}{d\nu}(m,s)(x)\ \text{continuous;}\ \label{eq:3-152}\\
    \left|\dfrac{dF}{d\nu}(m,s)(x)\right|\leq C(1+|x|^{2}),\ \left|D\dfrac{dF}{d\nu}(m,s)(x)\right|\leq C(1+|x|^{2})^{\frac{1}{2}};\nonumber\\
    \left|D^{2}\dfrac{dF}{d\nu}(m,s)(x)\right|\leq C;\label{eq:3-153}\\
    \left|DD_{1}\dfrac{d^{2}}{d\nu^{2}}F(m,s)(x,x^{1})\right|\leq C.\label{eq:3-154}
\end{numcases}
With the assumptions (\ref{eq:3-150}) to (\ref{eq:3-154}), we can assert from Proposition \ref{prop2-1} and Proposition \ref{prop2-2}
that the function $(Z_{\cdot},s)\mapsto \mathbb{E}\left(F((Z_{\cdot}\otimes m)^{\mathcal{B}},s)\right)$
is twice G\^{a}teaux differentiable and a.e. differentiable in $s$ with
the properties 
\begin{numcases}{}
    \left\lVert\mathbb{E}\left(F((Z_{\cdot}\otimes m)^{\mathcal{B}},s)\right)\right\rVert\leq C(1+||Z||^{2});\label{eq:3-155}\\
    \left\lVert \dfrac{\partial}{\partial s}\mathbb{E}\left(F((Z_{\cdot}\otimes m)^{\mathcal{B}},s)\right)\right\rVert \leq C(1+||Z||^{2})\ \text{a.e.}\ s;\nonumber\\
    \left\lVert D_{X}\mathbb{E}\left(F((Z_{\cdot}\otimes m)^{\mathcal{B}},s)\right)\right\rVert \leq C(1+||Z||);\nonumber\\
    (Z_{\cdot},s)\mapsto \mathbb{E}\left(F((Z_{\cdot}\otimes m)^{\mathcal{B}},s)\right)\ \text{continuous};\label{eq:3-156}\\
    (Z_{\cdot},s)\mapsto D_{X}\mathbb{E}\left(F((Z_{\cdot}\otimes m)^{\mathcal{B}},s)\right),\ D_{X}^{2}\mathbb{E}\left(F((Z_{\cdot}\otimes m)^{\mathcal{B}},s)\right)(Y_{\cdot}),\ \text{continuous, for fixed }Y_{\cdot}.\ \label{eq:3-157}      
\end{numcases}
 We then consider the function $s\mapsto \mathbb{E}\left(F((X_{X_{\cdot}t}(s)\otimes m)^{\mathcal{B}_{t}^{s}},s)\right)$
and study its differentiability.

\subsection{DIFFERENTIABILITY }

We state the following result: 
\begin{theorem}
\label{theo3-10} We assume (\ref{eq:3-1000}), (\ref{eq:3-155}), (\ref{eq:3-156}), (\ref{eq:3-157}). Then the function $s\mapsto \mathbb{E}\left(F((X_{X_{\cdot}t}(s)\otimes m)^{\mathcal{B}_{t}^{s}},s)\right)$
is a.e. differentiable on $(t,T)$ and we have the formula (It\^o's formula):
\begin{align}
\dfrac{d}{ds}\mathbb{E}\left(F((X_{X_{\cdot}t}(s)\otimes m)^{\mathcal{B}_{t}^{s}},s)\right)=&\,\dfrac{\partial}{\partial s}\mathbb{E}\left(F((X_{X_{\cdot}t}(s)\otimes m)^{\mathcal{B}_{t}^{s}},s)\right)+\left\langle D_{X}\mathbb{E}\left(F((X_{X_{\cdot}t}(s)\otimes m)^{\mathcal{B}_{t}^{s}},s)\right),u_{X_{\cdot}t}(s) \right\rangle\nonumber\\
&+\dfrac{1}{2} \left\langle D_{X}^{2}\mathbb{E}\left(F((X_{X_{\cdot}t}(s)\otimes m)^{\mathcal{B}_{t}^{s}},s)\right)(\sigma N_{s}),\sigma N_{s} \right\rangle \label{eq:3-113} \\
&+\dfrac{\beta^{2}}{2}\sum_{j=1}^{n} \left\langle D_{X}^{2}\mathbb{E}\left(F((X_{X_{\cdot}t}(s)\otimes m)^{\mathcal{B}_{t}^{s}},s)\right)(e^{j}),e^{j} \right\rangle,\text{a.e.}\ s\in(0,T),\nonumber 
\end{align}
where $N_{s}$ is an independent Gaussian with values in
$\mathbb{R}^{n}$ with  mean 0 and a unit variance; $N_s$ is independent
of the $\sigma$-algebra $\mathcal{F}_{X_{\cdot}t}^{s}$, and $e^{j}$'s are
the coordinate vectors of $\mathbb{R}^{n}.$ 
\end{theorem}

\begin{proof}
We can write 
\begin{align}
    X_{X_{x}t}(s+\epsilon)&=X_{X_{x}t}(s)+\int_{s}^{s+\epsilon}u_{X_{x}t}(\tau)d\tau+\sigma(w(s+\epsilon)-w(s))+\beta(b(s+\epsilon)-b(s)) =X_{X_{x}t}(s)+\widetilde{X}_{X_{x}t}(s).\label{eq:3-114}
\end{align}
We can use the second-order Taylor formula (\ref{eq:2-105}) to write 
\begin{align}
&\quad \, \cfrac{1}{\epsilon} \Bigg[\mathbb{E} \left(F \left((X_{X_{\cdot}t}(s+\epsilon)\otimes m)^{\mathcal{B}_{t}^{s+\epsilon}},s+\epsilon \right)\right)-\mathbb{E} \left(F \left((X_{X_{\cdot}t}(s)\otimes m)^{\mathcal{B}_{t}^{s}},s+\epsilon \right)\right) \Bigg]\nonumber
\\
&=\left\langle D_{X}\mathbb{E}\left(F((X_{X_{\cdot}t}(s)\otimes m)^{\mathcal{B}_{t}^{s}},s+\epsilon)\right),\dfrac{\widetilde{X}_{X_{x}t}(s)}{\epsilon} \right\rangle\nonumber
\\
&\quad \, +\int_{0}^{1}\int_{0}^{1}\lambda \left\langle D_{X}^{2}\mathbb{E}\left(F(((X_{X_{\cdot}t}(s)+\lambda\mu\widetilde{X}_{X_{x}t}(s))\otimes m)^{\mathcal{B}_{t}^{s}},s+\epsilon)\right) \left(\dfrac{\widetilde{X}_{X_{x}t}(s)}{\sqrt{\epsilon}} \right),\dfrac{\widetilde{X}_{X_{x}t}(s)}{\sqrt{\epsilon}} \right\rangle d\lambda d\mu.\label{eq:3-115}
\end{align}
Since $D_{X}\mathbb{E}\left(F((X_{X_{\cdot}t}(s)\otimes m)^{\mathcal{B}_{t}^{s}},s+\epsilon)\right)$
is $\mathcal{F}_{Xt}^{s}$-measurable,
\[
\left\langle D_{X}\mathbb{E}\left(F((X_{X_{\cdot}t}(s)\otimes m)^{\mathcal{B}_{t}^{s}},s+\epsilon)\right),\dfrac{\widetilde{X}_{X_{x}t}(s)}{\epsilon}\right\rangle= \left\langle D_{X}\mathbb{E}\left(F((X_{X_{\cdot}t}(s)\otimes m)^{\mathcal{B}_{t}^{s}},s+\epsilon)\right),\dfrac{\int_{s}^{s+\epsilon}u_{X_{x}t}(\tau)d\tau}{\epsilon}\right\rangle.
\]
Therefore,
\begin{equation}
\bigg\langle D_{X}\mathbb{E}\left(F((X_{X_{\cdot}t}(s)\otimes m)^{\mathcal{B}_{t}^{s}},s+\epsilon)\right),\dfrac{\widetilde{X}_{X_{x}t}(s)}{\epsilon}\bigg\rangle - \bigg\langle D_{X}\mathbb{E}\left(F((X_{X_{\cdot}t}(s)\otimes m)^{\mathcal{B}_{t}^{s}},s+\epsilon)\right),u_{X_{x}t}(s) \bigg\rangle\rightarrow0.\label{eq:3-116}
\end{equation}
Next, since 
\[
\left\lVert\dfrac{\widetilde{X}_{X_{x}t}(s)}{\sqrt{\epsilon}}-\dfrac{\sigma(w(s+\epsilon)-w(s))}{\sqrt{\epsilon}}-\dfrac{\beta(b(s+\epsilon)-b(s))}{\sqrt{\epsilon}} \right\rVert\leq C\sqrt{\epsilon},
\]
we can write 
\begin{align}
&\int_{0}^{1}\int_{0}^{1}\lambda \left\langle D_{X}^{2}\mathbb{E}\left(F(((X_{X_{\cdot}t}(s)+\lambda\mu\widetilde{X}_{X_{x}t}(s))\otimes m)^{\mathcal{B}_{t}^{s}},s+\epsilon)\right)\left(\dfrac{\widetilde{X}_{X_{x}t}(s)}{\sqrt{\epsilon}}\right),\dfrac{\widetilde{X}_{X_{x}t}(s)}{\sqrt{\epsilon}} \right\rangle d\lambda d\mu \label{eq:3-117}\\
    &-\int_{0}^{1}\int_{0}^{1}\lambda\bigg\langle D_{X}^{2}\mathbb{E} \left(F \left(((X_{X_{\cdot}t}(s)+\lambda\mu\widetilde{X}_{X_{\cdot}t}(s))\otimes m)^{\mathcal{B}_{t}^{s}},s+\epsilon \right)\right) (B(\epsilon)), B(\epsilon)\bigg\rangle d\lambda d\mu\rightarrow0,\nonumber
\end{align}
where
\begin{equation*}
    B(\epsilon) := \dfrac{\sigma(w(s+\epsilon)-w(s))+\beta(b(s+\epsilon)-b(s))}{\sqrt{\epsilon}}.
\end{equation*}
The next step is to observe that 
\begin{align}
  &\int_{0}^{1}\int_{0}^{1}\lambda \bigg\langle D_{X}^{2}\mathbb{E} \left(F\left(((X_{X_{\cdot}t}(s)+\lambda\mu\widetilde{X}_{X_{\cdot}t}(s))\otimes m)^{\mathcal{B}_{t}^{s}},s+\epsilon\right)\right)(B(\epsilon)),B(\epsilon) \bigg\rangle d\lambda d\mu\label{eq:3-118}  \\
  &-\dfrac{1}{2}\bigg\langle D_{X}^{2}\mathbb{E}\left(F((X_{X_{\cdot}t}(s)\otimes m)^{\mathcal{B}_{t}^{s}},s+\epsilon)\right) (B(\epsilon)),B(\epsilon) \bigg\rangle \rightarrow 0.\nonumber
\end{align}

Consider a term like 
\begin{align}
I_{\epsilon}&=\mathbb{E}\Bigg[\int_{\mathbb{R}^{n}}\Bigg[\bigg(\int_{0}^{1}\int_{0}^{1}\lambda D_{X}^{2}\mathbb{E}\left(F\left(((X_{X_{\cdot}t}(s)+\lambda\mu\widetilde{X}_{X_{\cdot}t}(s))\otimes m)^{\mathcal{B}_{t}^{s}},s+\epsilon\right) \right)d\lambda d\mu\nonumber\\
&\quad-D_{X}^{2}\mathbb{E}F\left(\left((X_{X_{\cdot}t}(s)\otimes m)^{\mathcal{B}_{t}^{s}},s+\epsilon\right)\right)\bigg)(\sigma e^{j})\cdot\sigma e^{j}\dfrac{(w_{j}(s+\epsilon)-w_{j}(s))^{2}}{\epsilon}\Bigg]dm(x)\Bigg].\nonumber
\end{align}
Using (\ref{eq:3-157}) and the fact that $\dfrac{(w_{j}(s+\epsilon)-w_{j}(s))^{2}}{\epsilon}$ has a bounded $L^2$ norm, we deduce that $I_\epsilon \to 0$; indeed, by \eqref{eq:3-157},
\begin{align}
|I_{\epsilon}|\leq&\  C\mathbb{E}\int_{\mathbb{R}^{n}}\bigg|\bigg(\int_{0}^{1}\int_{0}^{1}\lambda D_{X}^{2}\mathbb{E}F(((X_{X_{\cdot}t}(s)+\lambda\mu\widetilde{X}_{X_{x}t}(s))\otimes m)^{\mathcal{B}_{t}^{s}},s+\epsilon)d\lambda d\mu \nonumber \\
&\qquad \qquad -D_{X}^{2}\mathbb{E}F((X_{X_{\cdot}t}(s)\otimes m)^{\mathcal{B}_{t}^{s}},s+\epsilon)\bigg)(\sigma e^{j})\bigg| \dfrac{(w_{j}(s+\epsilon)-w_{j}(s))^{2}}{\epsilon}dm(x)\nonumber\\
\leq&\  C\sqrt{\mathbb{E}\int_{\mathbb{R}^{n}}\Big|\Big(\int_{0}^{1}\int_{0}^{1}\lambda D_{X}^{2}\mathbb{E}F\left(((X_{X_{\cdot}t}(s)+\lambda\mu\widetilde{X}_{X_{x}t}(s))\otimes m)^{\mathcal{B}_{t}^{s}},s+\epsilon\right)d\lambda d\mu}\nonumber\\
&\overline{-D_{X}^{2}\mathbb{E}F\left((X_{X_{\cdot}t}(s)\otimes m)^{\mathcal{B}_{t}^{s}},s+\epsilon\right)\Big)(\sigma e^{j})\Big|^{2}dm(x)}\rightarrow0,\nonumber
\end{align}
It remains to check that 
\begin{align}
&\quad \, \left\langle D_{X}^{2}\mathbb{E}\left(F((X_{X_{\cdot}t}(s)\otimes m)^{\mathcal{B}_{t}^{s}},s+\epsilon)\right)(B(\epsilon)), B(\epsilon) \right\rangle\label{eq:3-119}\\
&=\left\langle D_{X}^{2}\mathbb{E}\left(F\left((X_{X_{\cdot}t}(s)\otimes m)^{\mathcal{B}_{t}^{s}},s+\epsilon\right)\right)(\sigma N_{s}),\sigma N_{s} \right\rangle +\dfrac{\beta^{2}}{2}\sum_{j=1}^{n} \left\langle D_{X}^{2}\mathbb{E}\left(F\left((X_{X_{\cdot}t}(s)\otimes m)^{\mathcal{B}_{t}^{s}},s+\epsilon\right)\right)(e^{j}),e^{j}\right\rangle,\nonumber
\end{align}
using Remark \ref{rem3-10} and the explicit formula (\ref{eq:2-100}). 

Collecting results, we can assert that 
\begin{align}
&\quad \int_{0}^{1}\int_{0}^{1}\lambda \left\langle D_{X}^{2}\mathbb{E}\left(F\left(((X_{X_{\cdot}t}(s)+\lambda\mu\widetilde{X}_{X_{x}t}(s))\otimes m)^{\mathcal{B}_{t}^{s}},s+\epsilon\right)\right) \left(\dfrac{\widetilde{X}_{X_{x}t}(s)}{\sqrt{\epsilon}} \right),\dfrac{\widetilde{X}_{X_{x}t}(s)}{\sqrt{\epsilon}} \right\rangle d\lambda d\mu\label{eq:3-120}\\
&-\dfrac{1}{2} \left\langle D_{X}^{2}\mathbb{E}\left(F\left((X_{X_{\cdot}t}(s)\otimes m)^{\mathcal{B}_{t}^{s}},s+\epsilon\right)\right)(\sigma N_{s}),\sigma N_{s}\right\rangle +\dfrac{\beta^{2}}{2}\sum_{j=1}^{n}\left\langle D_{X}^{2}\mathbb{E}\left(F\left((X_{X_{\cdot}t}(s)\otimes m)^{\mathcal{B}_{t}^{s}},s+\epsilon\right)\right)(e^{j}),e^{j}\right \rangle \rightarrow0.\nonumber
\end{align}
So 
\begin{align}
&\cfrac{1}{\epsilon}\Bigg[\mathbb{E}\left(F\left((X_{X_{\cdot}t}(s+\epsilon)\otimes m)^{\mathcal{B}_{t}^{s+\epsilon}},s+\epsilon\right)\right)-\mathbb{E}\left(F\left((X_{X_{\cdot}t}(s)\otimes m)^{\mathcal{B}_{t}^{s}},s+\epsilon\right)\right)\Bigg]\label{eq:3-121}\\
&-\bigg\langle D_{X}\mathbb{E}\left(F\left((X_{X_{\cdot}t}(s)\otimes m)^{\mathcal{B}_{t}^{s}},s+\epsilon\right)\right),u_{X_{x}t}(s)\bigg\rangle -\dfrac{1}{2}\bigg\langle D_{X}^{2}\mathbb{E}\left(F\left((X_{X_{\cdot}t}(s)\otimes m)^{\mathcal{B}_{t}^{s}},s+\epsilon\right)\right)(\sigma N_{s}),\sigma N_{s}\bigg\rangle \nonumber\\
&-\dfrac{\beta^{2}}{2}\sum_{j=1}^{n}\bigg\langle D_{X}^{2}\mathbb{E}\left(F\left((X_{X_{\cdot}t}(s)\otimes m)^{\mathcal{B}_{t}^{s}},s+\epsilon\right)\right)(e^{j}),e^{j}\bigg\rangle \rightarrow0.\nonumber  
\end{align}
From the continuity in $s,$ (\ref{eq:3-157}), we then get:
\begin{align}
&\quad \,\,\, \cfrac{1}{\epsilon} \left[\mathbb{E}\left(F\left((X_{X_{\cdot}t}(s+\epsilon)\otimes m)^{\mathcal{B}_{t}^{s+\epsilon}},s+\epsilon\right)\right)-\mathbb{E}\left(F((X_{X_{\cdot}t}(s)\otimes m)^{\mathcal{B}_{t}^{s}},s+\epsilon)\right) \right]\label{eq:3-122}\\
&\rightarrow \left\langle D_{X}\mathbb{E}\left(F\left((X_{X_{\cdot}t}(s)\otimes m)^{\mathcal{B}_{t}^{s}},s\right)\right),u_{X_{x}t}(s) \right\rangle
+\dfrac{1}{2} \left\langle D_{X}^{2}\mathbb{E} \left(F\left((X_{X_{\cdot}t}(s)\otimes m)^{\mathcal{B}_{t}^{s}},s\right)\right)(\sigma N_{s}),\sigma N_{s}\right\rangle \nonumber\\
&\quad \,\, +\dfrac{\beta^{2}}{2}\sum_{j=1}^{n} \left\langle D_{X}^{2}\mathbb{E}\left(F\left((X_{X_{\cdot}t}(s)\otimes m)^{\mathcal{B}_{t}^{s}},s\right)\right)(e^{j}),e^{j} \right \rangle \nonumber.
\end{align}
From the partial differentiability in $s,$ see the second part of
(\ref{eq:3-152}), we finally obtain:
\begin{align}
&\quad \,\,\,\, \cfrac{1}{\epsilon} \Bigg[\mathbb{E}\left(F\left((X_{X_{\cdot}t}(s+\epsilon)\otimes m)^{\mathcal{B}_{t}^{s+\epsilon}},s+\epsilon\right)\right)-\mathbb{E}\left(F\left((X_{X_{\cdot}t}(s)\otimes m)^{\mathcal{B}_{t}^{s}},s\right)\right)\Bigg]\label{eq:3-123}\\
&\rightarrow\dfrac{\partial}{\partial s}\mathbb{E}\left(F\left((X_{X_{\cdot}t}(s)\otimes m)^{\mathcal{B}_{t}^{s}},s\right)\right)+\left\langle D_{X}\mathbb{E}\left(F\left((X_{X_{\cdot}t}(s)\otimes m)^{\mathcal{B}_{t}^{s}},s\right)\right),u_{X_{x}t}(s) \right\rangle \nonumber\\
&\quad \,\, +\dfrac{1}{2}\left\langle D_{X}^{2}\mathbb{E}\left(F\left((X_{X_{\cdot}t}(s)\otimes m)^{\mathcal{B}_{t}^{s}},s\right)\right)(\sigma N_{s}),\sigma N_{s}\right\rangle +\dfrac{\beta^{2}}{2}\sum_{j=1}^{n}\left\langle D_{X}^{2}\mathbb{E}\left(F\left((X_{X_{\cdot}t}(s)\otimes m)^{\mathcal{B}_{t}^{s}},s\right)\right)(e^{j}),e^{j}\right\rangle ,\ \text{a.e. }s,\nonumber
\end{align}
which completes the proof of (\ref{eq:3-113}).
\end{proof}

\section{CONTROL PROBLEM} \label{sec:control problem}

\subsection{SETTING OF THE PROBLEM}

The space of controls is the Hilbert space $L_{\mathcal{F}_{Xt}}^{2}(t,T;\mathcal{H}_{m}).$
A control is denoted by $v_{X_{\cdot}t}(s),$ where $X_{\cdot}=X_{x}\in\mathcal{H}_{m}$,
independent of $\mathcal{F}_{t}^{s}=\sigma(w(\tau)-w(t),b(\tau)-b(t),t\leq\tau\leq s),$ $\forall s.$
The state, denoted $X_{X_{\cdot}t}(s)$, associated with a control $v_{X_{\cdot}t}(\cdot)$,
is defined by: 
\begin{equation}
X_{X_{x}t}(s):=X_{x}+\int_{t}^{s}v_{X_{x}t}(\tau)d\tau+\sigma(w(s)-w(t))+\beta(b(s)-b(t)),\ s>t.\label{eq:3-500}
\end{equation}
We want to minimize the functional 
\begin{align}
J_{X_{\cdot}t}(v_{X_{\cdot}t}(\cdot))&=\mathbb{E}\left[\int_{t}^{T} \left[\int_{\mathbb{R}^{n}}l(X_{X_{x}t}(s),v_{X_{x}t}(s))dm(x)+F((X_{X_{\cdot}t}(s)\otimes m)^{\mathcal{B}_{t}^{s}}) \right]ds\right]\label{eq:3-503}\\
&\quad \, +\mathbb{E} \left[\int_{\mathbb{R}^{n}}h(X_{X_{x}t}(T))dm(x)+F_{T}((X_{X_{\cdot}t}(T)\otimes m)^{\mathcal{B}_{t}^{T}}) \right].\nonumber   
\end{align}
This problem is completely equivalent to the following: for a family
of processes $v_{\eta t}(s)$ in $L_{\mathcal{F}_{t}}^{2}(t,T;\mathcal{H}_{m})$,
for any fixed $\eta\in \mathbb{R}^{n},$ define the state $X_{\eta}(s)$ by 
\begin{equation}
X_{\eta t}(s):=\eta+\int_{t}^{s}v_{\eta t}(\tau)d\tau+\sigma(w(s)-w(t))+\beta(b(s)-b(t)),\ s>t.\label{eq:3-504}
\end{equation}
We want to minimize the functional 
\begin{align}
J_{X_{\cdot}\otimes m,t}(v_{\cdot,t}(\cdot))&=\mathbb{E}\left[\int_{t}^{T} \left[\int_{\mathbb{R}^{n}}l(X_{\eta\,t}(s),v_{\eta\,t}(s))d(X_{\cdot}\otimes m)(\eta)+F((X_{\cdot,t}(s)\otimes(X_{\cdot}\otimes m))^{\mathcal{B}_{t}^{s}})\right]ds\right]\label{eq:3-505}\\
&\quad \, +\mathbb{E}\left[\int_{\mathbb{R}^{n}}h(X_{\eta\,t}(T))d(X_{\cdot}\otimes m)(\eta)+F_{T}((X_{\cdot,t}(T)\otimes(X_{\cdot}\otimes m))^{\mathcal{B}_{t}^{T}})\right].\nonumber
\end{align}

For fixed $X_{\cdot},$ the functional $v_{X_{\cdot}t}(\cdot)\rightarrow J_{X_{\cdot}t}(v_{X_{\cdot}t}(\cdot))$
is defined on the Hilbert space $L_{\mathcal{F}_{X_{\cdot}t}}^{2}(t,T;\mathcal{H}_{m})$, a
sub Hilbert space of $L^{2}(t,T;\mathcal{H}_{m})$. The functional
$J_{X_{\cdot}\otimes m,t}(v_{\cdot,t}(\cdot))$ is defined on $L_{\mathcal{F}_{t}}^{2}(t,T;\mathcal{H}_{X_{\cdot}\otimes m}$).
We shall make precise the assumptions in the next section.  

\subsection{\label{subsec:ASSUMPTIONS}ASSUMPTIONS}

We assume that 
\begin{align}
&|l(x,v)|\leq c_{l}(1+|x|^{2}+|v|^{2}), \, |l_{x}(x,v)|, \, |l_{v}(x,v)|\leq c_{l}(1+|x|^{2}+|v|^{2})^{\frac{1}{2}},\nonumber\\
&|l_{xx}(x,v)|,\ |l_{xv}(x,v)|,\ |l_{vv}(x,v)|\leq c_{l},\label{eq:4-1000}\\
&|h(x)|\leq c_{h}(1+|x|^{2}),\ |h_{x}(x)|\leq c_{h}(1+|x|^{2})^{\frac{1}{2}},\ |h_{xx}(x)|\leq c_{h},\label{eq:4-1001}\\
&l_{xx}(x,v),l_{xv}(x,v),l_{vv}(x,v),h_{xx}(x)\ \text{continuous},\label{eq:4-1002}\\
&l_{xx}(x,v)\xi\cdot\xi+2l_{xv}(x,v)\eta\cdot\xi+l_{vv}(x,v)\eta\cdot\eta\geq\lambda|\eta|^{2}-c'_{l}|\xi|^{2},\,\forall\xi,\eta\in \mathbb{R}^{n},\label{eq:4-1003}\\
&h_{xx}\xi\cdot\xi\geq-c'_{h}|\xi|^{2},\label{eq:4-1004}\\
&m\mapsto F(m),F_{T}(m)\ \text{continuous},\label{eq:3-1}\\
&|F(m)|\leq c\left(1+\int_{\mathbb{R}^{n}}|x|^{2}dm(x)\right),\,\quad |F_{T}(m)|\leq c_{T}\left(1+\int_{\mathbb{R}^{n}}|x|^{2}dm(x)\right),\nonumber\\
&(x,m)\mapsto\dfrac{d}{d\nu}F(m)(x),\quad \dfrac{d}{d\nu}F_{T}(m)(x)\ \text{continuous},\label{eq:3-2}\\
&\left|\dfrac{d}{d\nu}F(m)(x)\right|\leq c(1+|x|^{2}),\,\quad \left|\dfrac{d}{d\nu}F_{T}(m)(x)\right|\leq c_{T}(1+|x|^{2}),\nonumber\\
&(x,m)\mapsto D\dfrac{d}{d\nu}F(m)(x),\ \quad D\dfrac{d}{d\nu}F_{T}(m)(x)\ \text{continuous},\label{eq:3-3}\\
&\left|D\dfrac{d}{d\nu}F(m)(x)\right|\leq c(1+|x|^{2})^{\frac{1}{2}},\,\quad \left|D\dfrac{d}{d\nu}F_{T}(m)(x)\right|\leq c_{T}(1+|x|^{2})^{\frac{1}{2}},\nonumber\\
&(x,m)\mapsto D^{2}\dfrac{d}{d\nu}F(m)(x),\quad D^{2}\dfrac{d}{\,d\nu}F_{T}(m)(x)\ \text{continuous},\label{eq:3-4}\\
&\left|D^{2}\dfrac{d}{d\nu}F(m)(x)\right|\leq c,\quad \left|D^{2}\dfrac{d}{d\nu}F_{T}(m)(x)\right|\leq c_{T},\nonumber\\
&(x,x^{1},m)\mapsto DD_{1}\frac{d^{2}F}{d\nu^{2}}(m)(x,x^{1}), \ \quad DD_{1}\frac{d^{2}F_{T}}{d\nu^{2}}(m)(x,x^{1})\ \text{continuous},\label{eq:3-50}\\
&\left|D_{1}\frac{d^{2}F}{d\nu^{2}}(m)(x,x^{1})\right|\leq c(1+|x|),\ \quad \left|D\frac{d^{2}F}{d\nu^{2}}(m)(x,x^{1})\right|\leq c(1+|x^{1}|),\nonumber\\
&\left|DD_{1}\frac{d^{2}F}{d\nu^{2}}(m)(x,x^{1})\right|\leq c,\ \quad \left|DD_{1}\frac{d^{2}F_{T}}{d\nu^{2}}(m)(x,x^{1})\right|\leq c_{T},\nonumber\\
&D^{2}\dfrac{d}{d\nu}F(m)(x)\geq-c'I,\ \quad D^{2}\dfrac{d}{d\nu}F_{T}(m)(x)\geq-c'_{T}I,\label{eq:3-51}\\
&DD_{1}\frac{d^{2}F}{d\nu^{2}}(m)(x,x^{1})\geq-c'I,\ \quad DD_{1}\frac{d^{2}F_{T}}{d\nu^{2}}(m)(x,x^{1})\geq-c'_{T}I.\label{eq:3-52}
\end{align}
Note that, with the above assumptions on $F(m)$ and $F_{T}(m)$
the functions $Z_{\cdot}\mapsto \mathbb{E}\left(F((Z_{\cdot}\otimes m)^{\mathcal{B}})\right),\mathbb{E}\left(F_{T}((Z_{\cdot}\otimes m)^{\mathcal{B}})\right)$
are G\textroundcap{a}teaux differentiable and Propositions \ref{prop2-1},
\ref{prop2-2}, \ref{prop2-3} apply.

{Conditions \eqref{eq:3-51} and \eqref{eq:3-52} are the crucial \emph{monotonicity assumptions}. When the constant $c'$ (resp.~$c_T'$) is zero, they imply that $F$ (resp.~$F_T$) satisfies the displacement monotonicity assumption of Gangbo, et al.~\cite{GMMZ}. This is in contrast to the more widely used Lasry-Lions monotonicity condition, as in \cite{CDLL}. See also \cite{GrM} for a comparison of monotonicity conditions. 
We also refer to our recent work \cite{BHTY} for a discussion on $\beta$-monotonicity for the forward-backward system associated with MFGs, which can include the displacement monotonicity condition \cite{ARY,HT} and the small mean field effect condition \cite{BTWY} for MFGs.
Most of the remaining assumptions are essentially regularity requirements on the data, although the convexity assumption \eqref{eq:4-1003} also plays a crucial role in the well-posedness of the control problem and the resulting regularity of the value function.}

\subsection{DIFFERENTIABILITY OF $v_{Xt}(\cdot)\mapsto J_{Xt}(v_{Xt}(\cdot))$}

Considering the map $v_{X_{\cdot}t}(\cdot)\mapsto J_{X_{\cdot}t}(v_{X_{\cdot}t}(\cdot))$
as a functional on the Hilbert space $L_{\mathcal{F}_{X_{\cdot}t}}^{2}(t,T;\mathcal{H}_{m})$, we get the following lemma.
\begin{lemma}
\label{lem3-1}Under the assumptions (\ref{eq:4-1000}), (\ref{eq:4-1001}),
(\ref{eq:3-1}), (\ref{eq:3-2}), (\ref{eq:3-3}), the functional $J_{X_{\cdot}t}(v_{X_{\cdot}t}(\cdot))$
has a G\^ateaux derivative, given by 
\begin{equation}
D_{v}J_{X_{\cdot}t}(v_{X_{\cdot}t}(\cdot))(s)=l_{v}(X_{X_{\cdot}t}(s),v_{X_{\cdot}t}(s))+Z_{X_{\cdot}t}(s),\label{eq:3-5}
\end{equation}
with $Z_{X_{\cdot}t}(\cdot)\in L_{\mathcal{F}_{X_{\cdot}t}}^{2}(t,T;\mathcal{H}_{m})$, solution
of the BSDE (Backward Stochastic Differential Equation)
\begin{equation}
\left\{
\begin{aligned}
-dZ_{X_{\cdot}t}(s)&= \left(l_{x}(X_{X_{\cdot}t}(s),v_{X_{\cdot}t}(s))+D_{X}\mathbb{E}\left(F((X_{X_{\cdot}t}(s)\otimes m)^{\mathcal{B}_{t}^{s}})\right)\right)ds-\sum_{j=1}^{n}r_{X_{\cdot}t}^{j}(s)dw_{j}(s)-\sum_{j=1}^{n}\rho_{X_{\cdot} t}^{j}(s)db_{j}(s),\\
Z_{X_{\cdot}t}(T)&=h_{x}(X_{X_{\cdot}t}(T))+D_{X}\mathbb{E}\left(F_{T}((X_{X_{\cdot}t}(T)\otimes m)^{\mathcal{B}_{t}^{T}})\right),
\end{aligned}\right.\label{eq:3-600}
\end{equation}
where \textup{$r_{X_{\cdot}t}^{j}(s), \rho_{X_{\cdot}t}^{j}(s)\in L_{\mathcal{F}_{X_{\cdot}t}}^{2}(t,T;\mathcal{H}_{m})$, $j=1,\cdots,n$.}
\end{lemma}

The proof can be found in Appendix A.

\subsection{CONVEXITY}

We next state the following proposition:
\begin{proposition}
\label{prop3-1}We assume (\ref{eq:4-1000}), (\ref{eq:4-1001}), (\ref{eq:4-1002}), (\ref{eq:4-1003}), (\ref{eq:4-1004}), (\ref{eq:3-1}), (\ref{eq:3-2}), (\ref{eq:3-3}), (\ref{eq:3-4}), (\ref{eq:3-50}), (\ref{eq:3-51}), (\ref{eq:3-52})
and 
\begin{equation}
\lambda-T(c'_{T}+c'_{h})-(c'+c'_{l})\dfrac{T^{2}}{2}>0,\label{eq:3-6}
\end{equation}
then the functional $J_{X_{\cdot}t}(v_{X_{\cdot}t}(\cdot))$ is strictly convex.
It is coercive, i.e. $J_{X_{\cdot}t}(v_{X_{\cdot}t}(\cdot))\rightarrow+\infty,$
as $\int_{t}^{T}||v_{X_{\cdot}t}(s)||^{2}ds\rightarrow+\infty.$ Consequently,
there exists one and only one minimum of $J_{X_{\cdot}t}(v_{X_{\cdot}t}(\cdot)).$
\end{proposition}

The proof can be found in Appendix A.

\subsection{NECESSARY AND SUFFICIENT CONDITION OF OPTIMALITY}

According to Proposition \ref{prop3-1}, there exists one and only
one optimal control $u_{X_{\cdot}t}(s)$. It must satisfy the necessary
and sufficient condition $D_{v}J_{X_{\cdot}t}(u_{X_{\cdot}t}(\cdot))(s)=0.$ Calling
$X_{X_{\cdot}t}$$(s)$ the optimal state and $Z_{X_{\cdot}t}(s)$, $r_{X_{\cdot}t}^{j}(s)$, $\rho_{X_{\cdot}t}^{j}(s)$
the solution of the BSDE (\ref{eq:3-600}), then the set $X_{X_{\cdot}t}(s)$,
$Z_{X_{\cdot}t}(s)$, $u_{X_{\cdot}t}(s)$$, r_{X_{\cdot}t}^{j}(s)$, $\rho_{X_{\cdot}t}^{j}(s)$
is the unique solution of the system 
\begin{equation}
\label{eq:3-7}
\left\{ \begin{aligned}   
&X_{X_{x}t}(s)=X_{x}+\int_{t}^{s}u_{X_{x}t}(\tau)d\tau+\sigma(w(s)-w(t))+\beta(b(s)-b(t));\\
&-dZ_{X_{\cdot}t}(s)=\left(l_{x}(X_{X_{\cdot}t}(s),u_{X_{\cdot}t}(s))+D_{X}\mathbb{E}\left(F((X_{X_{\cdot}t}(s)\otimes m)^{\mathcal{B}_{t}^{s}})\right)\right)ds-\sum_{j=1}^{n}r_{X_{\cdot}t}^{j}(s)dw_{j}(s)-\sum_{j=1}^{n}\rho_{X_{\cdot}t}^{j}(s)db_{j}(s);\\
&Z_{X_{\cdot}t}(T)=h_{x}(X_{X_{\cdot}t}(T))+D_{X}\mathbb{E}\left(F_{T}((X_{X_{\cdot}t}(T)\otimes m)^{\mathcal{B}_{t}^{T}})\right); \\
&l_{v}(X_{X_{\cdot}t}(s),u_{X_{\cdot}t}(s))+Z_{X_{\cdot}t}(s)=0. 
\end{aligned} \right.
\end{equation}
We can equivalently express the necessary and sufficient optimality
conditions as follows: there exists a unique $u_{\eta t}(\cdot)\in L_{\mathcal{F}_{t}}^{2}(t,T;\mathcal{H}_{m})$, which satisfies the condition
\begin{align}
\int_{t}^{T}\mathbb{E}\Bigg(\int_{\mathbb{R}^{n}}\bigg[&l_{v}(X_{\eta t}(s),u_{\eta t}(s))+\int_{s}^{T}\bigg(l_{x}(X_{\eta t}(\tau),u_{\eta t}(\tau))\label{eq:3-603}\\
&+D\dfrac{dF}{d\nu}((X_{\cdot, t}(\tau)\otimes(X_{\cdot}\otimes m))^{\mathcal{B}_{t}^{\tau}})(X_{\eta t}(\tau))\bigg)d\tau+h_{x}(X_{\eta t}(T))\nonumber\\
&+D\dfrac{dF_{T}}{d\nu}((X_{\cdot, t}(T)\otimes(X_{\cdot}\otimes m))^{\mathcal{B}_{t}^{T}})(X_{\eta t}(T))\bigg]\cdot\widetilde{v}_{\eta t}(s)d(X_{\cdot}\otimes m)(\eta)\Bigg)ds=0\nonumber
\end{align}
with 
\begin{equation}
X_{\eta t}(s)=\eta+\int_{t}^{s}u_{\eta t}(\tau)d\tau+\sigma(w(s)-w(t))+\beta(b(s)-b(t)),\label{eq:3-604}
\end{equation}
for any $\widetilde{v}_{\eta t}(\cdot)\in L_{\mathcal{F}_{t}}^{2}(t,T;\mathcal{H}_{m}).$
The advantage of this statement is that the process $Z_{\eta t}(s)$
does not appear explicitly in the condition. We have the important
property:
\begin{proposition}
\label{prop3-2} We have 
\begin{align}
&u_{X_{\cdot}t}(s)=u_{X_{X_{\cdot}t}(t+\epsilon),t+\epsilon}(s), \quad  \ X_{X_{\cdot}t}(s)=X_{X_{X_{\cdot}t}(t+\epsilon),t+\epsilon}(s), \label{eq:3-605}\\
&r_{X_{\cdot}t}^{j}(s)=r_{X_{X_{\cdot}t}(t+\epsilon),t+\epsilon}^{j}(s), \quad \rho_{X_{\cdot}t}^{j}(s)=\rho_{X_{X_{\cdot}t}(t+\epsilon),t+\epsilon}^{j}(s),\ \forall s>t+\epsilon.\nonumber   
\end{align}
\end{proposition}

The proof can be found in Appendix A. 
\begin{remark}
\label{rem4-1}We may enlarge the subspace of $L_{\mathcal{F}_{X_{\cdot}t}}^{2}(t,T;\mathcal{H}_{m})$
of controls against which $u_{Xt}(\cdot)$ is optimal. Consider a $\sigma$-algebra 
\[
\mathcal{X}_{x}=\sigma(X_{x},X_{x}^{1},\cdots,X_{x}^{j},\cdots),X_{x}^{j}\ \text{independent of }\mathcal{F}_{t}.
\]
 If we change the space of controls from $L_{\mathcal{F}_{X_{\cdot}t}}^{2}(t,T;\mathcal{H}_{m})$
to $L_{\mathcal{F}_{\mathcal{X}_{\cdot}t}}^{2}(t,T;\mathcal{H}_{m}),$ it is straightforward to verify that the system of equations \eqref{eq:3-7} satisfies the necessary conditions for optimality of the control problem with augmented $\sigma$-algebras. By uniqueness, these equations provide the solution to the necessary conditions. Consequently, the optimal control remains unchanged. 
\end{remark}

\subsection{VALUE FUNCTION}

We can express the value function 
\begin{align}
V(X\otimes m,t)=&\int_{t}^{T}\mathbb{E} \left[\int_{\mathbb{R}^{n}}l(X_{X_{\cdot}t}(s),u_{X_{\cdot}t}(s))dm(x)ds+F((X_{X_{\cdot}t}(s)\otimes m)^{\mathcal{B}_{t}^{s}}) \right]ds\label{eq:3-11}\\
&+\mathbb{E} \left[\int_{\mathbb{R}^{n}}h(X_{X_{\cdot}t}(T))dm(x)+F_{T}((X_{X_{\cdot}t}(T)\otimes m)^{\mathcal{B}_{t}^{T}}) \right]. \nonumber 
\end{align}
This quantity depends only on the probability measure $X\otimes m$
and $t.$ It can be written as follows:
\begin{align}
V(X\otimes m,t)=&\int_{t}^{T}\mathbb{E}\left[\int_{\mathbb{R}^{n}}l(X_{\xi t}(s),u_{\xi t}(s))d(X\otimes m)(\xi)+F((X_{\cdot ,t}(s)\otimes(X_{_{'}}\otimes m))^{\mathcal{B}_{t}^{s}})\right]ds\label{eq:3-12}\\
&+\mathbb{E}\left[\int_{\mathbb{R}^{n}}h(X_{\xi t}(T))d(X_{\cdot}\otimes m)(\xi)+F_{T}((X_{\cdot ,t}(T)\otimes(X_{_{'}}\otimes m))^{\mathcal{B}_{t}^{T}})\right].\nonumber
\end{align}

\subsection{OPTIMALITY PRINCIPLE}
The optimality principle is key to writing the Bellman equation. It is expressed as follows: 
\begin{align}
V(X\otimes m,t)=&\int_{t}^{t+\epsilon}\mathbb{E}\left[\int_{\mathbb{R}^{n}}l(X_{X_{x}t}(s),u_{X_{x}t}(s))dm(x)+F((X_{X_{\cdot}t}(s)\otimes m)^{\mathcal{B}_{t}^{s}})\right]ds\nonumber\\
&+\mathbb{E}\left[V((X_{X_{\cdot}t}(t+\epsilon)\otimes m)^{\mathcal{B}_{t}^{t+\epsilon}},t+\epsilon)\right]\nonumber\\
=&\int_{t}^{t+\epsilon}\mathbb{E} \left[\int_{\mathbb{R}^{n}}l(X_{\eta t}(s),u_{\eta t}(s))d(X_{_{'}}\otimes m)(\eta)+F((X_{\cdot, t}(s)\otimes(X_{_{'}}\otimes m))^{\mathcal{B}_{t}^{s}})\right]ds\nonumber\\
&+\mathbb{E}\left[V((X_{\cdot, t}(t+\epsilon)\otimes(X_{_{'}}\otimes m))^{\mathcal{B}_{t}^{t+\epsilon}},t+\epsilon)\right]. \label{eq:3-16}
\end{align}

This is an immediate consequence of Proposition \ref{prop3-2}. We
have indeed
\begin{align}
&\quad \, \int_{t+\epsilon}^{T}\mathbb{E} \left[\int_{\mathbb{R}^{n}}l(X_{\eta t}(s),u_{\eta t}(s))d(X_{_{'}}\otimes m)(\eta)+F((X_{\cdot, t}(s)\otimes(X_{_{'}}\otimes m))^{\mathcal{B}_{t}^{s}})\right]ds\label{eq:3-17}\\
&\quad \, +\mathbb{E} \left[\int_{\mathbb{R}^{n}}h(X_{\eta t}(T))d(X_{_{'}}\otimes m)(\eta)+F_{T}((X_{\cdot, t}(T)\otimes(X_{_{'}}\otimes m))^{\mathcal{B}_{t}^{T}})\right]\nonumber\\
&=\int_{t+\epsilon}^{T}\mathbb{E}\bigg[\int_{\mathbb{R}^{n}}l(X_{\eta,t+\epsilon}(s),u_{\eta,t+\epsilon}(s))d((X_{\cdot, t}(t+\epsilon)\otimes(X_{_{'}}\otimes m))^{\mathcal{B}_{t}^{t+\epsilon}})(\eta)\nonumber\\
&\qquad\qquad +F((X_{\cdot, t+\epsilon}(s)\otimes(X_{\cdot, t}(t+\epsilon)\otimes(X_{_{'}}\otimes m))^{\mathcal{B}_{t}^{t+\epsilon}})^{\mathcal{B}_{t+\epsilon}^{s}})\bigg]ds\nonumber\\
&\quad \, +\mathbb{E}\left[\int_{\mathbb{R}^{n}}h(X_{\eta,t+\epsilon}(T))d((X_{\cdot, t}(t+\epsilon)\otimes(X_{_{'}}\otimes m))^{\mathcal{B}_{t}^{t+\epsilon}})(\eta)+F_{T}((X_{\cdot, t+\epsilon}(T)\otimes(X_{\cdot, t}(t+\epsilon)\otimes(X_{_{'}}\otimes m))^{\mathcal{B}_{t}^{t+\epsilon}})^{\mathcal{B}_{t+\epsilon}^{T}})\right]\nonumber\\
&=\mathbb{E}\left[V((X_{\cdot, t}(t+\epsilon)\otimes(X_{_{'}}\otimes m))^{\mathcal{B}_{t}^{t+\epsilon}},t+\epsilon)\right].\nonumber
\end{align}

\section{PROPERTIES OF THE VALUE FUNCTION} \label{sec:properties of value}

\subsection{BOUNDS}

We begin with the following proposition:
\begin{proposition}
\label{prop4-1} We make the assumptions of Proposition \ref{prop3-1}.
We have 
\begin{align}
&||X_{X_{\cdot}t}(s)||,\ ||Z_{X_{\cdot}t}(s)||,\ ||u_{X_{\cdot}t}(s)||\leq C_{T}(1+||X||),\forall s\in(t,T),\label{eq:4-1}\\
&\int_{t}^{T} \left\lVert r_{X_{\cdot}t}^{j}(s)\right\rVert^{2}ds,\ \quad \int_{t}^{T}\left\lVert \rho_{X_{\cdot}t}^{j}(s)\right\rVert^{2}ds\leq C_{T}(1+||X||^{2}),\label{eq:4-200}\\
&|V(X_{\cdot}\otimes m,t)|\leq C_{T}(1+||X||^{2}),\label{eq:4-2}
\end{align}
where $C_{T}$ is a constant depending only on the constants of the
problem and $T.$ It is independent of $X_{\cdot}, m$, and $t<s<T.$
\end{proposition}

The proof can be found in Appendix B.

\subsection{REGULARITY IN $X$ }

We study now the regularity of $V(X\otimes m,t)$ with respect to
$X.$ This functional is defined of the closed subspace of $\mathcal{H}_{m}$
of random fields $X_{x}$ independent of $\mathcal{F}_{t}$. We have: 
\begin{proposition}
\label{prop4-2} We make the assumptions of Proposition \ref{prop3-1}.
Then the functional $X\rightarrow V(X\otimes m,t)$ is G\^ateaux differentiable
and 
\begin{equation}
D_{X}V(X_{_{\cdot}}\otimes m,t)=Z_{X_{\cdot}t}(t).\label{eq:4-3}
\end{equation}
Also, we have the Lipschitz property: 
\begin{equation}
||D_{X}V(X^{1}\otimes m,t)-D_{X}V(X^{2}\otimes m,t)||\leq C_{T}||X^{1}-X^{2}||,\label{eq:4-4}
\end{equation}
for all $X^{1},X^{2}\in\mathcal{H}_{m}$, independent of $\mathcal{F}_{t}$.
\end{proposition}

The proof can be found in Appendix B.

\subsection{REGULARITY IN TIME}

We state the following proposition:
\begin{proposition}
\label{prop4-5}We make the assumptions of Proposition \ref{prop3-1}.
We have the inequalities: for $X=X_{xt}$ independent of $\mathcal{W}_{t}$,
\begin{equation}
|V(X\otimes m,t+\epsilon)-V(X\otimes m,t)|\leq C_{T}\epsilon(1+||X||^{2}),\label{eq:4-26}
\end{equation}
\begin{equation}
||D_{X}V(X\otimes m,t+\epsilon)-D_{X}V(X\otimes m,t)||\leq C_{T}\epsilon||X||+C_{T}\epsilon^{\frac{1}{2}}.\label{eq:4-27}
\end{equation}
\end{proposition}

The proof can be found in Appendix B.

\section{FUNCTIONAL DERIVATIVE OF $V(m,t)$}

\subsection{THE CASE $X=J$}
When $X=J,$ meaning $J_{x}=x,$ we have $J\otimes m=m.$ The
processes $u_{Jt}(s)$, $X_{Jt}(s)$, $Z_{Jt}(s)$, $r_{Jt}^{j}(s)$, and $\rho_{Jt}^{j}(s)$
will be denoted by $u_{xmt}(s)$, $X_{xmt}(s)$, $Z_{xmt}(s)$, $r_{xmt}^{j}(s)$, $\rho_{xmt}^{j}(s)$ to emphasize the dependence in $m.$ The system (\ref{eq:3-7}) becomes:
\begin{equation}
\left\{\begin{aligned}
&X_{xmt}(s)=x+\int_{t}^{s}u_{xmt}(\tau)d\tau+\sigma(w(s)-w(t))+\beta(b(s)-b(t)),\label{eq:6-500}\\
&-dZ_{xmt}(s)=\left(l_{x}(X_{xmt}(s),u_{xmt}(s))+D\dfrac{d}{d\nu}F((X_{\cdot mt}(s)\otimes m)^{\mathcal{B}_{t}^{s}})(X_{xmt}(s))\right)ds\\
&\qquad \qquad\qquad-\sum_{j=1}^{n}r_{xmt}^{j}(s)dw_{j}(s)-\sum_{j=1}^{n}\rho_{xmt}^{j}(s)db_{j}(s), \\
&Z_{xmt}(T)=h_{x}(X_{xmt}(T))+D\dfrac{d}{d\nu}F_{T}((X_{\cdot mt}(T)\otimes m)^{\mathcal{B}_{t}^{T}})(X_{xmt}(T)),\\
&l_{v}(X_{xmt}(s),u_{xmt}(s))+Z_{xmt}(s)=0.
\end{aligned}\right.
\end{equation}

This is the system of necessary and sufficient conditions of optimality
of the control problem
\begin{numcases}{}
X_{xt}(s) =x+\int_{t}^{s}v_{xt}(\tau)d\tau+\sigma(w(s)-w(t))+\beta(b(s)-b(t)),\label{eq:4-7}\\
J_{mt}(v_{\cdot, t}(\cdot))=\int_{t}^{T}\mathbb{E}\left(\int_{\mathbb{R}^{n}}l(X_{xt}(s),v_{xt}(s))dm(x)\right)\,ds+\int_{t}^{T}\mathbb{E}\left(F((X_{\cdot, t}(s)\otimes m)^{\mathcal{B}_{t}^{s}})\right)ds\nonumber\\
\qquad\qquad\qquad \, + \, \mathbb{E}\left(\int_{\mathbb{R}^{n}}h(X_{xt}(T))dm(x)\right)+\mathbb{E}\left(F_{T}((X_{\cdot, t}(T)\otimes m)^{\mathcal{B}_{t}^{T}})\right), \label{eq:4-8}
\end{numcases}

with $v_{\cdot, t}(\cdot)$ $\in L_{\mathcal{F}_{t}}^{2}(t,T;\mathcal{H}_{m}).$
The optimal control is $u_{xmt}(s)$ and the optimal trajectory is
$X_{xmt}(s).$

We can then check the estimates (under the same conditions as the
general problem, see Proposition \ref{prop3-1}):
\begin{align}
&\mathbb{E}\left(|u_{xmt}(s)|^{2}\right),\ \mathbb{E}\left(|X_{xmt}(s)|^{2}\right),\ \mathbb{E}\left(|Z_{xmt}(s)|^{2}\right)\leq C_{T}(1+|x|^{2}),\label{eq:4-90}\\
&\sum_{j=1}^{n}\int_{t}^{T}|r_{xmt}^{j}(s)|^{2}ds,\ \sum_{j=1}^{n}\int_{t}^{T}|\rho_{xmt}^{j}(s)|^{2}ds\leq C_{T}(1+|x|^{2}).\nonumber
\end{align}
This means that the optimal control $u_{xmt}(s)$ belongs to the space of processes $v_{xt}(\cdot)$ such that $x\mapsto\frac{v_{xt}(\cdot)}{\left(1+|x|^{2}\right)^{\frac{1}{2}}}\in L^{\infty}(\mathbb{R}^{n};L_{\mathcal{F}_{t}}^{\infty}(t,T;L^{2}(\Omega,\mathcal{A},\mathbb{P};\mathbb{R}^{n})))$, which is a subspace of $L_{\mathcal{F}_{t}}^{2}(t,T;\mathcal{H}_{m})$
for any $m.$ We denote $V(m,t)=V(J\otimes m,t)$, given by the formula 
\begin{align}
V(m,t)=&\int_{t}^{T}\mathbb{E}\left(\int_{\mathbb{R}^{n}}l(X_{xmt}(s),u_{xmt}(s))dm(x)\right)\,ds+\int_{t}^{T}\mathbb{E}\left(F((X_{\cdot mt}(s)\otimes m)^{\mathcal{B}_{t}^{s}})\right)ds\label{eq:4-91}\\
&+\mathbb{E}\left(\int_{\mathbb{R}^{n}}h(X_{xmt}(T))dm(x)\right)+\mathbb{E}\left(F_{T}((X_{\cdot mt}(T)\otimes m)^{\mathcal{B}_{t}^{T}})\right), \nonumber
\end{align}
and 
\begin{equation}
D_{X}V(m,t)=Z_{xmt}(t).\label{eq:4-92}
\end{equation}
There is another way to see the system \eqref{eq:6-500}.
We use the lifting technique and the formulation of the distance
of Wasserstein metric given by (\ref{eq:1-2}). To the measures $m$ and $m'$, we
associate random variables $\widehat{X}_{m}$
and $\widehat{X}_{m'}$, independent of $\mathcal{F}_{t}$, such that $\mathcal{L}_{\widehat{X}_{m}}=m$, $\mathcal{L}_{\widehat{X}_{m'}}=m'$, and 
\begin{equation}
W_{2}^{2}(m,m')=\mathbb{E}\left(|\widehat{X}_{m}-\widehat{X}_{m'}|^{2}\right).\label{eq:4-93}
\end{equation}
We then consider the analog of \eqref{eq:6-500} as follows:
\begin{equation}
\label{eq:6-503}
    \left\{\begin{aligned}
    &X_{\widehat{X}_{m}t}(s)=\widehat{X}_{m}+\int_{t}^{s}u_{\widehat{X}_{m}t}(\tau)d\tau+\sigma(w(s)-w(t))+\beta(b(s)-b(t)),\\
    &-dZ_{\widehat{X}_{m}t}(s)=\left(l_{x}(X_{\widehat{X}_{m}t}(s),u_{\widehat{X}_{m}t}(s))+D\dfrac{d}{d\nu}F((\mathcal{L}_{X_{\widehat{X}_{m}t}(s)})^{\mathcal{B}_{t}^{s}})(X_{\widehat{X}_{m}t}(s))\right)ds\\
    &\qquad\qquad\qquad\quad -\sum_{j=1}^{n}r_{\widehat{X}_{m}t}^{j}(s)dw_{j}(s)-\sum_{j=1}^{n}\rho_{\widehat{X}_{m}t}^{j}(s)db_{j}(s),\\
    &Z_{\widehat{X}_{m}t}(T)=h_{x}(X_{\widehat{X}_{m}t}(T))+D\dfrac{d}{d\nu}F_{T}((\mathcal{L}_{X_{\widehat{X}_{m}t}(T)})^{\mathcal{B}_{t}^{T}})(X_{\widehat{X}_{m}t}(T)),\\
    &l_{v}(X_{\widehat{X}_{m}t}(s),u_{\widehat{X}_{m}t}(s))+Z_{\widehat{X}_{m}t}(s)=0.
\end{aligned}\right.
\end{equation}

This system is identical to \eqref{eq:3-7},
in which we have replaced $X_{x}$ with $\widehat{X}_{m}$. The interest
of this formulation is that 
\begin{equation}
\mathcal{L}_{X_{\widehat{X}_{m}t}(s)}=X_{\cdot mt}(s)\otimes m.\label{eq:6-5050}
\end{equation}
Therefore we can write 
\begin{equation}
V(m,t)=V(\widehat{X}_{m},t),\label{eq:6-506}
\end{equation}
with 
\begin{align}
V(\widehat{X}_{m},t)=&\int_{t}^{T}\mathbb{E}\left(l(X_{\widehat{X}_{m}t}(s),u_{\widehat{X}_{m}t}(s))\right)\,ds+\int_{t}^{T}\mathbb{E}\left(F((\mathcal{L}_{X_{\widehat{X}_{m}t}(s)})^{\mathcal{B}_{t}^{s}})\right)ds\label{eq:6-5060}\\
&+\mathbb{E}\left(h(X_{\widehat{X}_{m}t}(T))\right)+\mathbb{E}\left(F_{T}((\mathcal{L}_{X_{\widehat{X}_{m}t}(T)})^{\mathcal{B}_{t}^{T}})\right).\nonumber
\end{align}
As for (\ref{eq:4-3}) we have 
\begin{equation}
D_{X}V(\widehat{X}_{m},t)=Z_{\widehat{X}_{m}t}(t), \ \quad \mathbb{E}\left(|Z_{\widehat{X}_{m}t}(t)|^{2}\right) \leq C_{T}^{2}(1+\mathbb{E}(|\widehat{X}_{m}|^{2})).\label{eq:6-5061}
\end{equation}
Using 
\[
V(\widehat{X}_{m'},t)-V(\widehat{X}_{m},t)=\int_{0}^{1}\mathbb{E} \left(D_{X}V(\widehat{X}_{m}+\lambda(\widehat{X}_{m'}-\widehat{X}_{m}),t) \cdot (\widehat{X}_{m'}-\widehat{X}_{m})\right)d\lambda
\]
 and the estimate (\ref{eq:6-5061}), we can write:
\[
|V(\widehat{X}_{m'},t)-V(\widehat{X}_{m},t)|\leq C_{T}\left(1+\mathbb{E}\left(|\widehat{X}_{m}|^{2}\right)+\mathbb{E}\left(|\widehat{X}_{m'}-\widehat{X}_{m}|^{2}\right)\right)^{\frac{1}{2}}\left(\mathbb{E}\left(|\widehat{X}_{m}-\widehat{X}_{m'}|^{2}\right)\right)^{\frac{1}{2}},
\]
 which means 
\begin{equation}
|V(m',t)-V(m,t)|\leq C_{T}\left(1+\int_{\mathbb{R}^{n}}|x|^{2}dm(x)+W_{2}^{2}(m,m')\right)^{\frac{1}{2}}W_{2}(m,m').\label{eq:6-5062}
\end{equation}
We have also, by (\ref{eq:ApB202}), 
\begin{equation}\label{eq:6-5063}
\begin{aligned}
\mathbb{E}\left(|X_{\widehat{X}_{m'}\,t}(s)-X_{\widehat{X}_{m}t}(s)|^{2}\right) &\leq C_{T}\mathbb{E}\left(|\widehat{X}_{m}-\widehat{X}_{m'}|^{2}\right), \\
\mathbb{E}\left(|u_{\widehat{X}_{m'}\,t}(s)-u_{\widehat{X}_{m}t}(s)|^{2}\right) &\leq C_{T}\mathbb{E}\left(|\widehat{X}_{m}-\widehat{X}_{m'}|^{2}\right).
\end{aligned}
\end{equation}

\subsection{FUNCTIONAL DERIVATIVE}

We can then state the following:
\begin{proposition}
\label{prop4-3} We assume (\ref{eq:4-1000}), (\ref{eq:4-1001}), (\ref{eq:4-1002}), (\ref{eq:4-1003}), (\ref{eq:4-1004}), (\ref{eq:3-1}), (\ref{eq:3-2}), (\ref{eq:3-3}), (\ref{eq:3-4}), (\ref{eq:3-50}), (\ref{eq:3-51}), (\ref{eq:3-52}). We
also assume (\ref{eq:3-6}). Then the value function $V(m,t)$ has
a functional derivative $\dfrac{d}{d\nu}V(m,t)(x),$ given by: 
\begin{align}
\dfrac{d}{d\nu}V(m,t)(x)=&\int_{t}^{T}\mathbb{E}\bigg(l(X_{xmt}(s),u_{xmt}(s))\bigg)ds+\int_{t}^{T}\mathbb{E}\left(\dfrac{dF}{d\nu}((X_{\cdot mt}(s)\otimes m)^{\mathcal{B}_{t}^{s}})(X_{xmt}(s))\right)ds\label{eq:4-10}\\
&+\mathbb{E}\bigg(h(X_{xmt}(T))\bigg)+\mathbb{E}\left(\dfrac{dF_{T}}{d\nu}((X_{\cdot mt}(T)\otimes m)^{\mathcal{B}_{t}^{T}})(X_{xmt}(T))\right),
\end{align}
and 
\begin{equation}
D\dfrac{d}{d\nu}V(m,t)(x)=Z_{xmt}(t).\label{eq:4-11}
\end{equation}
We can also write 
\begin{equation}
\dfrac{d}{d\nu}V(m,t)(x)=\int_{0}^{1}Z_{\theta x,mt}(t) \cdot xd\theta+C.\label{eq:4-110}
\end{equation}
\end{proposition}

The proof can be found in Appendix B. 

More generally, we can write 
\begin{align}
Z_{xmt}(s)&=D\dfrac{dV}{d\nu}((X_{\cdot mt}(s)\otimes m)^{\mathcal{B}_{t}^{s}})(X_{xmt}(s)) =D_{X}\mathbb{E}\left(V((X_{\cdot mt}(s)\otimes m)^{\mathcal{B}_{t}^{s}})\right).\label{eq:4-111}
\end{align}
 Turning to the system (\ref{eq:3-7}), we can write 
\begin{equation}
Z_{X_{\cdot}t}(s)=D_{X}\mathbb{E}\left(V((X_{X_{\cdot}t}(s)\otimes m)^{\mathcal{B}_{t}^{s}})\right).\label{eq:4-112}
\end{equation}
We can then derive that the optimal control $u_{Xt}(s)$ is obtained
by a feedback. Indeed, introduce the standard notation in optimal
control:
\begin{equation}
    \text{Lagrangian} \quad \ L(x,v,p)=l(x,v)+v\cdot p \label{eq:4-41}
\end{equation}
defined on $(\mathbb{R}^{n})^{3}$, then from assumption (\ref{eq:4-1003}),
the function $v\mapsto L(x,v,p)$ is strictly convex and $\rightarrow+\infty,$
as $|v|\rightarrow+\infty.$ So it has a unique minimum $u(x,p).$
The function $(x,p) \mapsto u(x,p$) is $C^{1}$, with formulas
\begin{equation}
u_{x}(x,p)=-(l_{vv}(x,u))^{-1}l_{vx}(x,u),\ u_{p}(x,p)=-(l_{vv}(x,u))^{-1}\label{eq:4-42}
\end{equation}
and estimates
\[
|u_{x}(x,p)|\leq\dfrac{c_{l}}{\lambda},\ |u_{p}(x,p)|\leq\dfrac{1}{\lambda}.
\]
Next, we introduce the Hamiltonian:
\begin{equation}
H(x,p)=\inf_{v}L(x,v,p)=l(x,u(x,p))+p\cdot u(x,p).\label{eq:4-43}
\end{equation}
 which is also $C^{1}$, with formulas 
\begin{equation}
H_{x}(x,p)=l_{x}(x,u(x,p)),\ H_{p}(x,p)=u(x,p).\label{eq:4-44}
\end{equation}
Therefore, we can write the feedback rule:
\begin{equation}
u_{Xt}(s)=H_{p}\left(X_{Xt}(s),D_{X}\mathbb{E}\left(V((X_{X_{\cdot}t}(s)\otimes m)^{\mathcal{B}_{t}^{s}})\right)\right).\label{eq:4-45}
\end{equation}

\section{SECOND-ORDER DIFFERENTIABILITY } \label{sec:2nd order diff}

\subsection{DERIVATIVE OF $D^{2}\dfrac{d}{d\nu}V(m,t)(x)$}
\begin{proposition}
\label{prop7-10} We make the assumptions of Proposition \ref{prop4-3}.
Then the processes $X_{xmt}(s)$, $u_{xmt}(s)$,$Z_{zmt}(s)$, $r_{xmt}^{j}(s)$, $\rho_{xmt}^{j}(s)$
are continuously differentiable in $x,$ in the sense 
\[
\int_{t}^{T}\mathbb{E} \left(\left|\dfrac{X_{x+\epsilon y,mt}(s)-X_{xmt}(s)}{\epsilon}-\mathcal{X}_{xmt}(s)y \right|^{2}\right)ds\rightarrow0
\]
 and similar definition for the other processes. The gradients denoted
$\mathcal{X}_{xmt}(s)$, $\mathcal{U}_{xmt}(s)$, $\mathcal{Z}_{zmt}(s)$, $\mathcal{R}_{xmt}^{j}(s)$, $\Theta{}_{xmt}^{j}(s)$
are the unique solution of the system:
\begin{numcases}{}
\mathcal{X}_{xmt}(s)=I+\int_{t}^{s}\mathcal{U}_{xmt}(\tau)d\tau,\label{eq:4-175}\\
l_{vx}(X_{xmt}(s),u_{xmt}(s))\mathcal{X}_{xmt}(s)+l_{vv}(X_{xmt}(s),u_{xmt}(s))\mathcal{U}_{xmt}(s)+\mathcal{Z}_{xmt}(s)=0,\label{eq:4-176}\\
-d\mathcal{Z}_{xmt}(s)=\Big[\Big(l_{xx}(X_{xmt}(s),u_{xmt}(s))+D^{2}\dfrac{dF}{d\nu}((X_{\cdot mt}(s)\otimes m)^{\mathcal{B}_{t}^{s}})(X_{xmt}(s))\Big)\mathcal{X}_{xmt}(s)\nonumber\\
\qquad\qquad\qquad+l_{xv}(X_{xmt}(s),u_{xmt}(s))\mathcal{U}_{xmt}(s)\Big]ds-\sum_{j=1}^{n}\mathcal{R}_{xmt}^{j}(s)dw_{j}(s)-\sum_{j=1}^{n}\Theta{}_{xmt}^{j}(s)db_{j}(s),\label{eq:4-177}\\
\mathcal{Z}_{xmt}(T)= \left(h_{xx}(X_{xmt}(T))+D^{2}\dfrac{dF_{T}}{d\nu}((X_{\cdot mt}(T)\otimes m)^{\mathcal{B}_{t}^{T}})(X_{xmt}(T))\right)\mathcal{X}_{xmt}(T),\nonumber\\
D^{2}\dfrac{d}{d\nu}V(m,t)(x)=\mathcal{Z}_{xmt}(t).\label{eq:4-178} 
\end{numcases}
\end{proposition}

The proof can be found in Appendix C. 

Let $y\in \mathbb{R}$. We consider a linear quadratic control problem: the
control is in $L_{\mathcal{F}_{t}}^{2}(t,T;\mathcal{H}_{m}).$ For
$\mathcal{V}_{t}(\cdot)\in L_{\mathcal{F}_{t}}^{2}(t,T;\mathcal{H}_{m})$, the state $\mathcal{X}_{yt}(\cdot)$ is defined by 
\begin{equation}
\mathcal{X}_{yt}(s) :=y+\int_{t}^{s}\mathcal{V}_{t}(\tau)d\tau, \label{eq:4-179}
\end{equation}
and the payoff is 
\begin{align}
\mathcal{J}_{xymt}(\mathcal{V}_{t}(\cdot))=&\, \dfrac{1}{2}\int_{t}^{T}\mathbb{E}\bigg[\left(l_{xx}(X_{xmt}(s),u_{xmt}(s))+D^{2}\dfrac{d}{d\nu}F((X_{xmt}(s)\otimes m)^{\mathcal{B}_{t}^{s}})(X_{xmt}(s))\right)\mathcal{X}_{yt}(s) \cdot \mathcal{X}_{yt}(s) \nonumber\\
&+2l_{xv}(X_{xmt}(s),u_{xmt}(s))\mathcal{V}_{t}(s)\cdot\mathcal{X}_{yt}(s)+l_{vv}(Y_{xmt}(s),u_{xmt}(s))\mathcal{V}_{t}(s) \cdot \mathcal{V}_{t}(s)\bigg]ds \nonumber\\
&+\dfrac{1}{2}\mathbb{E} \left[\left(h_{xx}(X_{xmt}(T))+D^{2}\dfrac{d}{d\nu}F_{T}((X_{xmt}(T)\otimes m)^{\mathcal{B}_{t}^{T}})(X_{xmt}(T))\right)\mathcal{X}_{\mathcal{X}t}(T) \cdot \mathcal{X}_{\mathcal{X}t}(T) \right]. \label{eq:4-180}
\end{align}
The optimal control is $\mathcal{U}_{xmt}(s)y$ and the optimal state
$\mathcal{X}_{xmt}(s)y.$
\begin{proposition}
\label{prop7-100} We make the assumptions of Proposition \ref{prop4-3}.
The function $(x,m,t)\mapsto D^{2}\dfrac{d}{d\nu}V(m,t)(x)$ satisfies 
\begin{equation}
\left|D^{2}\dfrac{d}{d\nu}V(m,t)(x)\right|\leq C_{T}.\label{eq:4-181}
\end{equation}
\end{proposition}
The proof can be found in Appendix C. 

\subsection{EXISTENCE OF THE SECOND DERIVATIVE OF THE VALUE FUNCTION }

We state the following important proposition:
\begin{proposition}
\label{prop5-10}We make the assumptions of Proposition \ref{prop4-3}.
The value function $V(X_{\cdot}\otimes m,t)$ has a second-order G\^{a}teaux
derivative in $X\in\mathcal{H}_{m}$ independent of $\mathcal{F}_{t}$,
denoted $D_{X}^{2}V(X\otimes m,t)\in\mathcal{L}(\mathcal{H}_{m};\mathcal{H}_{m})$
and 
\begin{equation}
||D_{X}^{2}V(X_{\cdot}\otimes m,t)(\mathcal{X}_{\cdot})||\leq C_{T}||\mathcal{X}||,\label{eq:5-2000}
\end{equation}
where $C_{T}$ is a constant not depending on $X$, $t$, and $m.$
In addition, we have the following continuity property: let $t_{k}\downarrow t$
and $X_{k},\mathcal{X}_{k}$ independent of $\mathcal{F}_{t_{k}}$
converge to $X_{\cdot},\mathcal{X}_{\cdot}$ in $\mathcal{H}_{m}$, then 
\begin{equation}
D_{X}^{2}V(X_{k}\otimes m,t_{k})(\mathcal{X}_{k})\rightarrow D_{X}^{2}V(X_{\cdot}\otimes m,t)(\mathcal{X}_{\cdot})\ \text{in}\ \mathcal{H}_{m}.\label{eq:5-2001}
\end{equation}
The limits $X,\mathcal{X}$ are independent of $\mathcal{F}_{t}.$
We can give an explicit formula for the second-order G\^ateaux derivative.
Let $u_{Xt}(s)$, $X_{Xt}(s)$, $Z_{Xt}(s)$, $r_{Xt}^{j}(s)$, $\rho_{Xt}^{j}(s)$ be
the solution of the system \eqref{eq:3-7}. We define $\mathcal{\mathcal{U}}_{X_{\cdot}\mathcal{X}_{\cdot}t}(s)$, $\mathcal{X}_{X_{\cdot}\mathcal{X}_{\cdot}t}(s)$, $\mathcal{Z}_{X_{\cdot}\mathcal{X}_{\cdot}t}(s)$, $\mathcal{R}_{X_{\cdot}\mathcal{X}_{\cdot}t}^{j}(s)$, $\Theta_{X_{\cdot}\mathcal{X}_{\cdot}t}^{j}(s)$ to be the unique solution of the system:
\begin{equation} \label{eq:5-2002}
\left\{
\begin{aligned}
&\mathcal{X}_{X_{\cdot}\mathcal{X}_{\cdot}t}(s)=\mathcal{X}_{\cdot}+\int_{t}^{s}\mathcal{U}_{X_{\cdot}\mathcal{X}_{\cdot}t}(\tau)d\tau,\\
&l_{vx}(X_{Xt}(s),u_{Xt}(s))\mathcal{X}_{X_{\cdot}\mathcal{X}_{\cdot}t}(s)+l_{vv}(X_{X_{\cdot}t}(s),u_{X_{\cdot}t}(s))\mathcal{\mathcal{U}}_{X_{\cdot}\mathcal{X}_{\cdot}t}(s)+\mathcal{Z}_{X_{\cdot}\mathcal{X}_{\cdot}t}(s)=0, \\
&-d\mathcal{Z}_{X_{\cdot}\mathcal{X}_{\cdot}t}(s)=\Big[l_{xx}(X_{X_{\cdot}t}(s),u_{X_{\cdot}t}(s))\mathcal{X}_{X_{\cdot}\mathcal{X}_{\cdot}t}(s)+l_{xv}(X_{Xt}(s),u_{Xt}(s))\mathcal{\mathcal{U}}_{X_{\cdot}\mathcal{X}_{\cdot}t}(s) \\
&\qquad\qquad+D_{X}^{2}\mathbb{E}\left(F((X_{X_{\cdot}t}(s)\otimes m)^{\mathcal{B}_{t}^{s}})\right)(\mathcal{X}_{X_{\cdot}\mathcal{X}_{\cdot}t}(s))\Big]ds-\sum_{j=1}^{n}\mathcal{R}_{X_{\cdot}\mathcal{X}_{\cdot}t}^{j}(s)dw_{j}(s)-\sum_{j=1}^{n}\Theta_{X_{\cdot}\mathcal{X}_{\cdot}t}^{j}(s)db_{j}(s), \\
&\mathcal{Z}_{X_{\cdot}\mathcal{X}_{\cdot}t}(T)=h_{xx}(X_{X_{\cdot}t}(T))\mathcal{X}_{X_{\cdot}\mathcal{X}_{\cdot}t}(T)+D_{X}^{2}\mathbb{E}\left(F_{T}((X_{X_{\cdot}t}(T)\otimes m)^{\mathcal{B}_{t}^{T}})\right)(\mathcal{X}_{X_{\cdot}\mathcal{X}_{\cdot}t}(T)).
\end{aligned}\right.
\end{equation}

The random variables $\mathcal{\mathcal{U}}_{X_{\cdot}\mathcal{X}_{\cdot}t}(s)$, $\mathcal{X}_{X_{\cdot}\mathcal{X}_{\cdot}t}(s)$, $\mathcal{Z}_{X_{\cdot}\mathcal{X}_{\cdot}t}(s)$, $\mathcal{R}_{X_{\cdot}\mathcal{X}_{\cdot}t}^{j}(s)$, $\Theta_{X_{\cdot}\mathcal{X}_{\cdot}t}^{j}(s)$ are $\mathcal{F}_{X_{\cdot}\mathcal{X}_{\cdot}t}^{s}=\sigma(X_{\cdot},\mathcal{X}_{\cdot})\cup\mathcal{F}_{t}^{s}$ measurable. 
They belong to $L_{\mathcal{F}_{X_{\cdot}\mathcal{X}_{\cdot}t}}^{2}(t,T;\mathcal{H}_{m}),$
where $\mathcal{F}_{X_{\cdot}\mathcal{X}_{\cdot}t}$ is the filtration generated
by the $\sigma$-algebras $\mathcal{F}_{X_{\cdot}\mathcal{X}_{\cdot}t}^{s}$.
Then 
\begin{equation}
D_{X}^{2}V(X_{\cdot}\otimes m,t)(\mathcal{X}_{\cdot})=\mathcal{Z}_{X_{\cdot}\mathcal{X}_{\cdot}t}(t). \label{eq:5-2005}
\end{equation}
\end{proposition}

The proof can be found in Appendix C. 

\subsection{CASE WHEN $\mathcal{X}$ IS INDEPENDENT OF $X$}
In this section, we consider the characterization of $D_{X}^{2}V(X\otimes m,t)(\mathcal{X})$
when $\mathcal{X}$ is independent of $X.$ It is assumed that $\mathcal{X}$
is independent of the filtration $\mathcal{F}_{t}$. Therefore, $\mathcal{X}$
is also independent of the filtration $\mathcal{F}_{Xt}.$ We then have the following proposition: 
\begin{proposition}
\label{prop5-100} We make the assumptions of Proposition \ref{prop5-10}
and $\mathcal{X}$ independent of the filtration $\mathcal{F}_{Xt}$,
with $\mathbb{E}(\mathcal{X})=0$. Then the system \eqref{eq:5-2002} becomes 
\begin{numcases}{}
\mathcal{X}_{X_{\cdot}\mathcal{X}_{\cdot}t}(s)=\mathcal{X}+\int_{t}^{s}\mathcal{U}_{X_{\cdot}\mathcal{X}_{\cdot}t}(\tau)d\tau,\label{eq:5-2010}\\
l_{vx}(X_{Xt}(s),u_{Xt}(s))\mathcal{X}_{X_{\cdot}\mathcal{X}_{\cdot}t}(s)+l_{vv}(X_{Xt}(s),u_{Xt}(s))\mathcal{\mathcal{U}}_{X_{\cdot}\mathcal{X}_{\cdot}t}(s)+\mathcal{Z}_{X_{\cdot}\mathcal{X}_{\cdot}t}(s)=0,\label{eq:5-2011}\\
-d\mathcal{Z}_{X_{\cdot}\mathcal{X}_{\cdot}t}(s)=\Big[\Big(l_{xx}(X_{Xt}(s),u_{Xt}(s))+D^{2}\dfrac{dF}{d\nu}((X_{Xt}(s)\otimes m)^{\mathcal{B}_{t}^{s}})(X_{Xt}(s))\Big)\mathcal{X}_{X_{\cdot}\mathcal{X}_{\cdot}t}(s)\nonumber\\
\qquad\qquad\qquad+l_{xv}(X_{Xt}(s),u_{Xt}(s))\mathcal{\mathcal{U}}_{X_{\cdot}\mathcal{X}_{\cdot}t}(s)\Big]ds -\sum_{j=1}^{n}\mathcal{R}_{X_{\cdot}\mathcal{X}_{\cdot}t}^{j}(s)dw_{j}(s)-\sum_{j=1}^{n}\Theta_{X_{\cdot}\mathcal{X}_{\cdot}t}^{j}(s)db_{j}(s),\label{eq:5-2013}\\
\mathcal{Z}_{X_{\cdot}\mathcal{X}_{\cdot}t}(T)=\left(h_{xx}(X_{Xt}(T))+D^{2}\dfrac{dF_{T}}{d\nu}((X_{Xt}(T)\otimes m)^{\mathcal{B}_{t}^{T}})(X_{Xt}(T))\right)\mathcal{X}_{X_{\cdot}\mathcal{X}_{\cdot}t}(T).
\end{numcases}


\end{proposition}

\begin{proof}
It is sufficient to show that 
\begin{equation}
D^{2}\dfrac{dF}{d\nu}((X_{Xt}(s)\otimes m)^{\mathcal{B}_{t}^{s}})\mathcal{X}_{X\mathcal{X}t}(s)=D_{X}^{2}\mathbb{E}\left(F((X_{Xt}(s)\otimes m)^{\mathcal{B}_{t}^{s}})\right)(\mathcal{X}_{X\mathcal{X}t}(s)).\label{eq:5-2014}
\end{equation}
From formula (\ref{eq:2-100}), we have 
\begin{equation}
\mathbb{E}^{1\mathcal{B}_{t}^{s}}\left(\int_{\mathbb{R}^{n}}DD_{1}\dfrac{d^{2}F}{d\nu^{2}}((X_{Xt}(s)\otimes m)^{\mathcal{B}_{t}^{s}})(X_{Xt}(s),X_{X_{\cdot}^{1}t}^{1}(s))\mathcal{X}_{X_{\cdot}^{1}\mathcal{\chi}^{1}t}^{1}(s)dm(x^{1})\right)=0\label{eq:5-2500}
\end{equation}
since $\mathcal{\chi}^{1}$ is independent of the filtration $\mathcal{F}_{X^{1}t}$
and $\mathcal{X}_{X_{\cdot}^{1}\mathcal{\chi}^{1}t}^{1}(s)$ can be written
as $\mathcal{X}_{X^{1}t}^{1}(s)(\mathcal{\chi}^{1})$, where $\mathcal{X}_{X^{1}t}^{1}(s)(\cdot)\in L_{\mathcal{F}_{X^{1}t}}^{2}(t,T;\mathcal{L}(\mathcal{H}_{m};\mathcal{H}_{m})).$
So $\mathcal{\chi}^{1}$ is independent of $X_{X_{\cdot}^{1}t}^{1}(\cdot)$  and $\mathcal{X}_{X^{1}t}^{1}(\cdot)(\cdot).$ Therefore, the left hand side
of (\ref{eq:5-2500}) is 
\[
\mathbb{E}^{1\mathcal{B}_{t}^{s}}\left(\int_{\mathbb{R}^{n}}DD_{1}\dfrac{d^{2}F}{d\nu^{2}}((X_{Xt}(s)\otimes m)^{\mathcal{B}_{t}^{s}})(X_{Xt}(s),X_{X_{\cdot}^{1}t}^{1}(s))\mathcal{X}_{X^{1}t}^{1}(s)(\mathbb{E}^{1}(\mathcal{\chi}^{1}))dm(x^{1})\right)=0
\]
 since $\mathbb{E}^{1}(\mathcal{\chi}^{1})=0.$  
\end{proof}

\subsection{SECOND-ORDER FUNCTIONAL DERIVATIVE OF THE VALUE FUNCTION}

We want to compute the second-order functional derivative of the value
function $\dfrac{d^{2}}{d\nu^{2}}V(m,t)(x,x^{1}).$ We know that it
is sufficient to obtain the second gradient $D_{1}D_{2}\dfrac{d^{2}}{d\nu^{2}}V(m,t)(x,x^{1}).$
We can use the formula (\ref{eq:2-3003}). If $x\mapsto y_{x}$ is
in $L_{m}^{2}(\mathbb{R}^{n};\mathbb{R}^{n})$ then we can write 
\begin{equation}
D_{X}^{2}V(m,t)(y_{\cdot})=D^{2}\dfrac{d}{d\nu}V(m,t)(x)y_{x}+\int_{\mathbb{R}^{n}}DD_{1}\dfrac{d^{2}}{d\nu^{2}}V(m,t)(x,x^{1})y_{x^{1}}dm(x^{1}).\label{eq:5-2015}
\end{equation}
Consider $\mathcal{\mathcal{U}}_{X_{\cdot}\mathcal{X}_{\cdot}t}(s)$, $\mathcal{X}_{X_{\cdot}\mathcal{X}_{\cdot}t}(s)$, $\mathcal{Z}_{X_{\cdot}\mathcal{X}_{\cdot}t}(s)$, $\mathcal{R}_{X_{\cdot}\mathcal{X}_{\cdot}t}^{j}(s)$, $\Theta_{X_{\cdot}\mathcal{X}_{\cdot}t}^{j}(s)$
with $X_{\cdot}=J_{\cdot},\mathcal{X}_{\cdot}=y_{\cdot}.$ Since they represent the
value of random linear operators from $L_{m}^{2}(\mathbb{R}^{n})$ into itself, we shall represent them with kernels denoted $u_{mt}(x,x^{1},s)$, $X_{mt}(x,x^{1},s)$, $Z_{mt}(x,x^{1},s)$, $r_{mt}^{j}(x,x^{1},s)$, $\rho_{mt}^{j}(x,x^{1},s).$
In particular,
\begin{equation}
(D_{X}^{2}V(m,t)(y_{\cdot}))_{x}=\int_{\mathbb{R}^{n}}Z_{mt}(x,x^{1},t)y_{x^{1}}dm(x^{1}).\label{eq:5-2016}
\end{equation}
We need to introduce the identity operator on $L_{m}^{2}(\mathbb{R}^{n})$,
represented formally by the kernel $I_{m}(x,x^{1}),$ so that 
\[
\int_{\mathbb{R}^{n}}I_{m}(x,x^{1})y_{x^{1}}dm(x^{1})=y_{x}.
\]
Formally $I_{m}(x,x^{1})=\delta(x-x^{1}).$

The system \eqref{eq:5-2002}
yields 
\begin{numcases}{}
 X_{mt}(x,x^{1},s)=I_{m}(x,x^{1})+\int_{t}^{s}u_{mt}(x,x^{1},\tau)d\tau,\label{eq:5-2017}\\
l_{vx}(X_{xmt}(s),u_{xmt}(s))X_{mt}(x,x^{1},s)+l_{vv}(X_{xmt}(s),u_{xmt}(s))u_{mt}(x,x^{1},s)+Z_{mt}(x,x^{1},s)=0,\label{eq:5-2018}\\
-dZ_{mt}(x,x^{1},s)=\Bigg[l_{xx}(X_{xmt}(s),u_{xmt}(s))X_{mt}(x,x^{1},s)+l_{xv}(X_{xmt}(s),u_{xmt}(s))u_{mt}(x,x^{1},s)\label{eq:5-2019}\\
\qquad\qquad\qquad\qquad \quad +D^{2}\dfrac{dF}{d\nu}((X_{\cdot mt}(s)\otimes m)^{\mathcal{B}_{t}^{s}})(X_{xmt}(s))X_{mt}(x,x^{1},s)\nonumber\\
\qquad\qquad\qquad\qquad \quad +\mathbb{E}^{1\mathcal{B}_{t}^{s}} \left(\int_{\mathbb{R}^{n}}DD_{1}\dfrac{d^{2}F}{d\nu^{2}}((X_{\cdot mt}(s)\otimes m)^{\mathcal{B}_{t}^{s}})(X_{xmt}(s),X_{\eta mt}^{1}(s))X_{mt}^{1}(\eta,x^{1},s)dm(\eta)\right)\Bigg]ds\nonumber\\
\qquad\qquad\qquad\qquad -\sum_{j=1}^{n}r_{mt}^{j}(x,x^{1},s)dw_{j}(s)-\sum_{j=1}^{n}\rho_{mt}^{j}(x,x^{1},s)db_{j}(s\nonumber),\\
 Z_{mt}(x,x^{1},T)=h_{xx}(X_{xmt}(T),u_{xmt}(T))X_{mt}(x,x^{1},T)\nonumber\\
\qquad\qquad\qquad\quad+D^{2}\dfrac{dF}{d\nu}((X_{\cdot mt}(T)\otimes m)^{\mathcal{B}_{t}^{T}})(X_{xmt}(T))X_{mt}(x,x^{1},T)\nonumber \\
\qquad\qquad\qquad\quad+\mathbb{E}^{1\mathcal{B}_{t}^{T}}\left(\int_{\mathbb{R}^{n}}DD_{1}\dfrac{d^{2}F_{T}}{d\nu^{2}}((X_{\cdot mt}(T)\otimes m)^{\mathcal{B}_{t}^{T}})(X_{xmt}(T),X_{\eta mt}^{1}(T))X_{mt}^{1}(\eta,x^{1},T)dm(\eta)\right).\nonumber  
\end{numcases}





From (\ref{eq:5-2005}), (\ref{eq:5-2015}), and (\ref{eq:5-2016}), we can write
\begin{align}
    \int_{\mathbb{R}^{n}}Z_{mt}(x,x^{1},t)y_{x^{1}}dm(x^{1})&=D^{2}\dfrac{d}{d\nu}V(m,t)(x)y_{x}+\int_{\mathbb{R}^{n}}DD_{1}\dfrac{d^{2}}{d\nu^{2}}V(m,t)(x,x^{1})y_{x^{1}}dm(x^{1}) \nonumber\\
    &=\mathcal{Z}_{xmt}(t)y_{x}+\int_{\mathbb{R}^{n}}DD_{1}\dfrac{d^{2}}{d\nu^{2}}V(m,t)(x,x^{1})y_{x^{1}}dm(x^{1}).\label{eq:5-2020}
\end{align}

So we obtain the formula:
\begin{equation}
Z_{mt}(x,x^{1},t)=\mathcal{Z}_{xmt}(t)I_{m}(x,x^{1})+DD_{1}\dfrac{d^{2}}{d\nu^{2}}V(m,t)(x,x^{1}).\label{eq:5-2021}
\end{equation}
We introduce, more generally,
\begin{equation}\label{eq:5-2022}
\begin{aligned}
\overline{X}_{mt}(x,x^{1},s) &:=X_{mt}(x,x^{1},s)-\mathcal{X}_{xmt}(s)I_{m}(x,x^{1}), \\
\overline{u}_{mt}(x,x^{1},s) &:=u_{mt}(x,x^{1},s)-\mathcal{U}_{xmt}(s)I_{m}(x,x^{1}), \\
\overline{Z}_{mt}(x,x^{1},s) &:=Z_{mt}(x,x^{1},s)-\mathcal{Z}_{xmt}(s)I_{m}(x,x^{1}), \\
\overline{r}_{mt}^{j}(x,x^{1},s) &:=r_{mt}^{j}(x,x^{1},s)-\mathcal{R}_{xmt}^{j}(s)I_{m}(x,x^{1}), \\
\overline{\rho}_{mt}^{j}(x,x^{1},s) &:=\rho_{mt}^{j}(x,x^{1},s)-\Theta_{xmt}^{j}(s)I_{m}(x,x^{1}).
\end{aligned}
\end{equation}
Combining (\ref{eq:5-2017}), (\ref{eq:5-2018}), (\ref{eq:5-2019})
with (\ref{eq:4-175}), (\ref{eq:4-176}), (\ref{eq:4-177}) we obtain the system:
\begin{numcases}{}
    \overline{X}_{mt}(x,x^{1},s)=\int_{t}^{s}\overline{u}_{mt}(x,x^{1},\tau)d\tau,\label{eq:5-2023}\\
     l_{vx}(X_{xmt}(s),u_{xmt}(s))\overline{X}_{mt}(x,x^{1},s)+l_{vv}(X_{xmt}(s),u_{xmt}(s))\overline{u}_{mt}(x,x^{1},s)+\overline{Z}_{mt}(x,x^{1},s)=0,\label{eq:5-2024}\\
     -d\overline{Z}_{mt}(x,x^{1},s)=\Bigg[l_{xx}(X_{xmt}(s),u_{xmt}(s))\overline{X}_{mt}(x,x^{1},s)+l_{xv}(X_{xmt}(s),u_{xmt}(s))\overline{u}_{mt}(x,x^{1},s)\label{eq:5-2025}\\
     \qquad\qquad\qquad\qquad \quad +D^{2}\dfrac{dF}{d\nu}((X_{\cdot mt}(s)\otimes m)^{\mathcal{B}_{t}^{s}})(X_{xmt}(s))\overline{X}_{mt}(x,x^{1},s)\nonumber\\
\qquad\qquad\qquad\qquad \quad +\mathbb{E}^{1\mathcal{B}_{t}^{s}}\left(\int_{\mathbb{R}^{n}}DD_{1}\dfrac{d^{2}F}{d\nu^{2}}((X_{\cdot mt}(s)\otimes m)^{\mathcal{B}_{t}^{s}})(X_{xmt}(s),X_{\eta mt}^{1}(s))\overline{X}_{mt}^{1}(\eta,x^{1},s)dm(\eta)\right)\nonumber\\
\qquad\qquad\qquad\qquad \quad +\mathbb{E}^{1\mathcal{B}_{t}^{s}}\left(DD_{1}\dfrac{d^{2}F}{d\nu^{2}}((X_{\cdot mt}(s)\otimes m)^{\mathcal{B}_{t}^{s}})(X_{xmt}(s),X_{x^{1}mt}^{1}(s))\right)\Bigg]ds\nonumber\\
     \qquad\qquad\qquad\qquad-\sum_{j=1}^{n}\overline{r}_{mt}^{j}(x,x^{1},s)dw_{j}(s)-\sum_{j=1}^{n}\overline{\rho}_{mt}^{j}(x,x^{1},s)db_{j}(s),\nonumber\\
     \overline{Z}_{mt}(x,x^{1},T)=h_{xx}(X_{xmt}(T),u_{xmt}(T))\overline{X}_{mt}(x,x^{1},T)\nonumber\\
    \qquad\qquad\qquad\quad +D^{2}\dfrac{dF_{T}}{d\nu}((X_{\cdot mt}(T)\otimes m)^{\mathcal{B}_{t}^{T}})(X_{xmt}(T))\overline{X}_{mt}(x,x^{1},T)\nonumber\\
\qquad\qquad\qquad\quad+\mathbb{E}^{1\mathcal{B}_{t}^{T}}\left(\int_{\mathbb{R}^{n}}DD_{1}\dfrac{d^{2}F_{T}}{d\nu^{2}}((X_{\cdot mt}(T)\otimes m)^{\mathcal{B}_{t}^{T}})(X_{xmt}(T),X_{\eta mt}^{1}(T))\overline{X}_{mt}^{1}(\eta,x^{1},T)dm(\eta)\right)\nonumber\\
\qquad\qquad\qquad\quad+\mathbb{E}^{1\mathcal{B}_{t}^{T}}\left(DD_{1}\dfrac{d^{2}F_{T}}{d\nu^{2}}((X_{\cdot mt}(T)\otimes m)^{\mathcal{B}_{t}^{T}})(X_{xmt}(T),X_{x^{1}mt}^{1}(T))\right),\label{eq:5-2026}
\end{numcases}
and we have 
\begin{equation}
\overline{Z}_{mt}(x,x^{1},t)=D_{1}D_{2}\dfrac{d^{2}}{d\nu^{2}}V(m,t)(x,x^{1}).\label{eq:5-2028}
\end{equation}
Using the estimate (\ref{eq:5-2000}) and the relation (\ref{eq:5-2015})
we can state
\begin{equation}
\left|\int_{\mathbb{R}^{n}}\int_{\mathbb{R}^{n}}DD_{1}\dfrac{d^{2}}{d\nu^{2}}V(m,t)(x,x^{1})y_{x} \cdot y_{x^{1}}dm(x)dm(x^{1})\right|\leq C_{T}\int_{\mathbb{R}^{n}}|y_{x}|^{2}dm(x).\label{eq:5-027}
\end{equation}
Since $y_{x}$ is arbitrary, we can state 
\begin{equation}
\left|DD_{1}\dfrac{d^{2}}{d\nu^{2}}V(m,t)(x,x^{1})\right|\leq C_{T}.\label{eq:5-029}
\end{equation}
A rigorous proof can be obtained by using the system (\ref{eq:5-2023}),
(\ref{eq:5-2024}), (\ref{eq:5-2025}), (\ref{eq:5-2026}), and proceeding
as in the proof of Proposition \ref{prop7-100}. As a consequence
the function $(x,m) \mapsto D\dfrac{d}{d\nu}V(m,t)$ is continuous.
From (\ref{eq:4-27}) we also know that the function $t \mapsto D\dfrac{d}{d\nu}V(m,t)$
is continuous. 

\subsection{BELLMAN EQUATION }

We can now state the following:
\begin{theorem}
\label{theo5-10} We make the assumptions of Proposition \ref{prop4-3}, the value function $V(X\otimes m,t)$ is solution of the Bellman equation:
\begin{equation} \label{eq:5-2008}
\left\{
\begin{aligned}
&\dfrac{\partial V}{\partial t}(X\otimes m,t)+\mathbb{E}\left(\int_{\mathbb{R}^{n}}H(X_{\cdot},D_{X}V(X_{\cdot}\otimes m,t))\,dm(x)\right)+F(X_{\cdot}\otimes m)\\
&\qquad+\dfrac{1}{2}\sum_{j=1}^{n} \left\langle D_{X}^{2}V(X_{\cdot}\otimes m,t)(\sigma N_{t}),\sigma N_{t}\right\rangle +\dfrac{\beta^{2}}{2}\sum_{j=1}^{n}\left\langle D_{X}^{2}V(X_{\cdot}\otimes m,t)(e^{j}),e^{j}\right \rangle =0,\ \text{a.e. }\ t;\\
&V(X_{\cdot}\otimes m,T)=\mathbb{E}\left(\int_{\mathbb{R}^{n}}h(X_{x})dm(x)\right)+F_{T}(X_{\cdot}\otimes m),
\end{aligned}\right.
\end{equation}
where $N_{t}$ is a standard Gaussian variable in $\mathbb{R}^{n}$ independent
of $X_{\cdot}$. Among functions which satisfy the regularity properties
of the value function, it is the only one solution. 
\end{theorem}

The proof can be found in Appendix C. 

\subsection{BELLMAN EQUATION AT $X=J$ }

When $X_{\cdot}=J_{\cdot}$, we have 
\begin{align}
    D_{X}V(m,t)&=D\dfrac{d}{d\nu}V(m,t)(x),\nonumber\\
    D_{X}^{2}V(m,t)(Y_{\cdot})&=D^{2}\dfrac{d}{d\nu}V(m,t)(x)Y_{x}+\mathbb{E}^{1}\left(\int_{\mathbb{R}^{n}}DD_{1}\dfrac{d^{2}}{d\nu^{2}}V(m,t)(x,x^{1})Y_{x^{1}}^{1}dm(x^{1})\right).\nonumber
\end{align}

Therefore, 
\begin{equation*} \hspace{-0.5cm}
\begin{aligned}
    D_{X}^{2}V(m,t)(\sigma N_{t})&=D^{2}\dfrac{d}{d\nu}V(m,t)(x)\sigma N_{t},\\
    D_{X}^{2}V(m,t)(e^{j})&=D^{2}\dfrac{d}{d\nu}V(m,t)(x)e^{j}+\int_{\mathbb{R}^{n}}DD_{1}\dfrac{d^{2}}{d\nu^{2}}V(m,t)(x,x^{1})e^{j}dm(x^{1}),\\
    \dfrac{1}{2}\left\langle D_{X}^{2}V(m,t)(\sigma N_{t}),\sigma N_{t}\right\rangle &=\dfrac{1}{2}\int_{\mathbb{R}^{n}}\text{tr}\left(D^{2}\dfrac{d}{d\nu}V(m,t)(x)\sigma\sigma^{*}\right)dm(x),\\
    \dfrac{\beta^{2}}{2}\sum_{j=1}^{n}\left\langle D_{X}^{2}V(X_{\cdot}\otimes m,t)(e^{j}),e^{j} \right\rangle &=\dfrac{\beta^{2}}{2}\left(\int_{\mathbb{R}^{n}}\Delta\dfrac{d}{d\nu}V(m,t)(x)dm(x)+\sum_{j=1}^{n}\int_{\mathbb{R}^{n}}D_{j}D_{1j}\dfrac{d^{2}}{d\nu^{2}}V(m,t)(x,x^{1})dm(x)dm(x^{1})\right).
\end{aligned}
\end{equation*}

The Bellman equation reads:
\begin{equation}
\left\{ 
\begin{aligned} 
    &\dfrac{\partial V}{\partial t}(m,t)+\int_{\mathbb{R}^{n}}H\left(x,D\dfrac{d}{d\nu}V(m,t)(x)\right)dm(x)+F(m)+\dfrac{1}{2}\int_{\mathbb{R}^{n}}\text{tr}\left(D^{2}\dfrac{d}{d\nu}V(m,t)(x)\sigma\sigma^*\right)dm(x)\label{eq:5-2009}\\
    &\qquad+\dfrac{\beta^{2}}{2} \left(\int_{\mathbb{R}^{n}}\Delta\dfrac{d}{d\nu}V(m,t)(x)dm(x)+\sum_{j=1}^{n}\int_{\mathbb{R}^{n}}D_{j}D_{1j}\dfrac{d^{2}}{d\nu^{2}}V(m,t)(x,x^{1})dm(x)dm(x^{1}) \right)=0, \\
    & V(m,T)=\int_{\mathbb{R}^{n}}h(x)dm(x)+F_{T}(m).
\end{aligned}\right.
\end{equation}



\section{THE MASTER EQUATION}

\subsection{OBTAINING THE MASTER EQUATION }

The Bellman equation (\ref{eq:5-2008}) links the value function $V(X_{\cdot}\otimes m,t)$
and its first- and second-order gradients $D_{X}V(X_{\cdot}\otimes m,t)$, $D_{X}^{2}V(X_{\cdot}\otimes m,t)$, as well as the time derivative $\dfrac{\partial}{\partial t}V(X_{\cdot}\otimes m,t).$
It is a partial differential equation defined on $\mathcal{H}_{m}\times(0,T).$ Recall that $(X_{\cdot},t)\mapsto D_{X}V(X_{\cdot}\otimes m,t)$ takes $\mathcal{H}_{m}\times(0,T)$
into itself, with $D_{X}V(X_{\cdot}\otimes m,t)$ being $\sigma(X)$-measurable.
Similarly, the mapping $(X_{\cdot},t)\mapsto D_{X}^{2}V(X_{\cdot}\otimes m,t)$ takes
$\mathcal{H}_{m}\times(0,T)$ into $\mathcal{L}(\mathcal{H}_{m},\mathcal{H}_{m})$,
where $D_{X}^{2}V(X_{\cdot}\otimes m,t)(Y_{\cdot})\in\mathcal{H}_{m}$ is $\sigma(X_{\cdot},Y_{\cdot})$-measurable. 

It is well-known that with the Bellman equation, there exists an equation
for the gradient of the value function that does not include
the value function itself. This equation is derived by taking the gradient
of the Bellman equation. Rather than being a scalar equation, it forms a system
of equations. This feature also applies to the Bellman equation (\ref{eq:5-2008}).
By taking the gradient in $X_{\cdot}$ formally we obtain: 
\begin{equation} \label{eq:8-500}
\left\{
\begin{aligned}
    &\dfrac{\partial}{\partial t}D_{X}V(X_{\cdot}\otimes m,t)+H_{x}(X_{\cdot},D_{X}V(X_{\cdot}\otimes m,t))+D_{X}^{2}V(X_{\cdot}\otimes m,t)(H_{p}(X_{\cdot},D_{X}V(X_{\cdot}\otimes m,t)))+D_{X}F(X_{\cdot}\otimes m) \\
    &\qquad \qquad +\dfrac{1}{2}D_{X}\,\left\langle D_{X}^{2}V(X_{\cdot}\otimes m,t)(\sigma N_{t}),\sigma N_{t}\right\rangle +\dfrac{\beta^{2}}{2}\sum_{j=1}^{n}D_{X}\ \left\langle D_{X}^{2}V(X_{\cdot}\otimes m,t)(e^{j}),e^{j}\right\rangle =0,\ \text{a.e. }\ t, \\
    &D_{X}V(X_{\cdot}\otimes m,T)=h_{x}(X_{x})+D_{X}F_{T}(X_{\cdot}\otimes m),
\end{aligned}\right.
\end{equation}
where $N_{t}$ is a standard Gaussian random variable independent of $X_{\cdot}$

\subsection{DEVELOPED FORMULAS}

The key function in the following formulas is 
\begin{equation}
U(x,m,t)=\dfrac{d}{d\nu}V(m,t)(x).\label{eq:8-501}
\end{equation}
We then consider the functional derivative of $U(x,m,t),$ namely 
\begin{equation}
\dfrac{d}{d\nu}U(x,m,t)(x^{1})=\dfrac{d^{2}}{d\nu^{2}}V(m,t)(x,x^{1}),\label{eq:8-502}
\end{equation}
and also the second-order functional derivative:
\begin{equation}
\dfrac{d^{2}}{d\nu^{2}}U(x,m,t)(x^{1},x^{2})=\dfrac{d^{3}}{d\nu^{3}}V(m,t)(x,x^{1},x^{2}).\label{eq:8-503}
\end{equation}
We shall use also the notation $D$ for the gradient in $x,$ $D_{1}$
for the gradient with respect to $x^{1}$ and $D_{2}$ for the gradient
with respect to $x^{2}.$ The variables $x,x^{1},x^{2}$ may not be
listed in that order. Let the reader not confuse $D_{2}$ with $D^{2}.$ According
to formula (\ref{eq:2-3002}) we thus have: 
\begin{align}
    \left\langle D_{X}^{2}V(X_{\cdot}\otimes m,t)(Y_{\cdot}),Y_{\cdot}\right\rangle &=\mathbb{E}\left(\int_{\mathbb{R}^{n}}D^{2}U(X_{x},X_{\cdot}\otimes m,t)Y_{x} \cdot Y_{x}dm(x)\right)\label{eq:8-504}\\
    &\quad +\mathbb{E}\left(\int_{\mathbb{R}^{n}}\mathbb{E}^{1}\left(\int_{\mathbb{R}^{n}}D\,D_{1}\dfrac{d}{d\nu}U(X_{x},X_{\cdot}\otimes m,t)(X_{x^{1}}^{1})Y_{x^{1}}^{1} \cdot Y_{x}dm(x^{1})\right)dm(x)\right),\nonumber
\end{align}

 where $X_{x^{1}}^{1},$$Y_{x^{1}}^{1}$ is an independent copy of
$X_{x},Y_{x}$.

We first take $Y_{\cdot}=\sigma N_{t}$, to obtain, since $N_{t}$ is independent
of $X_{\cdot}$,
\begin{equation}
\left\langle D_{X}^{2}V(X_{\cdot}\otimes m,t)(\sigma N_{t}),\sigma N_{t}\right\rangle =\mathbb{E}\left(\int_{\mathbb{R}^{n}}\text{tr}\left(D^{2}U(X_{x},X_{\cdot}\otimes m,t)\sigma\sigma^{*}\right)dm(x)\right)\label{eq:8-505}
\end{equation}

and also
\begin{align}
    \sum_{j=1}^{n} \left\langle D_{X}^{2}V(X_{\cdot}\otimes m,t)(e_{j}),e_{j}\right\rangle &=\mathbb{E}\left(\int_{\mathbb{R}^{n}}\Delta U(X_{x},X_{\cdot}\otimes m,t)dm(x)\right)\label{eq:8-510}\\
    &\quad +\mathbb{E}\left(\int_{\mathbb{R}^{n}}\mathbb{E}^{1}\left(\int_{\mathbb{R}^{n}}\sum_{j=1}^{n}D_{j}D_{1j}\dfrac{d}{d\nu}U(X_{x},X_{\cdot}\otimes m,t)(X_{x^{1}}^{1})dm(x^{1})\right)dm(x)\right).\nonumber
\end{align}


\subsection{FURTHER COMPUTATIONS}

We can then compute 
\begin{align}
    \bigg\langle D_{X} \Big(\Big\langle D_{X}^{2}V(X_{\cdot}&\otimes m,t)(\sigma N_{t}),\sigma N_{t} \Big\rangle\Big) ,Z_{\cdot} \bigg\rangle =\mathbb{E}\left(\int_{\mathbb{R}^{n}}D\,(\text{tr}(D^{2}U(X_{x},X_{\cdot}\otimes m,t)\sigma\sigma^{*})) \cdot Z_{x}dm(x)\right) \nonumber\\
&+\mathbb{E}\left(\int_{\mathbb{R}^{n}}\mathbb{E}^{1}\left(\int_{\mathbb{R}^{n}}D \left(\,\text{tr}\left(D_{1}^{2}\,\dfrac{d}{d\nu}U(X_{x^{1}}^{1},X_{\cdot}\otimes m,t)(X_{x})\sigma\sigma^{*} \right)\right) \cdot Z_{x}dm_{1}(x)\right)dm(x)\right).\label{eq:8-511}
\end{align}


Thus,
\begin{align}
    D_{X} \Big(\left\langle D_{X}^{2}V(X_{\cdot}\otimes m,t)(\sigma N_{t}),\sigma N_{t} \right\rangle\Big) &=D\,(\text{tr}(D^{2}U(X_{x},X_{\cdot}\otimes m,t)\sigma\sigma^{*}))\nonumber\\
&\quad \, +\mathbb{E}^{1}\left(\int_{\mathbb{R}^{n}}D \left(\,\text{tr}\left(D_{1}^{2}\,\dfrac{d}{d\nu}U(X_{x},X_{\cdot}\otimes m,t)(X_{x^{1}}^{1})\sigma\sigma^{*} \right)\right)dm(x^{1})\right).\label{eq:8-512}
\end{align}

Similarly,
\begin{align}
    &\Bigg\langle D_{X} \Bigg(\sum_{j=1}^{n} \Big\langle D_{X}^{2}V(X_{\cdot}\otimes m,t)(e_{j}),e_{j}\Big\rangle\Bigg),Z_{\cdot} \Bigg\rangle=\mathbb{E}\left(\int_{\mathbb{R}^{n}}D(\Delta U(X_{x},X_{\cdot}\otimes m,t)) \cdot Z_{x}dm(x)\right) \nonumber\\ &\qquad+\mathbb{E}\left(\int_{\mathbb{R}^{n}}\mathbb{E}^{1}\left(\int_{\mathbb{R}^{n}}D \left(\Delta_{1}\dfrac{d}{d\nu}U(X_{x},X_{\cdot}\otimes m,t)(X_{x^{1}}^{1}) \right) \cdot Z_{x}dm(x^{1})\right)dm(x)\right)\nonumber\\
&\qquad+2\mathbb{E}\left(\int_{\mathbb{R}^{n}}\mathbb{E}^{1}\left(\int_{\mathbb{R}^{n}}D\, \left(\sum_{j=1}^{n}D_{j}D_{1j}\dfrac{d}{d\nu}U(X_{x},X_{\cdot}\otimes m,t)(X_{x^{1}}^{1}) \right) \cdot Z_{x}dm(x^{1})\right)dm(x)\right)\nonumber\\
&\qquad+\mathbb{E}\left(\int_{\mathbb{R}^{n}}\mathbb{E}^{1}\left(\int_{\mathbb{R}^{n}}\mathbb{E}^{2}\left(\int_{\mathbb{R}^{n}}D\,\left(\sum_{j=1}^{n}D_{1j}D_{2j}\dfrac{d^{2}}{d\nu^{2}}U(X_{x^{1}}^{1},X_{\cdot}\otimes m,t)(X_{x^{2}}^{2},X_{x}) \right) \cdot Z_{x}dm(x^{2})\right)dm(x^{1})\right)dm(x)\right). \label{eq:8-513}
\end{align}



Therefore, 
\begin{align}
    &D_{X} \Bigg(\sum_{j=1}^{n}\left\langle D_{X}^{2}V(X_{\cdot}\otimes m,t)(e_{j}),e_{j}\right\rangle \Bigg) =D(\Delta U(X_{x},X_{\cdot}\otimes m,t)) \nonumber\\
    &\qquad+\mathbb{E}^{1}\left(\int_{\mathbb{R}^{n}}D \left(\Delta_{1}\dfrac{d}{d\nu}U(X_{x},X_{\cdot}\otimes m,t)(X_{x^{1}}^{1}) \right) dm(x^{1})\right)\nonumber\\
    &\qquad+2\mathbb{E}^{1} \left(\int_{\mathbb{R}^{n}}D\, \left(\sum_{j=1}^{n}D_{j}D_{1j}\dfrac{d}{d\nu}U(X_{x},X_{\cdot}\otimes m,t)(X_{x^{1}}^{1}) \right)dm(x^{1})\right)\nonumber\\
&\qquad+\mathbb{E}^{1}\left(\int_{\mathbb{R}^{n}}\mathbb{E}^{2} \left(\int_{\mathbb{R}^{n}}D\, \left(\sum_{j=1}^{n}D_{1j}D_{2j}\dfrac{d^{2}}{d\nu^{2}}U(X_{x^{1}}^{1},X_{\cdot}\otimes m,t)(X_{x^{2}}^{2},X_{x}) \right)dm(x^{2})\right)dm(x^{1})\right).\label{eq:8-514}
\end{align}




\subsection{MASTER EQUATION}

Going back to (\ref{eq:8-500}) we first note, with formula (\ref{eq:8-504}),
\begin{align}
    &D_{X}^{2}V(X_{\cdot}\otimes m,t)(H_{p}(X_{\cdot},D_{X}V(X_{\cdot}\otimes m,t)))=D^{2}U(X_{x},X_{\cdot}\otimes m,t)H_{p}(X_{x},DU(X_{x},X_{\cdot}\otimes m,t)) \nonumber\\
    &\quad +\mathbb{E}^{1}\left(\int_{\mathbb{R}^{n}}D\, \left(D_{1}\dfrac{d}{d\nu}U(X_{x},X_{\cdot}\otimes m,t)(X_{x^{1}}^{1})\right)H_{p}(X_{x^{1}}^{1},DU(X_{x^{1}}^{1},X_{\cdot}\otimes m,t))dm(x^{1})\right).\label{eq:8-515}
\end{align}


 Collecting results, the master equation (\ref{eq:8-500}) yields: 
\begin{equation} \label{eq:8-516}
\left\{
\begin{aligned}
    &\dfrac{\partial}{\partial t}DU(X_{x},X_{\cdot}\otimes m,t)+DH(X_{x},DU(X_{x},X_{\cdot}\otimes m,t))\\
    &+D^{2}U(X_{x},X_{\cdot}\otimes m,t)H_{p}(X_{x},DU(X_{x},X_{\cdot}\otimes m,t))+D\dfrac{dF}{d\nu}(X_{\cdot}\otimes m)(X_{x})\\
    &+\mathbb{E}^{1}\left(\int_{\mathbb{R}^{n}}D\, \left(D_{1}\dfrac{d}{d\nu}U(X_{x},X_{\cdot}\otimes m,t)(X_{x^{1}}^{1})\right)H_{p}(X_{x^{1}}^{1},DU(X_{x^{1}}^{1},X_{\cdot}\otimes m,t))dm(x^{1})\right)\\
    &+\dfrac{1}{2}D\,(\text{tr}(D^{2}U(X_{x},X_{\cdot}\otimes m,t)\sigma\sigma^{*}))+\dfrac{1}{2}\mathbb{E}^{1}\left(\int_{\mathbb{R}^{n}}D \left(\,\text{tr}\left(D_{1}^{2}\,\dfrac{d}{d\nu}U(X_{x},X_{\cdot}\otimes m,t)(X_{x^{1}}^{1})\sigma\sigma^{*}\right) \right)dm(x^{1})\right)\\
    &+\dfrac{\beta^{2}}{2}D(\Delta U(X_{x},X_{\cdot}\otimes m,t))+\dfrac{\beta^{2}}{2}\mathbb{E}^{1}\left(\int_{\mathbb{R}^{n}}D \left(\Delta_{1}\dfrac{d}{d\nu}U(X_{x},X_{\cdot}\otimes m,t)(X_{x^{1}}^{1})\right)dm(x^{1})\right)\\
    &+\beta^{2}\mathbb{E}^{1}\left(\int_{\mathbb{R}^{n}}D\,\left(\sum_{j=1}^{n}D_{j}D_{1j}\dfrac{d}{d\nu}U(X_{x},X_{\cdot}\otimes m,t)(X_{x^{1}}^{1})\right)dm(x^{1})\right)\\
    &+\dfrac{\beta^{2}}{2}\mathbb{E}^{1}\left(\int_{\mathbb{R}^{n}}\mathbb{E}^{2}\left(\int_{\mathbb{R}^{n}}D\,\left(\sum_{j=1}^{n}D_{1j}D_{2j}\dfrac{d^{2}}{d\nu^{2}}U(X_{x^{1}}^{1},X_{\cdot}\otimes m,t)(X_{x^{2}}^{2},X_{x})\right)dm(x^{2})\right)dm(x^{1})\right)=0,\\
    &DU(X_{x},X_{\cdot}\otimes m,T)=Dh(X_{x})+D\dfrac{dF_{T}}{d\nu}(X_{\cdot}\otimes m)(X_{x}).
\end{aligned}\right.
\end{equation}









\subsection{THE CASE $X_{\cdot}=J_{\cdot}$}

In the case $X_{\cdot}=J_{\cdot}$, we obtain: 
\begin{equation} \label{eq:8-517}
\left\{
\begin{aligned}
    &\dfrac{\partial}{\partial t}DU(x,m,t)+DH(x,DU(x,m,t))+D^{2}U(x,m,t)H_{p}(x,DU(x,m,t))+D\dfrac{dF}{d\nu}(m)(x)\\
    &\qquad\qquad+\int_{\mathbb{R}^{n}}D\, \left(D_{1}\dfrac{d}{d\nu}U(x,m,t)(x^{1}) \right) H_{p}(x^{1},DU(x^{1},m,t))dm(x^{1})\\
    &\qquad\qquad+\dfrac{1}{2}D\,(\text{tr}(D^{2}U(x,m,t)\sigma\sigma^{*}))+\dfrac{1}{2}\int_{\mathbb{R}^{n}}D\left(\,\text{tr}\left(D_{1}^{2}\,\dfrac{d}{d\nu}U(x,m,t)(x^{1})\sigma\sigma^{*}\right)\right)dm(x^{1})\\
    &\qquad\qquad+\dfrac{\beta^{2}}{2}D(\Delta U(x,m,t))+\dfrac{\beta^{2}}{2}\int_{\mathbb{R}^{n}}D \left(\Delta_{1}\dfrac{d}{d\nu}U(x,m,t)(x^{1})\right)dm(x^{1})\\
    &\qquad\qquad+\beta^{2}\int_{\mathbb{R}^{n}}D\, \left(\sum_{j=1}^{n}D_{j}D_{1j}\dfrac{d}{d\nu}U(x,m,t)(x^{1})\right)dm(x^{1})\\
    &\qquad\qquad+\dfrac{\beta^{2}}{2}\int_{\mathbb{R}^{n}}\int_{\mathbb{R}^{n}}D\, \left(\sum_{j=1}^{n}D_{1j}D_{2j}\dfrac{d^{2}}{d\nu^{2}}U(x^{1},m,t)(x^{2},x)\right)dm(x^{2})dm(x^{1})=0,\\
    &DU(x,m,T)=Dh(x)+D\dfrac{dF_{T}}{d\nu}(m)(x).
\end{aligned}\right.
\end{equation}








We recognize that (\ref{eq:8-517}) is the gradient of a scalar equation, which is the equation for $U(x,m,t)$, i.e.
\begin{equation} \label{eq:8-518}
\left\{
\begin{aligned}
    &\dfrac{\partial}{\partial t}U(x,m,t)+H(x,DU(x,m,t))+\dfrac{dF}{d\nu}(m)(x)\\
    &\qquad\qquad+\int_{\mathbb{R}^{n}}D_{1}\dfrac{d}{d\nu}U(x,m,t)(x^{1}) H_{p}(x^{1},DU(x^{1},m,t))dm(x^{1})\\
    &\qquad\qquad+\dfrac{1}{2}\text{tr}(D^{2}U(x,m,t)\sigma\sigma^{*})+\dfrac{1}{2}\int_{\mathbb{R}^{n}}\left(\,\text{tr}\left(D_{1}^{2}\,\dfrac{d}{d\nu}U(x,m,t)(x^{1})\sigma\sigma^{*}\right)\right)dm(x^{1})\\
    &\qquad\qquad+\dfrac{\beta^{2}}{2}\Delta U(x,m,t)+\dfrac{\beta^{2}}{2}\int_{\mathbb{R}^{n}}\Delta_{1}\dfrac{d}{d\nu}U(x,m,t)(x^{1})dm(x^{1})\\
    &\qquad\qquad+\beta^{2}\int_{\mathbb{R}^{n}} \left(\sum_{j=1}^{n}D_{j}D_{1j}\dfrac{d}{d\nu}U(x,m,t)(x^{1})\right) dm(x^{1})\\
    &\qquad\qquad+\dfrac{\beta^{2}}{2}\int_{\mathbb{R}^{n}}\int_{\mathbb{R}^{n}} \left(\sum_{j=1}^{n}D_{1j}D_{2j}\dfrac{d^{2}}{d\nu^{2}}U(x,m,t)(x^{1},x^{2})\right)dm(x^{2})dm(x^{1})=0,\\
    &U(x,m,T)=h(x)+\dfrac{dF_{T}}{d\nu}(m)(x).
\end{aligned}\right.
\end{equation}






Equation (\ref{eq:8-518}) is commonly called the Master equation. 

\subsection{JUSTIFICATION }

The above calculations are justified if the functions $X_{\cdot}\mapsto \langle D_{X}^{2}V(X_{\cdot}\otimes m,t)(\sigma N_{t}),\sigma N_{t}\rangle $
and $X_{\cdot}\mapsto\sum_{j=1}^{n} \langle D_{X}^{2}V(X_{\cdot}\otimes m,t)(e_{j}),e_{j} \rangle$
have a G\textroundcap{a}teaux differential. We need additional regularity
assumptions. We assume: 
\begin{equation}
l_{xxx}(x,v),l_{xxv}(x,v),l_{xvv}(x,v),l_{vvv}(x,v)\ \text{ continuous and bounded},\label{eq:8-600}
\end{equation}
\begin{equation}
(x,m)\mapsto D^{3}\dfrac{dF}{d\nu}(m)(x),\ \text{ continuous and bounded},\label{eq:8-601}
\end{equation}
\begin{equation}
(x,m)\mapsto D^{2}D_{1}\dfrac{d^{2}F}{d\nu^{2}}(m)(x,x^{1}),DD_{1}^{2}\dfrac{d^{2}F}{d\nu^{2}}(m)(x,x^{1}),\text{ continuous and bounded},\label{eq:8-602}
\end{equation}
\begin{equation}
(x,m)\mapsto DD_{1}D_{2}\dfrac{d^{3}F}{d\nu^{3}}(m)(x,x^{1},x^{2}),\ \text{ continuous and bounded}.\label{eq:8-603}
\end{equation}

We then have: 
\begin{proposition}
\label{prop8-1} We make the assumptions \ref{theo5-10} and (\ref{eq:8-600}), (\ref{eq:8-601}), (\ref{eq:8-602}), (\ref{eq:8-603}).
Then the functions $X_{\cdot}\mapsto \langle D_{X}^{2}V(X_{\cdot}\otimes m,t)(\sigma N_{t}),\sigma N_{t} \rangle$
and $X_{\cdot}\mapsto\sum_{j=1}^{n} \langle D_{X}^{2}V(X_{\cdot}\otimes m,t)(e_{j}),e_{j} \rangle$
have a G\textroundcap{a}teaux differential.
\end{proposition}

The proof can be found in Appendix D.

\appendix

\section{PROOFS FROM SECTION \ref{sec:control problem}} \label{appendix A}

\subsection{PROOF OF LEMMA \ref{lem3-1}}

Consider a control $v_{X_{\cdot}t}(s)+\theta\tilde{v}_{X_{\cdot}t}(s)$,
the corresponding state is $X_{X_{\cdot}t}(s)+\theta\int_{t}^{s}\tilde{v}_{X_{\cdot}t}(\tau)d\tau.$
We first check that 
\begin{align}
    \dfrac{d}{d\theta}J_{X_{\cdot}t}(v_{Xt}(\cdot)+\theta\tilde{v}_{X_{\cdot}t}(\cdot)) \bigg|_{\theta=0} &=\int_{t}^{T} \langle l_{v}(X_{X_{\cdot}t}(s),v_{X_{\cdot}t}(s)),\tilde{v}_{X_{\cdot}t}(s)\rangle ds\label{eq:8-100}\\
    &\quad \,+\int_{t}^{T} \left\langle l_{x}(X_{X_{\cdot}t}(s),v_{X_{\cdot}t}(s))+D_{X}\mathbb{E}\left(F((X_{X_{\cdot}t}(s)\otimes m)^{\mathcal{B}_{t}^{^{s}}})\right),\int_{t}^{s}\widetilde{v}_{X_{\cdot}t}(\tau)d\tau \right\rangle ds\nonumber\\
    &\quad \,+ \left\langle h_{x}(X_{X_{\cdot}t}(T))+D_{X}\mathbb{E}\left( F_{T}((X_{X_{\cdot}t}(T)\otimes m)^{\mathcal{B}_{t}^{^{T}}})\right),\int_{t}^{T}\widetilde{v}_{X_{\cdot}t}(\tau)d\tau \right\rangle \,. \nonumber
\end{align}


Then 
\begin{align*}
    &\quad \, \int_{t}^{T} \left\langle l_{x}(X_{X_{\cdot}t}(s),v_{X_{\cdot}t}(s))+D_{X}\mathbb{E}\left(F((X_{X_{\cdot}t}(s)\otimes m)^{\mathcal{B}_{t}^{^{s}}})\right),\int_{t}^{s}\widetilde{v}_{X_{\cdot}t}(\tau)d\tau \right\rangle ds\\
    &\quad \,+ \left\langle h_{x}(X_{X_{\cdot}t}(T))+D_{X}\mathbb{E}\left(F_{T}((X_{X_{\cdot}t}(T)\otimes m)^{\mathcal{B}_{t}^{^{T}}})\right),\int_{t}^{s}\widetilde{v}_{X_{\cdot}t}(\tau)d\tau \right\rangle \\
    &=\int_{t}^{T}\int_{\mathbb{R}^{n}}\mathbb{E}\left[ \left\langle l_{x}(X_{\eta t}(s),v_{\eta t}(s))+D\dfrac{dF}{d\nu}((X_{\cdot, t}(s)\otimes(X_{\cdot}\otimes m))^{\mathcal{B}_{t}^{s}})(X_{\eta t}(s)), \int_{t}^{s}\widetilde{v}_{\eta t}(\tau)d\tau \right\rangle \right]d(X_{\cdot}\otimes m)(\eta)ds\\
    &\quad \, +\int_{\mathbb{R}^{n}}\mathbb{E}\left[\left\langle h_{x}(X_{\eta t}(T))+D\dfrac{dF_{T}}{d\nu}((X_{\cdot, t}(T)\otimes(X_{\cdot}\otimes m))^{\mathcal{B}_{t}^{T}})(X_{\eta t}(T)), \int_{t}^{T}\widetilde{v}_{\eta t}(\tau)d\tau\right\rangle \right]d(X_{\cdot}\otimes m)(\eta) \,.
\end{align*}



Define $\Gamma_{\eta t}$ by 
\begin{equation} \label{eq:8-101}
\begin{aligned}
    \Gamma_{\eta t}:=&\int_{t}^{T} \left(l_{x}(X_{\eta t}(s),v_{\eta t}(s))+D\dfrac{dF}{d\nu} \left((X_{\cdot, t}(s)\otimes(X_{\cdot}\otimes m))^{\mathcal{B}_{t}^{s}} \right)(X_{\eta t}(s))\right)ds\\
    &+h_{x}(X_{\eta t}(T))+D\dfrac{dF_{T}}{d\nu} \left((X_{\cdot, t}(T)\otimes(X_{\cdot}\otimes m))^{\mathcal{B}_{t}^{T}} \right)(X_{\eta t}(T)).
\end{aligned}  
\end{equation}

$\Gamma_{\eta t}$ is $\mathcal{F}_{t}^{T}$-measurable. By the martingale
representation theorem we can write 
\begin{equation}
\mathbb{E}(\Gamma_{\eta t}|\mathcal{F}_{t}^{s})=\mathbb{E}(\Gamma_{\eta t})+\sum_{j=1}^{n}\int_{t}^{s}r_{\eta t}^{j}(\tau)dw^{j}(\tau)+\sum_{j=1}^{n}\int_{t}^{s}\rho_{\eta t}^{j}(\tau)db^{j}(\tau).\label{eq:8-1000}
\end{equation}
Define 
\begin{equation}
Z_{\eta t}(s):=\mathbb{E}(\Gamma_{\eta t}|\mathcal{F}_{t}^{s})-\int_{t}^{s} \left(l_{x}(X_{\eta t}(\tau),v_{\eta t}(\tau))+D\dfrac{dF}{d\nu} \left((X_{\cdot, t}(\tau)\otimes(X_{\cdot}\otimes m))^{\mathcal{B}_{t}^{\tau}} \right)(X_{\eta t}(\tau)) \right)d\tau.\label{eq:8-1001}
\end{equation}
Then $Z_{\eta t}(s)$ is unique solution of the backward SDE:
\begin{equation} \hspace{-0.5cm}\label{eq:8-1002}
\left\{
\begin{aligned}
&-dZ_{\eta t}(s)= \left(l_{x}(X_{\eta t}(s),v_{\eta t}(s))+D\dfrac{dF}{d\nu} \left((X_{\cdot, t}(s)\otimes(X_{\cdot}\otimes m))^{\mathcal{B}_{t}^{s}} \right)(X_{\eta t}(s)) \right)ds-\sum_{j=1}^{n}r_{\eta t}^{j}(s)dw^{j}(s)-\sum_{j=1}^{n}\rho_{\eta t}^{j}(s)db^{j}(s), \\
&Z_{\eta t}(T)=h_{x}(X_{\eta t}(T))+D\dfrac{dF_{T}}{d\nu} \left((X_{\cdot, t}(T)\otimes(X_{\cdot}\otimes m))^{\mathcal{B}_{t}^{T}} \right)(X_{\eta t}(T)),
\end{aligned}\right.
\end{equation}
and 
\begin{equation}\label{eq:8-1003}
\begin{aligned}
    &\quad \int_{t}^{T}\int_{\mathbb{R}^{n}}\mathbb{E} \left[\left(l_{x}(X_{\eta t}(s),v_{\eta t}(s))+D\dfrac{dF}{d\nu}((X_{\cdot, t}(s)\otimes(X_{\cdot}\otimes m))^{\mathcal{B}_{t}^{s}})(X_{\eta t}(s)) \right) \cdot \int_{t}^{s}\widetilde{v}_{\eta t}(\tau)d\tau \right]d(X_{\cdot}\otimes m)(\eta)ds\\
    &\quad \, +\int_{\mathbb{R}^{n}}\mathbb{E}\left[\left(h_{x}(X_{\eta t}(T))+D\dfrac{dF_T}{d\nu} \left((X_{\cdot, t}(T)\otimes(X_{\cdot}\otimes m))^{\mathcal{B}_{t}^{T}} \right)(X_{\eta t}(T))\right) \cdot \int_{t}^{T}\widetilde{v}_{\eta t}(\tau)d\tau\right] d(X_{\cdot}\otimes m)(\eta)\\
    &=\int_{t}^{T}\int_{\mathbb{R}^{n}}\mathbb{E}\left(Z_{\eta t}(s) \cdot \widetilde{v}_{\eta t}(s)\right)d(X_{\cdot}\otimes m)(\eta)ds=\int_{t}^{T} \langle Z_{X_{\cdot}t}(s),\widetilde{v}_{X_{\cdot}t}(s) \rangle ds.
\end{aligned}
\end{equation}


Therefore
\[
\dfrac{d}{d\theta}J_{X_{\cdot}t}(v_{X_{\cdot}t}(\cdot)+\theta\tilde{v}_{X_{\cdot}t}(\cdot)) \bigg|_{\theta=0}=\int_{t}^{T} \langle l_{v}(X_{X_{\cdot}t}(s),v_{X_{\cdot}t}(s))+Z_{X_{\cdot}t}(s),\tilde{v}_{X_{\cdot}t}(s) \rangle ds,
\]
which implies the result (\ref{eq:3-5}).  

\subsection{PROOF OF PROPOSITION \ref{prop3-1}}

We take two controls $v_{X_{\cdot}t}^{1}$, $v_{X_{\cdot}t}^{2}$. We are
going to check that 
\begin{align}
    &\quad \, \int_{t}^{T} \bigg\langle D_{v}J_{X_{\cdot}t}(v_{X_{\cdot}t}^{1}(\cdot))(s)-D_{v}J_{X_{\cdot}t}(v_{X_{\cdot}t}^{2}(\cdot))(s),v_{X_{\cdot}t}^{1}(s)-v_{X_{\cdot}t}^{2}(s) \bigg\rangle ds\nonumber\\
    &\geq \left(\lambda-T(c'_{T}+c'_{h})-\dfrac{(c'+c'_{l})}{2}T^{2} \right)\int_{t}^{T}\left\lVert v_{X_{\cdot}t}^{1}(s)-v_{X_{\cdot}t}^{2}(s) \right\rVert^{2}ds,\label{eq:Ap1}
\end{align}
and from the assumption (\ref{eq:3-6}) the result will follow immediateley.
To simplify notation, we set $v^{1}(s)=v_{X_{\cdot}t}^{1}(s),\ v^{2}(s)=v_{X_{\cdot}t}^{2}(s)$
and 
\[
X^{1}(s)=X_{X_{\cdot}t}(s;v_{X_{\cdot}t}^{1}(\cdot)),\ X^{2}(s)=X_{X_{\cdot}t}(s;v_{X_{\cdot}t}^{2}(\cdot)),
\]
and $Z^{1}(s),Z^{2}(s)$ for the corresponding solutions of (\ref{eq:3-600}). 

From formula (\ref{eq:3-5}) we have 
\begin{align}
    &\quad \int_{t}^{T} \left\langle D_{v}J_{X_{\cdot}t}(v_{X_{\cdot}t}^{1}(\cdot))(s)-D_{v}J_{X_{\cdot}t}(v_{X_{\cdot}t}^{2}(\cdot))(s),v_{X_{\cdot}t}^{1}(s)-v_{X_{\cdot}t}^{2}(s) \right\rangle ds\nonumber\\
    &=\int_{t}^{T} \left\langle l_{v}(X^{1}(s),v^{1}(s))-l_{v}(X^{2}(s),v^{2}(s)),v^{1}(s)-v^{2}(s)\right\rangle ds+\int_{t}^{T}\left\langle Z^{1}(s)-Z^{2}(s),v^{1}(s)-v^{2}(s) \right\rangle ds.\nonumber
\end{align}

 Next, since $v^{1}(s)-v^{2}(s)=\dfrac{d}{ds}(X^{1}(s)-X^{2}(s))$
and $X^{1}(t)-X^{2}(t)=0,$ we have 
\begin{equation*} \hspace{-1.5cm}
\begin{aligned}
    &\quad \int_{t}^{T} \left\langle Z^{1}(s)-Z^{2}(s),v^{1}(s)-v^{2}(s) \right\rangle ds \\
    &= \left\langle Z^{1}(T)-Z^{2}(T),X^{1}(T)-X^{2}(T) \right\rangle \\
    &\quad \, + \int_{t}^{T} \bigg\langle l_{x}(X^{1}(s),v^{1}(s))-l_{x}(X^{2}(s),v^{2}(s))+D_{X}\mathbb{E}\left(F((X^{1}(s)\otimes m)^{\mathcal{B}_{t}^{s}})\right) \\
    &\quad\,-D_{X}\mathbb{E}\left(F((X^{2}(s)\otimes m)^{\mathcal{B}_{t}^{s}})\right),X^{1}(s)-X^{2}(s)\bigg\rangle ds.
\end{aligned}
\end{equation*}

Therefore, using the assumptions (\ref{eq:4-1003}), (\ref{eq:4-1004}), (\ref{eq:3-51}), (\ref{eq:3-52})
we obtain 
\begin{align}
    &\quad \int_{t}^{T} \left\langle D_{v}J_{X_\cdot t}(v^{1}(\cdot))(s)-D_{v}J_{X_\cdot t}(v^{2}(\cdot))(s),v^{1}(s)-v^{2}(s)\right\rangle ds\nonumber\\
    &\geq \lambda\int_{t}^{T} \left\lVert v^{1}(s)-v^{2}(s)\right\rVert^{2}ds-(c'_{l}+c')\int_{t}^{T}\left\lVert X^{1}(s)-X^{2}(s)\right\rVert^{2}ds-(c'_{T}+c'_{h})\left\lVert X^{1}(T)-X^{2}(T)\right\rVert^{2}.\nonumber
\end{align}

We next use 
\[
X^{1}(s)-X^{2}(s)=\int_{t}^{s}(v^{1}(\tau)-v^{2}(\tau))d\tau,
\]
 hence 
\begin{align*}
\left\lVert X^{1}(T)-X^{2}(T) \right\rVert^{2} &\leq T\int_{t}^{T} \left\lVert v^{1}(s)-v^{2}(s)\right\rVert^{2}ds, \\
\int_{t}^{T} \left\lVert X^{1}(s)-X^{2}(s)\right\rVert^{2}ds &\leq\dfrac{T^{2}}{2}\int_{t}^{T} \left\lVert v^{1}(s)-v^{2}(s)\right\rVert^{2}ds.
\end{align*}
Collecting results, we obtain easily (\ref{eq:Ap1}) and the convexity
is proven. From (\ref{eq:Ap1}), it also follows that 
\begin{equation}
\int_{t}^{T} \left\langle D_{v}J_{X_{\cdot}t}(v_{X_{\cdot}t}(\cdot))(s)-D_{v}J_{X_{\cdot}t}(0)(s),v_{X_{\cdot}t}(s)\right\rangle ds\geq c_{0}\int_{t}^{T} \left\lVert v_{X_{\cdot}t}(s)\right\rVert^{2}ds.\label{eq:Ap2}
\end{equation}
 But 
\[
J_{X_{\cdot}t}(v_{X_{\cdot}t}(\cdot))-J_{X_{\cdot}t}(0)=\int_{0}^{1} \int_{t}^{T} \left\langle D_{v}J_{X_{\cdot}t}(\theta v_{X_{\cdot}t}(\cdot))(s),v_{X_{\cdot}t}(s)\right\rangle dsd\theta,
\]
and combining with (\ref{eq:Ap2}) we obtain 
\[
J_{X_{\cdot}t}(v_{X_{\cdot}t}(\cdot))-J_{X_{\cdot}t}(0)\geq\int_{t}^{T} \langle D_{v}J_{X_{\cdot}t}(0)(s),v_{X_{\cdot}t}(s) \rangle ds+\dfrac{c_{0}}{2}\int_{t}^{T}||v_{X_{\cdot}t}(s)||^{2}ds,
\]
which implies the coercivity, and the existence and uniqueness of
a minimum of $J_{Xt}(v_{Xt}(\cdot)).$ This completes the proof. 
\subsection{PROOF OF PROPOSITION \ref{prop3-2}}

We begin by writing the equivalent of (\ref{eq:3-603}) for a problem
where the initial conditions $X_{x},t$ are replaced with $X_{X_{x}t}(t+\epsilon),t+\epsilon,$
conditioning with respect to $\mathcal{B}_{t}^{t+\epsilon}.$ The
optimal control is $u_{X_{X_{\cdot}t}(t+\epsilon),t+\epsilon}(s)$ and
the optimal state is $X_{X_{X_{\cdot}t}(t+\epsilon),t+\epsilon}(s)$ defined
by:
\begin{equation}
X_{X_{X_{\cdot}t}(t+\epsilon),t+\epsilon}(s)=X_{X_{\cdot}t}(t+\epsilon)+\int_{t+\epsilon}^{s}u_{X_{X_{\cdot}t}(t+\epsilon),t+\epsilon}(\tau)d\tau+\sigma(w(s)-w(t+\epsilon))+\beta(b(s)-b(t+\epsilon)).\label{eq:Ap3}
\end{equation}
Also $X_{\cdot}\otimes m$ must be replaced with $(X_{\cdot, t}(t+\epsilon)\otimes(X_{\cdot}\otimes m))^{\mathcal{B}_{t}^{t+\epsilon}}.$

The condition is then the following: 
\begin{align}
    &\int_{t+\epsilon}^{T}\mathbb{E}\Bigg(\int_{\mathbb{R}^{n}}\left[l_{v}(X_{\eta,t+\epsilon}(s),u_{\eta,t+\epsilon}(s))+\int_{s}^{T}\Bigg(l_{x}(X_{\eta,t+\epsilon}(\tau),u_{\eta,t+\epsilon}(\tau))\right.\nonumber \\
    &\qquad\qquad +D\dfrac{dF}{d\nu} \left(\left(X_{\cdot ,t+\epsilon}(\tau)\otimes(X_{\cdot, t}(t+\epsilon)\otimes(X_{\cdot}\otimes m))^{\mathcal{B}_{t}^{t+\epsilon}}\right)^{\mathcal{B}_{t+\epsilon}^{\tau}}\right)(X_{\eta,t+\epsilon}(\tau))\Bigg)d\tau+h_{x}(X_{\eta,t+\epsilon}(T))\nonumber\\
    &\qquad\qquad +\left.D\dfrac{dF_{T}}{d\nu}\left(\left(X_{\cdot ,t+\epsilon}(T)\otimes(X_{\cdot, t}(t+\epsilon)\otimes(X_{\cdot}\otimes m))^{\mathcal{B}_{t}^{t+\epsilon}}\right)^{\mathcal{B}_{t+\epsilon}^{T}}\right)(X_{\eta,t+\epsilon}(T))\right]\nonumber\\
    &\qquad\qquad \cdot\widetilde{v}_{\eta,t+\epsilon}(s)d \left((X_{\cdot, t}(t+\epsilon)\otimes(X_{\cdot}\otimes m))^{\mathcal{B}_{t}^{t+\epsilon}}\right)(\eta)\Bigg)ds=0.  \label{eq:Ap4}
\end{align}


By the independence of $\mathcal{B}_{t}^{t+\epsilon}$ and $X_{\eta,t+\epsilon}(s),u_{\eta,t+\epsilon}(s)$
this can be written as follows:
\begin{align}
    &\int_{t+\epsilon}^{T}\mathbb{E}\Bigg(\int_{\mathbb{R}^{n}}\Bigg[l_{v}(X_{X_{\eta t}(t+\epsilon),t+\epsilon}(s),u_{X_{\eta t}(t+\epsilon),t+\epsilon}(s))+\int_{s}^{T}\Bigg(l_{x}(X_{X_{\eta t}(t+\epsilon),t+\epsilon}(\tau),u_{X_{\eta t}(t+\epsilon),t+\epsilon}(\tau) )\nonumber  \\
    &\qquad\qquad +D\dfrac{dF}{d\nu}\left(\left(X_{\cdot ,t+\epsilon}(\tau)\otimes(X_{\cdot, t}(t+\epsilon)\otimes(X_{\cdot}\otimes m))^{\mathcal{B}_{t}^{t+\epsilon}}\right)^{\mathcal{B}_{t+\epsilon}^{\tau}}\right)(X_{X_{\eta t}(t+\epsilon),t+\epsilon}(\tau))\Bigg)d\tau+h_{x}(X_{X_{\eta t}(t+\epsilon),t+\epsilon}(T))\nonumber\\
    &\qquad\qquad + D\dfrac{dF_{T}}{d\nu}\left(\left(X_{\cdot ,t+\epsilon}(T)\otimes(X_{\cdot, t}(t+\epsilon)\otimes(X_{\cdot}\otimes m))^{\mathcal{B}_{t}^{t+\epsilon}}\right)^{\mathcal{B}_{t+\epsilon}^{T}}\right)(X_{X_{\eta t}(t+\epsilon),t+\epsilon}(T))\Bigg]\nonumber\\
    &\qquad\qquad \cdot \widetilde{v}_{X_{\eta t}(t+\epsilon),t+\epsilon}(s)d(X_{\cdot}\otimes m)(\eta)\Bigg) ds=0.\label{eq:Ap40}
\end{align}


If we consider the control $u_{X_{\cdot}t}(s)$, for which the corresponding
state is $X_{X_{\cdot}t}(s)$, we first check 
\begin{equation}
\left(X_{\cdot, t+\epsilon}(s)\otimes(X_{\cdot, t}(t+\epsilon)\otimes(X_{\cdot}\otimes m))^{\mathcal{B}_{t}^{t+\epsilon}} \right)^{\mathcal{B}_{t+\epsilon}^{s}}= (X_{\cdot, t}(s)\otimes(X_{\cdot}\otimes m))^{\mathcal{B}_{t}^{s}}, \quad \forall s>t+\epsilon.\label{eq:Ap5}
\end{equation}
Indeed for a continuous bounded test function $\varphi(\xi)$ we have 
\begin{align}
    &\quad \int_{\mathbb{R}^{n}}\varphi(\xi)d \left(X_{\cdot, t+\epsilon}(s)\otimes(X_{\cdot, t}(t+\epsilon)\otimes(X_{\cdot}\otimes m))^{\mathcal{B}_{t}^{t+\epsilon}}\right)^{\mathcal{B}_{t+\epsilon}^{s}}(\xi)\nonumber\\
&=\mathbb{E}^{\mathcal{B}_{t+\epsilon}^{s}} \left(\int_{\mathbb{R}^{n}}\varphi(X_{\eta,t+\epsilon}(s))d(X_{\cdot, t}(t+\epsilon)\otimes(X_{\cdot}\otimes m))^{\mathcal{B}_{t}^{t+\epsilon}}(\eta)\right)\nonumber\\
    &=\mathbb{E}^{\mathcal{B}_{t+\epsilon}^{s}}\left[\mathbb{E}^{1\mathcal{B}_{t}^{^{t+\epsilon}}}\left(\int_{\mathbb{R}^{n}}\varphi(X_{X_{\zeta t}^{1}(t+\epsilon),t+\epsilon}(s))d(X_{\cdot}\otimes m)(\zeta)\right)\right],\nonumber
\end{align}

where, given $\mathcal{B}_{t}^{t+\epsilon}$, $X_{\zeta t}^{1}(t+\epsilon)$
is an independent copy of $X_{\zeta t}(t+\epsilon)$ and the conditional
expectation $\mathbb{E}^{1\mathcal{B}_{t}^{^{t+\epsilon}}}$ is taken with
respect to this copy. We can then interchange the conditional expectations to simplify it to
\[
\mathbb{E}^{1\mathcal{B}_{t}^{^{t+\epsilon}}} \left[\mathbb{E}^{\mathcal{B}_{t+\epsilon}^{s}}\left(\int_{\mathbb{R}^{n}}\varphi(X_{X_{\zeta t}^{1}(t+\epsilon),t+\epsilon}(s))d(X_{\cdot}\otimes m)(\zeta)\right) \right],
\]
and from the independence of $X_{\xi,t+\epsilon}(s)$ and $X_{\zeta t}(t+\epsilon$)
for fixed $\xi,\zeta$, we finally can write the integral in question as
\[
\mathbb{E}^{\mathcal{B}_{t}^{s}}\left(\int_{\mathbb{R}^{n}}\varphi(X_{\zeta t}(s))d(X_{\cdot}\otimes m)(\zeta)\right),
\]
which proves (\ref{eq:Ap5}). But then (\ref{eq:Ap40}) becomes: 
\begin{align}
    &\int_{t+\epsilon}^{T}\mathbb{E}\Bigg(\int_{\mathbb{R}^{n}}\Bigg[l_{v}(X_{\eta t}(s),u_{\eta t}(s))+\int_{s}^{T} \Bigg(l_{x}(X_{\eta t}(\tau),u_{\eta t}(\tau)) + D\dfrac{dF}{d\nu} \left((X_{\cdot, t}(\tau)\otimes(X_{\cdot}\otimes m))^{\mathcal{B}_{t}^{\tau}} \right)(X_{\eta t}(\tau))\Bigg) d\tau\nonumber\\
    &\qquad\qquad +h_{x}(X_{\eta t}(T))+D\dfrac{dF_{T}}{d\nu}\left((X_{\cdot, t}(T)\otimes(X_{\cdot}\otimes m))^{\mathcal{B}_{t}^{T}} \right)(X_{\eta t}(T))\Bigg]\cdot\widetilde{v}_{X_{\eta t}(t+\epsilon),t+\epsilon}(s)d(X_{\cdot}\otimes m)(\eta)\Bigg)ds=0,\label{eq:Ap41}
\end{align}


and this is true, by taking in (\ref{eq:3-603}):
\begin{numcases}
    {\widetilde{v}_{\eta t}(s)=}
    0, & if $t<s<t+\epsilon$;\nonumber\\
    \widetilde{v}_{X_{\eta t}(t+\epsilon),t+\epsilon}(s), & if $s>t+\epsilon$.\nonumber
\end{numcases}
This completes the proof.

\section{PROOFS FROM SECTION \ref{sec:properties of value}} \label{appendix B}

\subsection{PROOF OF PROPOSITION \ref{prop4-1}}

To simplify notation, we omit the indices $X_{\cdot}t$ in $X_{X_{\cdot}t}(s),Z_{X_{\cdot}t}(s),u_{X_{\cdot} t}(s).$
We have 
\begin{align}
    &\quad \dfrac{d}{ds} \langle Z(s),X(s)-\sigma(w(s)-w(t))-\beta(b(s)-b(t))\rangle \nonumber\\
    &=\langle Z(s),u(s)\rangle -\bigg\langle l_{x}(X(s),u(s))+D_{X}\mathbb{E}\left(F((X_{X_{\cdot}t}(s)\otimes m)^{\mathcal{B}_{t}^{s}})\right),X(s)-\sigma(w(s)-w(t))-\beta(b(s)-b(t))\bigg\rangle .\nonumber
\end{align}


Integrating between $t$ and $T,$ we obtain 
\begin{align}
    \langle X,Z(t)\rangle &=\int_{t}^{T} \langle l_{v}(X(s),u(s)),u(s)\rangle ds \nonumber\\
    &\quad \, +\int_{t}^{T} \left\langle l_{x}(X(s),u(s))+D_{X}\mathbb{E}\left(F\left((X_{X_{\cdot}t}(s)\otimes m)^{\mathcal{B}_{t}^{s}} \right)\right),X(s)-\sigma(w(s)-w(t))-\beta(b(s)-b(t))\right\rangle ds\nonumber\\
    &\quad \, +\left\langle h_{x}(X(T))+D_{X}\mathbb{E}\left(F\left((X_{X_{\cdot}t}(T)\otimes m)^{\mathcal{B}_{t}^{T}}\right)\right),X(T)-\sigma(w(T)-w(t))-\beta(b(T)-b(t) \right\rangle. \label{eq:ApB2}
\end{align}

We also have 
\begin{equation} \label{eq:ApB3}
\begin{aligned}
    \langle X,Z(t)\rangle =& \left\langle X,\int_{t}^{T}\bigg(l_{x}(X(s),u(s))+D_{X}\mathbb{E}\left(F\left((X_{X_{\cdot}t}(s)\otimes m)^{\mathcal{B}_{t}^{s}}\right) \right)\bigg)ds\right\rangle \\
    &+\left\langle X,h_{x}(X(T))+D_{X}\mathbb{E}\left(F\left((X_{X_{\cdot}t}(T)\otimes m)^{\mathcal{B}_{t}^{T}}\right)\right)\right\rangle.
\end{aligned}
\end{equation}


Therefore, 
\begin{align}
    0=&\int_{t}^{T} \langle l_{v}(X(s),u(s)),u(s) \rangle ds \nonumber\\
    &+\int_{t}^{T} \left\langle l_{x}(X(s),u(s))+D_{X}\mathbb{E}\left(F\left((X_{X_{\cdot}t}(s)\otimes m)^{\mathcal{B}_{t}^{s}}\right)\right),X(s)-X-\sigma(w(s)-w(t))-\beta(b(s)-b(t))\right\rangle ds\nonumber\\
    &+\left\langle h_{x}(X(T))+\ D_{X}\mathbb{E}\left(F_{T}\left((X_{X_{\cdot}t}(T)\otimes m)^{\mathcal{B}_{t}^{T}}\right)\right),X(T)-X-\sigma(w(T)-w(t))-\beta(b(T)-b(t))\right\rangle \,. \label{eq:ApB4}
\end{align}


From the assumptions (\ref{eq:4-1000}), (\ref{eq:4-1001}), (\ref{eq:4-1003}), (\ref{eq:4-1004}), (\ref{eq:3-2}), (\ref{eq:3-3}), (\ref{eq:3-4}), (\ref{eq:3-51}), (\ref{eq:3-52}) we obtain, after
easy majorations 
\begin{equation}
0\geq \left(\lambda-(c'_{l}+c')\dfrac{T^{2}}{2}-(c'_{h}+c'_{T})T\right)\int_{t}^{T}||u(s)||^{2}ds-C_{T}(1+||X_{\cdot}||)\left(\int_{t}^{T}||u(s)||^{2}ds\right)^{\frac{1}{2}}. \label{eq:ApB5}
\end{equation}
Thanks to (\ref{eq:3-6}), we obtain 
\begin{equation}
\int_{t}^{T}||u(s)||^{2}ds\leq C_{T}(1+||X_{\cdot}||^{2}), \label{eq:ApB6}
\end{equation}
therefore
\begin{equation}
\sup_{s\in(t,T)}||X(s)||\leq C_{T}(1+||X||). \label{eq:ApB8}
\end{equation}
From (\ref{eq:3-600}) we can check that 
\begin{align}
    ||Z(s)||^{2} &=2\int_{s}^{T} \left\langle Z(\tau),l_{x}(X(\tau),u(\tau))+D_{X}\mathbb{E}\left(F\left((X_{X_{\cdot}t}(\tau)\otimes m)^{\mathcal{B}_{t}^{\tau}} \right)\right)\  \right\rangle d\tau \nonumber\\
    &\quad \, +\left\lVert h_{x}(X(T))+D_{X}\mathbb{E}\left(F_{T}\left((X_{X_{\cdot}t}(\tau)\otimes m)^{\mathcal{B}_{t}^{\tau}}\right)\right)\right\rVert^{2}-\sum_{j=1}^{n}\int_{s}^{T}||r^{j}(\tau)||^{2}d\tau-\sum_{j=1}^{n}\int_{s}^{T}||\rho^{j}(\tau)||^{2}d\tau,\label{eq:ApB7}
\end{align}

from which it follows immediately that
\begin{equation} \label{eq:ApB70}
\left\{
\begin{aligned}
    &\sup_{s\in(t,T)}||Z(s)||\leq C_{T}(1+||X||),\\
    &\sum_{j=1}^{n}\int_{s}^{T}||r^{j}(\tau)||^{2}d\tau+\sum_{j=1}^{n}\int_{s}^{T}||\rho^{j}(\tau)||^{2}d\tau\leq C_{T}(1+||X||^{2}).
\end{aligned}\right.
\end{equation}
From (\ref{eq:3-7}) and the assumptions (\ref{eq:4-1000}), (\ref{eq:4-1003})
it follows also that
\begin{equation}
\sup_{s\in(t,T)}||u(s)||\leq C_{T}(1+||X||), \label{eq:Ap71}
\end{equation}
which completes the proof.

\subsection{PROOF OF PROPOSITION \ref{prop4-2} }

Let $X_{\cdot}^{1},X_{\cdot}^{2}\in\mathcal{H}_{m}$, independent of $\mathcal{W}_{t}.$
We consider the systems of necessary and sufficient conditions \eqref{eq:3-7} corresponding to the initial
conditions $X_{\cdot}=X_{\cdot}^{1}$ and $X_{\cdot} = X_{\cdot}^{2}$, respectively. To simplify
notation, we denote by $u^{1},X^{1},Z^{1},r^{j,1},\rho^{j,1}$ and
$u^{2},X^{2},Z^{2},r^{j,2},\rho^{j,2}$ the related processes $u,X,Z,r^{j},\rho^{j}.$

According to Remark \ref{rem4-1}, we can use a common space of controls
for the two problems with initial conditions $X_{\cdot}^{1}$ or $X_{\cdot}^{2},$
namely $L_{\mathcal{F}_{X^{1}X^{2}t}}^{2}(t,T;\mathcal{H}_{m}).$
Therefore, we may consider using the control $u^{2}(s)$ with the initial condition $X^{1}.$
This is suboptimal, and the corresponding trajectory is $X^{2}(s)+X_{\cdot}^{1}-X_{\cdot}^{2}$. The suboptimality allows to write the inequality: 
\begin{align}
    &\quad V(X_{\cdot}^{1}\otimes m,t)-V(X_{\cdot}^{2}\otimes m,t) \nonumber\\
    &\leq\int_{t}^{T} \Bigg[\mathbb{E}\left(\int_{\mathbb{R}^{n}}\left(l(X^{2}(s)+X_{\cdot}^{1}-X_{\cdot}^{2},u^{2}(s))-l(X^{2}(s),u^{2}(s)) \right)dm(x)\right) \notag\\
    &\qquad\qquad+\mathbb{E}\left(F\left(((X_{X_{\cdot}^{2}t}(s)+X_{\cdot}^{1}-X_{\cdot}^{2})\otimes m)^{\mathcal{B}_{t}^{s}}\right)\right)-\mathbb{E}\left(F \left((X_{X_{\cdot}^{2}t}(s)\otimes m)^{\mathcal{B}_{t}^{s}}\right)\right)\Bigg]ds\nonumber\\
    &\qquad\quad+\mathbb{E}\left(\int_{\mathbb{R}^{n}} \left(h(X^{2}(T)+X_{\cdot}^{1}-X_{\cdot}^{2})-h(X^{2}(T))\right)dm(x)\right)\nonumber\\
    &\qquad\quad+\mathbb{E}\left(F_{T} \left(((X_{X_{\cdot}^{2}t}(T)+X_{\cdot}^{1}-X_{\cdot}^{2})\otimes m)^{\mathcal{B}_{t}^{T}}\right)\right)-\mathbb{E}\left(F\left((X_{X_{\cdot}^{2}t}(T)\otimes m)^{\mathcal{B}_{t}^{T}}\right)\right)\nonumber\\
    &=\int_{t}^{T}\left\langle l_{x}(X^{2}(s),u^{2}(s))+D_{X}\mathbb{E}\left(F((X_{X_{\cdot}^{2}t}(s)\otimes m)^{\mathcal{B}_{t}^{s}})\right),X_{\cdot}^{1}-X_{\cdot}^{2}\right\rangle ds\nonumber\\
    &\quad \,+\int_{0}^{1}\int_{0}^{1}\int_{t}^{T}\lambda \Bigg[\left\langle l_{xx}(X^{2}(s)+\lambda\mu(X_{\cdot}^{1}-X_{\cdot}^{2}),u^{2}(s))(X_{\cdot}^{1}-X_{\cdot}^{2}),X_{\cdot}^{1}-X_{\cdot}^{2}\right\rangle \nonumber\\
    &\qquad\qquad\qquad\qquad +\left\langle D_{X}^{2}\mathbb{E}\left(F\left(((X_{X_{\cdot}^{2}t}(s)+\lambda\mu(X_{\cdot}^{1}-X_{\cdot}^{2}))\otimes m)^{\mathcal{B}_{t}^{s}}\right)\right)(X_{\cdot}^{1}-X_{\cdot}^{2}),X_{\cdot}^{1}-X_{\cdot}^{2}\right\rangle \Bigg]dsd\lambda d\mu\nonumber\\
    &\quad+\left\langle h_{x}(X^{2}(T))+D_{X}\mathbb{E}\left(F_{T}\left((X_{X_{\cdot}^{2}t}(T)\otimes m)^{\mathcal{B}_{t}^{T}}\right)\right),X_{\cdot}^{1}-X_{\cdot}^{2}\right\rangle \nonumber\\
    &\quad+\int_{0}^{1}\int_{0}^{1}\lambda\Bigg[\left\langle h_{xx}(X^{2}(T)+\lambda\mu(X_{\cdot}^{1}-X_{\cdot}^{2}))(X_{\cdot}^{1}-X_{\cdot}^{2}),X_{\cdot}^{1}-X_{\cdot}^{2}\right\rangle \nonumber\\
    &\qquad\qquad\qquad \, +\left\langle D_{X}^{2}\mathbb{E}\left(F_{T}\left(((X_{X_{\cdot}^{2}t}(T)+\lambda\mu(X_{\cdot}^{1}-X_{\cdot}^{2}))\otimes m)^{\mathcal{B}_{t}^{T}}\right)\right)(X_{\cdot}^{1}-X_{\cdot}^{2}),X_{\cdot}^{1}-X_{\cdot}^{2}\right\rangle \Bigg]d\lambda d\mu,\label{eq:ApB14}
\end{align}

and from the assumptions (\ref{eq:4-1000}), (\ref{eq:4-1001}), (\ref{eq:3-4}), (\ref{eq:3-50}) we have
\begin{align} \label{eq:ApB15}
    V(X_{\cdot}^{1}\otimes m,t)-V(X_{\cdot}^{2}\otimes m,t)\leq \left\langle Z^{2}(t),X_{\cdot}^{1}-X_{\cdot}^{2}\right\rangle + \dfrac{1}{2}(c_{h}+c_{T}+(c_{l}+c)T) \left\lVert X_{\cdot}^{1}-X_{\cdot}^{2}\right\rVert^{2}.
\end{align}

By changing the role of $X_{\cdot}^{1}$ and $X_{\cdot}^{2}$, we get 
\begin{equation}
V(X^{1}\otimes m,t)-V(X^{2}\otimes m,t)\geq \left\langle Z^{2}(t),X^{1}-X^{2}\right\rangle -\dfrac{1}{2}(c_{h}+c_{T}+(c_{l}+c)T) \left\lVert X^{1}-X^{2}\right\rVert ^{2}+ \left\langle Z^{1}(t)-Z^{2}(t),X_{\cdot}^{1}-X_{\cdot}^{2}\right\rangle.\label{eq:ApB16}
\end{equation}
Next, 
\begin{equation*} \hspace{-1.5cm}
\begin{aligned}
    &\quad \dfrac{d}{ds} \left\langle X^{1}(s)-X^{2}(s),Z^{1}(s)-Z^{2}(s)\right\rangle \\
    &= \left\langle u^{1}(s)-u^{2}(s),l_{v}(X^{1}(s),u^{1}(s))-l_{v}(X^{2}(s),u^{2}(s))\right\rangle \\
    &\quad \, -\bigg\langle X^{1}(s)-X^{2}(s),l_{x}(X^{1}(s),u^{1}(s))-l_{x}(X^{2}(s),u^{2}(s))\bigg\rangle\\
    &\quad \, + \bigg\langle D_{X}\mathbb{E}\left(F\left((X_{X_{\cdot}^{1}t}(s)\otimes m)^{\mathcal{B}_{t}^{s}}\right)\right)-D_{X}\mathbb{E}\left(F\left((X_{X_{\cdot}^{2}t}(s)\otimes m)^{\mathcal{B}_{t}^{s}}\right)\right), X^1(s) - X^2(s)\bigg\rangle .
\end{aligned}
\end{equation*}


Integrating between $t$ and $T$ and using (\ref{eq:4-1003}), (\ref{eq:4-1004}), (\ref{eq:3-51}), (\ref{eq:3-52})
it follows that
\begin{equation} \label{eq:ApB17}
\begin{aligned}
\left\langle Z^{1}(t)-Z^{2}(t),X^{1}-X^{2} \right\rangle &\geq\lambda\int_{t}^{T} \left\lVert u^{1}(s)-u^{2}(s) \right\rVert ^{2}ds-(c'_{l}+c')\int_{t}^{T} \left\lVert X^{1}(s)-X^{2}(s)\right\rVert^{2}ds  \\
&\quad \, -(c'_{h}+c'_{T}) \left\lVert X^{1}(T)-X^{2}(T) \right\rVert^{2}.
\end{aligned}  
\end{equation}
But 
\[
X^{1}(s)-X^{2}(s)=X^{1}-X^{2}+\int_{t}^{s}(u^{1}(\tau)-u^{2}(\tau))d\tau,
\]
from which we get 
\begin{align*}
    \int_{t}^{T} \left\lVert X^{1}(s)-X^{2}(s)\right\rVert^{2}ds&\leq(1+\epsilon)\dfrac{T^{2}}{2}\int_{t}^{T} \left\lVert u^{1}(s)-u^{2}(s)\right\rVert^{2}ds+T\left(1+\dfrac{1}{\epsilon}\right)\left\lVert X^{1}-X^{2}\right\rVert^{2}, \\
    \left\lVert X^{1}(T)-X^{2}(T)\right\rVert^{2}&\leq(1+\epsilon)T\int_{t}^{T}\left\lVert u^{1}(s)-u^{2}(s)\right\rVert^{2}ds+\left(1+\dfrac{1}{\epsilon}\right)\left\lVert X^{1}-X^{2}\right\rVert^{2},
\end{align*}
so we obtain the estimate:
\begin{equation} \label{eq:ApB18}
\begin{aligned}
    \left\langle Z^{1}(t)-Z^{2}(t),X_{\cdot}^{1}-X_{\cdot}^{2}\right\rangle &\geq \lambda-T(1+\epsilon)\left((c'_{l}+c')\dfrac{T}{2}+(c'_{h}+c'_{T})\right)\int_{t}^{T} \left\lVert u^{1}(s)-u^{2}(s)\right\rVert^{2}ds\\
    &\quad -\left(1+\dfrac{1}{\epsilon}\right)((c'_{l}+c')T+c'_{h}+c'_{T})\left\lVert X^{1}-X^{2}\right\rVert^{2} \,.
\end{aligned}   
\end{equation}


From the assumption (\ref{eq:3-6}) and choosing $\epsilon$ sufficiently
small, the first term in the right-hand side is positive. Combined
with (\ref{eq:ApB16}) and taking account of (\ref{eq:ApB15}) it
follows that 
\begin{equation}
\left|V(X^{1}\otimes m,t)-V(X^{2}\otimes m,t)-\left\langle Z^{2}(t),X^{1}-X^{2}\right\rangle \right|\leq C_{T}\left\lVert X^{1}-X^{2}\right\rVert^{2}.\label{eq:ApB19}
\end{equation}
This proves immediately the result (\ref{eq:4-3}). We have also proved
the following estimate:
\begin{equation}
\alpha\int_{t}^{T}\left\lVert u^{1}(s)-u^{2}(s)\right\rVert^{2}ds\leq \left\langle Z^{1}(t)-Z^{2}(t),X_{\cdot}^{1}-X_{\cdot}^{2}\right\rangle +\beta \left\lVert X_{\cdot}^{1}-X_{\cdot}^{2}\right\rVert^{2},\label{eq:ApB200}
\end{equation}
for convenient constants $\alpha,\beta.$ Next, we also have
\begin{align}
    &\quad \, \left\langle Z^{1}(t)-Z^{2}(t),X_{\cdot}^{1}-X_{\cdot}^{2}\right\rangle \nonumber\\
    &=\Bigg\langle \int_{t}^{T} \left[l_{x}(X^{1}(s),u^{1}(s))-l_{x}(X^{2}(s),u^{2}(s))+D_{X}\mathbb{E}\left(F\left((X_{X_{\cdot}^{1}t}(s)\otimes m)^{\mathcal{B}_{t}^{s}}\right)\right)-D_{X}\mathbb{E}\left(F\left((X_{X_{\cdot}^{2}t}(s)\otimes m)^{\mathcal{B}_{t}^{s}}\right)\right)\right]ds\nonumber\\
    &\quad +h_{x}(X^{1}(T))-h_{x}(X^{2}(T))+D_{X}\mathbb{E}\left(F_{T}\left((X_{X_{\cdot}^{1}t}(T)\otimes m)^{\mathcal{B}_{t}^{T}}\right)\right)-D_{X}\mathbb{E}\left(F_{T}\left((X_{X_{\cdot}^{2}t}(T)\otimes m)^{\mathcal{B}_{t}^{T}}\right)\right),X_{\cdot}^{1}-X_{\cdot}^{2}\Bigg\rangle \nonumber\\
    &\leq\left\lVert X_{\cdot}^{1}-X_{\cdot}^{2}\right\rVert \Bigg[(c_{l}+c)\int_{t}^{T} \left\lVert X^{1}(s)-X^{2}(s)\right\rVert ds+(c_{h}+c_{T}) \left\lVert X^{1}(T)-X^{2}(T)\right\rVert \Bigg],\label{eq:ApB201}
\end{align}

 and 
\begin{align*}
    \int_{t}^{T} \left\lVert X^{1}(s)-X^{2}(s)\right\rVert ds&\leq \left\lVert X^{1}-X^{2}\right\rVert +\dfrac{2}{3}\sqrt{T^{3}}\sqrt{\int_{t}^{T} \left\lVert u^{1}(s)-u^{2}(s)\right\rVert^{2}ds}, \\
    \left\lVert X^{1}(T)-X^{2}(T)\right\rVert&\leq \left\lVert X^{1}-X^{2}\right\rVert +\sqrt{T}\sqrt{\int_{t}^{T} \left\lVert u^{1}(s)-u^{2}(s)\right\rVert^{2}ds}.
\end{align*}
Using these estimates in (\ref{eq:ApB201}), combined with (\ref{eq:ApB200}),
we obtain 
\begin{equation}
\int_{t}^{T} \left\lVert u^{1}(s)-u^{2}(s)\right\rVert ^{2}ds\leq C_{T} \left\lVert X_{\cdot}^{1}-X_{\cdot}^{2} \right\rVert^{2},\label{eq:ApB21}
\end{equation}
which implies 
\begin{equation} \label{eq:ApB202}
\begin{aligned}
    \left\lVert X^{1}(s)-X^{2}(s)\right\rVert &\leq C_{T} \left\lVert X_{\cdot}^{1}-X_{\cdot}^{2}\right\lVert ,\\
    \left\lVert Z^{1}(s)-Z^{2}(s)\right\rVert &\leq C_{T} \left\lVert X_{\cdot}^{1}-X_{\cdot}^{2} \right\rVert , \\
    \left\lVert u^{1}(s)-u^{2}(s)\right\rVert &\leq C_{T} \left\lVert X_{\cdot}^{1}-X_{\cdot}^{2} \right\rVert , \\
    \sum_{j=1}^{n}\int_{t}^{T} \left\lVert r_{X^{1}t}^{j}(s)-r_{X^{2}t}^{j}(s)\right\rVert ^{2}ds+\sum_{j=1}^{n}\int_{t}^{T} \left\lVert \rho_{X^{1}t}^{j}(s)-\rho_{X^{2}t}^{j}(s)\right\rVert^{2}ds &\leq C_{T} \left\lVert X_{\cdot}^{1}-X_{\cdot}^{2} \right\rVert^{2},
\end{aligned}
\end{equation}
and thus (\ref{eq:4-4}) has been proven, which completes the proof. 

\subsection{PROOF OF PROPOSITION \ref{prop4-5}}

We begin with (\ref{eq:4-26}). We have, from the optimality principle
(\ref{eq:3-16}),
\begin{equation} \label{eq:ApB22}
\begin{aligned}
    V(X\otimes m,t)-V(X\otimes m,t+\epsilon)=&\int_{t}^{t+\epsilon}\left[\mathbb{E}\left(\int_{\mathbb{R}^{n}}l(X_{X_{\cdot}t}(s),u_{X_{\cdot}t}(s))dm(x)\right)+F\left((Y_{X_{\cdot}t}(s)\otimes m)^{\mathcal{B}_{t}^{s}}\right)\right]ds\\
    &+\mathbb{E}\left(V\left((X_{X_{\cdot}t}(t+\epsilon)\otimes m)^{\mathcal{B}_{t}^{t+\epsilon}},t+\epsilon\right)\right)-V(X\otimes m,t+\epsilon).
\end{aligned}
\end{equation}


Since 
\begin{equation}
\left|\int_{t}^{t+\epsilon}\left[\mathbb{E}\left(\int_{\mathbb{R}^{n}}l(X_{X_{\cdot}t}(s),u_{X_{\cdot}t}(s))dm(x)\right)+F(Y_{X_{\cdot}t}(s)\otimes m)\right]ds\right|\leq C_{T}\epsilon(1+||X_{\cdot}||^{2}), \label{eq:ApB220}
\end{equation}
it suffices to prove 
\begin{equation}
\left|\mathbb{E}\left(V\left((X_{X_{\cdot}t}(t+\epsilon)\otimes m)^{\mathcal{B}_{t}^{t+\epsilon}},t+\epsilon\right)\right)-V(X\otimes m,t+\epsilon)\right|\leq C_{T}\epsilon(1+||X_{\cdot}||^{2}).\label{eq:ApB221}
\end{equation}
But 
\begin{equation}\label{eq:ApB222}
    \begin{aligned}
    &\quad \, \mathbb{E}\left(V\left((X_{X_{\cdot}t}(t+\epsilon)\otimes m)^{\mathcal{B}_{t}^{t+\epsilon}},t+\epsilon\right)\right)-V(X\otimes m,t+\epsilon)\\
    &=\int_{t+\epsilon}^{T}\mathbb{E} \Bigg[\int_{\mathbb{R}^{n}}\bigg(l\left(X_{X_{X_{x}t}(t+\epsilon),t+\epsilon}(s),u_{X_{X_{x}t}(t+\epsilon),t+\epsilon}(s)\right)-l \left(X_{X_{x},t+\epsilon}(s),u_{X_{x},t+\epsilon}(s)\right)\bigg)dm(x)\\
    &\qquad\qquad\,\,\,+F\left((X_{X_{X_{\cdot}t}(t+\epsilon),t+\epsilon}(s)\otimes m)^{\mathcal{B}_{t}^{s}}\right)-F\left((X_{X_{\cdot},t+\epsilon}(s)\otimes m)^{\mathcal{B}_{t}^{s}}\right)\Bigg]ds\\
    &\quad+\mathbb{E}\Bigg[\int_{\mathbb{R}^{n}} \bigg(h(X_{X_{X_{x}t}(t+\epsilon),t+\epsilon}(T))-h(X_{X_{x},t+\epsilon}(T))\bigg)dm(x)\\
    &\qquad\quad\,\,\, +F_{T}\left((X_{X_{X_{\cdot}t}(t+\epsilon),t+\epsilon}(T)\otimes m)^{\mathcal{B}_{t}^{T}}\right)-F_{T}\left((X_{X_{\cdot},t+\epsilon}(T)\otimes m)^{\mathcal{B}_{t}^{T}}\right)\Bigg].
\end{aligned}
\end{equation}




Note that we have used the fact $(X_{X_{\cdot},t+\epsilon}(s)\otimes m)^{\mathcal{B}_{t}^{s}}=(X_{X_{\cdot},t+\epsilon}(s)\otimes m)^{\mathcal{B}_{t+\epsilon}^{s}}$
and the same with $s=T.$

Using Taylor expansion and making use of assumptions (\ref{eq:4-1000}), (\ref{eq:4-1001}), (\ref{eq:3-4}), (\ref{eq:3-50}),
we can check that 
\begin{align}
    &\quad \Bigg|\mathbb{E}\left(V\left((X_{X_{\cdot}t}(t+\epsilon)\otimes m)^{\mathcal{B}_{t}^{t+\epsilon}},t+\epsilon\right)\right)-V(X\otimes m,t+\epsilon)\nonumber\\
    &\quad -\int_{t+\epsilon}^{T} \left\langle l_{v}(X_{X_{\cdot},t+\epsilon}(s),u_{X_{\cdot},t+\epsilon}(s)),u_{X_{X_{\cdot}t}(t+\epsilon),t+\epsilon}(s)-u_{X_{\cdot},t+\epsilon}(s)\right\rangle ds\nonumber\\
    &\quad -\int_{t+\epsilon}^{T}\left\langle l_{x}(X_{X_{\cdot},t+\epsilon}(s),u_{X_{\cdot},t+\epsilon}(s))+D_{X}\mathbb{E}\left(F\left((X_{X_{\cdot},t+\epsilon}(s)\otimes m)^{\mathcal{B}_{t+\epsilon}^{s}}\right)\right),X_{X_{X_{\cdot}t}(t+\epsilon),t+\epsilon}(s)-X_{X_{\cdot},t+\epsilon}(s)\right\rangle ds\nonumber\\
    &\quad -\left\langle h_{x}(X_{X_{\cdot},t+\epsilon}(T))+D_{X}\mathbb{E}\left(F\left((X_{X_{\cdot},t+\epsilon}(T)\otimes m)^{\mathcal{B}_{t+\epsilon}^{T}}\right)\right),X_{X_{X_{\cdot}t}(t+\epsilon),t+\epsilon}(T)-X_{X_{\cdot},t+\epsilon}(T)\right\rangle \Bigg|\nonumber\\
    &\leq C_{T}\sup_{t<s<T} \left\lVert X_{X_{X_{\cdot}t}(t+\epsilon),t+\epsilon}(s)-X_{X_{\cdot},t+\epsilon}(s)\right\rVert^{2}. \label{eq:ApB223}
\end{align}


Since $X_{X_{X_{\cdot}t}(t+\epsilon),t+\epsilon}(s),X_{X_{\cdot},t+\epsilon}(s)$
are the optimal states corresponding to different initial conditions
$X_{X_{\cdot}t}(t+\epsilon)$ and $X_{\cdot}$ at initial time $t+\epsilon,$
we can use (\ref{eq:ApB202}) to see that 
\begin{equation}
\sup_{t<s<T} \left\lVert X_{X_{X_{\cdot}t}(t+\epsilon),t+\epsilon}(s)-X_{X_{\cdot},t+\epsilon}(s) \right\rVert^{2}\leq C_{T} \left\lVert X_{X_{\cdot}t}(t+\epsilon)-X_{\cdot} \right\rVert^{2}\leq C_{T}\epsilon(1+||X_{\cdot}||^{2}).\label{eq:ApB224}
\end{equation}
Next,
\begin{equation}
X_{X_{X_{\cdot}t}(t+\epsilon),t+\epsilon}(s)-X_{X_{\cdot},t+\epsilon}(s)=X_{X_{\cdot}t}(t+\epsilon)-X_{\cdot}+\int_{t+\epsilon}^{s}\widetilde{u}_{\mathcal{F}_{X_{\cdot}t}^{t+\epsilon}}(\tau)d\tau,\label{eq:ApB225}
\end{equation}
where we have denoted $\widetilde{u}_{\mathcal{F}_{X_{\cdot}t}^{t+\epsilon}}(\tau)=u_{X_{X_{\cdot}t}(t+\epsilon),t+\epsilon}(\tau)-u_{X_{\cdot},t+\epsilon}(\tau).$
From the optimality of $u_{X_{\cdot},t+\epsilon}(s)$ from the problem
with initial condition $X_{\cdot}$ at $t+\epsilon$, we can write, by
(\ref{eq:3-603}).
\begin{equation}\label{eq:ApB226}
    \begin{aligned}
    &\int_{t+\epsilon}^{T}\left\langle l_{v}(X_{X_{\cdot},t+\epsilon}(s),u_{X_{\cdot},t+\epsilon}(s)),\widetilde{u}_{\mathcal{F}_{X_{\cdot}t}^{t+\epsilon}}(s)\right\rangle ds\\
    &\quad+\int_{t+\epsilon}^{T}\left\langle l_{x}(X_{X_{\cdot},t+\epsilon}(s),u_{X_{\cdot},t+\epsilon}(s))+D_{X}\mathbb{E}\left(F\left((X_{X_{\cdot},t+\epsilon}(s)\otimes m)^{\mathcal{B}_{t+\epsilon}^{s}}\right)\right),\int_{t+\epsilon}^{s}\widetilde{u}_{\mathcal{F}_{X_{\cdot}t}^{t+\epsilon}}(\tau)d\tau\right\rangle ds\\
    &\quad+\left\langle h_{x}(X_{X_{\cdot},t+\epsilon}(T)+D_{X}\mathbb{E}\left(F\left((X_{X_{\cdot},t+\epsilon}(T)\otimes m)^{\mathcal{B}_{t+\epsilon}^{T}}\right)\right),\int_{t+\epsilon}^{T}\widetilde{u}_{\mathcal{F}_{X_{\cdot}t}^{t+\epsilon}}(\tau)d\tau\right\rangle =0.
\end{aligned}
\end{equation}


In addition, 
\begin{align}
    &\Bigg\langle \int_{t+\epsilon}^{T} \left(l_{x}(X_{X_{\cdot},t+\epsilon}(s),u_{X_{\cdot},t+\epsilon}(s))+D_{X}\mathbb{E}\left(F\left((X_{X_{\cdot},t+\epsilon}(s)\otimes m)^{\mathcal{B}_{t+\epsilon}^{s}}\right)\right)\right)ds +h_{x}(X_{X_{\cdot},t+\epsilon}(T))\nonumber\\
    &\quad +D_{X}\mathbb{E}\left(F\left((X_{X_{\cdot},t+\epsilon}(T)\otimes m)^{\mathcal{B}_{t+\epsilon}^{T}}\right)\right),\sigma(w(t+\epsilon)-w(t))+\beta(b(t+\epsilon)-b(t))\ \Bigg\rangle =0.\label{eq:ApB227}
\end{align}

Thanks to (\ref{eq:ApB226}), (\ref{eq:ApB227}) the inequality (\ref{eq:ApB223})
reduces to 
\begin{equation} \label{eq:ApB228}
\begin{aligned}
    &\Bigg|\mathbb{E}\left(V\left((X_{X_{\cdot}t}(t+\epsilon)\otimes m)^{\mathcal{B}_{t}^{t+\epsilon}},t+\epsilon\right)\right)-V(X\otimes m,t+\epsilon)\\
    &-\Bigg\langle \int_{t+\epsilon}^{T}\left(l_{x}(X_{X_{\cdot},t+\epsilon}(s),u_{X_{\cdot},t+\epsilon}(s))+D_{X}\mathbb{E}\left(F\left((X_{X_{\cdot},t+\epsilon}(s)\otimes m)^{\mathcal{B}_{t+\epsilon}^{s}}\right)\right)\right)ds\\
    &+h_{x}(X_{X_{\cdot},t+\epsilon}(T))+D_{X}\mathbb{E}\left(F\left((X_{X_{\cdot},t+\epsilon}(T)\otimes m)^{\mathcal{B}_{t+\epsilon}^{T}}\right)\right),\int_{t}^{t+\epsilon}u_{X_{\cdot}t}(\tau)d\tau\ \Bigg\rangle \Bigg|\leq C_{T}\epsilon(1+||X_{\cdot}||^{2}),
\end{aligned}
\end{equation}


from which (\ref{eq:ApB221}) follows immediately. We have thus proved
(\ref{eq:4-26}). We turn to (\ref{eq:4-27}), which is equivalent to 
\begin{equation}
||Z_{X,t+\epsilon}(t+\epsilon)-Z_{Xt}(t)||\leq C_{T} \left(\epsilon^{\frac{1}{2}}+\epsilon||X||\right).\label{eq:ApB24}
\end{equation}
From the BSDE (\ref{eq:3-7}) we can write
$$Z_{X_{\cdot}t}(t)=\mathbb{E}[Z_{X_{\cdot}t}(t+\epsilon)|\sigma(X)]+\mathbb{E} \left[\int_{t}^{t+\epsilon}\left(l_{x}(X_{X_{\cdot}t}(s),u_{X_{\cdot}t}(s))+D_{X}\mathbb{E}\left(F\left((X_{X_{\cdot}t}(s)\otimes m)^{\mathcal{B}_{t}^{s}}\right)\right)\right)ds\bigg|\sigma(X)\right],$$
and
\begin{equation*} \hspace{-1cm}
\begin{aligned}
    Z_{X_{\cdot},t+\epsilon}(t+\epsilon)-Z_{X_{\cdot}t}(t)&=Z_{X_{\cdot},t+\epsilon}(t+\epsilon)-\mathbb{E}[Z_{X_{\cdot}t}(t+\epsilon)|\sigma(X)]\\
   &\quad \,\, -\mathbb{E}\left[\int_{t}^{t+\epsilon}\left(l_{x}(X_{X_{\cdot}t}(s),u_{X_{\cdot}t}(s))+D_{X}\mathbb{E}\left(F\left((X_{X_{\cdot}t}(s)\otimes m)^{\mathcal{B}_{t}^{s}}\right)\right)\right)ds\bigg|\sigma(X)\right].
\end{aligned} 
\end{equation*}


Therefore, 
\[
||Z_{X_{\cdot},t+\epsilon}(t+\epsilon)-Z_{X_{\cdot}t}(t)||\leq C_{T}\epsilon(1+||X_{\cdot}||)+||Z_{X_{\cdot},t+\epsilon}(t+\epsilon)-Z_{X_{\cdot}t}(t+\epsilon)||.
\]
Noting that $Z_{X_{\cdot}t}(t+\epsilon)=Z_{X_{X_{\cdot}t}(t+\epsilon),t+\epsilon}(t+\epsilon)$
we can write from Proposition \ref{prop4-2}:
\[
||Z_{X_{X_{\cdot}t}(t+\epsilon),t+\epsilon}(t+\epsilon)-Z_{X_{\cdot},t+\epsilon}(t+\epsilon)||\leq C_{T}||X_{X_{\cdot}t}(t+\epsilon)-X_{\cdot}||,
\]
from which it follows that
\[
|Z_{X_{\cdot},t+\epsilon}(t+\epsilon)-Z_{X_{\cdot}t}(t)||\leq C_{T}\epsilon(1+||X_{\cdot}||)+C_{T}||X_{X_{\cdot}t}(t+\epsilon)-X_{\cdot}||,
\]
 and the result (\ref{eq:4-27}) is obtained immediately. 

\subsection{PROOF OF PROPOSITION \ref{prop4-3}}
\begin{proof}
Since
\begin{equation}
Z_{xmt}(t)=D_{X}V(m,t),\label{eq:4-600}
\end{equation}
we first need to show that the right hand side is a gradient. For
that purpose, we introduce an ordinary stochastic control parametrized
by the trajectory $X_{xmt}(\cdot)$ through the conditional probabilities
$(X_{\cdot mt}(s)\otimes m)^{\mathcal{B}_{t}^{s}}.$ We take controls in
$L_{\mathcal{F}_{t}}^{2}(t,T;\mathbb{R}^{n})$ and the state is defined by 
\begin{equation}
x_{xt}(s)=x+\int_{t}^{s}v(\tau)d\tau+\sigma(w(s)-w(t))+\beta(b(s)-b(t)).\label{eq:4-13}
\end{equation}
The payoff to minimize is given by 
\begin{align}
    K_{xmt}(v(\cdot))=& \, \mathbb{E}\left(\int_{t}^{T}l(x_{xt}(s),v(s))ds\right)+\int_{t}^{T}\mathbb{E}\left(\dfrac{dF}{d\nu}\left((X_{\cdot mt}(s)\otimes m)^{\mathcal{B}_{t}^{s}}\right)(x_{xt}(s))\right)ds\label{eq:4-14}\\
    &+\mathbb{E} (h(x_{xt}(T)))+\mathbb{E}\left(\dfrac{dF_{T}}{d\nu}\left((X_{\cdot mt}(T)\otimes m)^{\mathcal{B}_{t}^{T}}\right)(x_{xt}(T))\right).\nonumber
\end{align}

If we write the necessary conditions of optimality, it is easy to
check that the optimal control coincides with $u_{xmt}(s)$ and the
optimal state is $X_{xmt}(s).$ So 
\begin{equation}
\inf_{v(\cdot)}K_{xmt}(v(\cdot))=K_{xmt}(u_{xmt}(\cdot)),\label{eq:4-15}
\end{equation}
which is the right-hand side of (\ref{eq:4-10}). We set $\Psi(x,m,t)=K_{xmt}(u_{xmt}(\cdot)).$
So we need to prove that 
\begin{equation}
Z_{xmt}(t)=D\Psi(x,m,t).\label{eq:4-16}
\end{equation}
This will be a consequence of the estimate 
\begin{equation}
|\Psi(x_{1},m,t)-\Psi(x_{2},m,t)-Z_{x_{2}mt}(t) \cdot (x_{1}-x_{2})|\leq C_{T}|x_{1}-x_{2}|^{2}.\label{eq:4-17}
\end{equation}
The proof has similarities with that of Proposition \ref{prop4-2},
although it is simpler. To simplify notation, we set 
\[
u^{1}(s)=u_{x_{1}mt}(s),\ u^{2}(s)=u_{x_{2}mt}(s),\  X^{1}(s)=X_{x_{1}mt}(s),\ X^{2}(s)=X_{x_{2}mt}(s).
\]
If we use the control $u^{2}(s)$ with initial condition $x_{1},$
to get the state $X^{2}(s)+x_{1}-x_{2,}$ it is suboptimal, so we
can write: 
\begin{align}
    &\quad \, \Psi(x_{1},m,t)-\Psi(x_{2},m,t)\label{eq:4-18}\\
    &\leq\int_{t}^{T}\bigg[\mathbb{E}\left(l(X^{2}(s)+x_{1}-x_{2},u^{2}(s))-l(X^{2}(s),u^{2}(s))\right)\nonumber\\
    &\quad\qquad \,\, +\mathbb{E}\left(\dfrac{dF}{d\nu}\left((X_{\cdot mt}(s)\otimes m)^{\mathcal{B}_{t}^{s}}\right)(X^{2}(s)+x_{1}-x_{2})-\dfrac{dF}{d\nu}\left((X_{\cdot mt}(s)\otimes m)^{\mathcal{B}_{t}^{s}}\right)(X^{2}(s))\right)\bigg]ds\nonumber\\
    &\quad \,+\mathbb{E}\left(h(X^{2}(T)+x_{1}-x_{2})-h(X^{2}(T))\right)\nonumber\\
    &\quad \,+\mathbb{E}\left(\dfrac{dF}{d\nu}\left((X_{\cdot mt}(T)\otimes m)^{\mathcal{B}_{t}^{T}}\right)(X^{2}(T)+x_{1}-x_{2})-\dfrac{dF}{d\nu}\left((X_{\cdot mt}(T)\otimes m)^{\mathcal{B}_{t}^{T}}\right)(X^{2}(T))\right).\nonumber
\end{align}



Using Taylor expansion and making use of the assumptions, we obtain
the estimate:
\begin{equation}
\Psi(x_{1},m,t)-\Psi(x_{2},m,t)\leq Z_{x_{2}mt}(t)\cdot(x_{1}-x_{2})+C|x_{1}-x_{2}|^{2}.\label{eq:4-19}
\end{equation}
By interchanging $x_{1}$ and $x_{2}$ and rearranging we have also 
\begin{equation}
\Psi(x_{1},m,t)-\Psi(x_{2},m,t)\geq Z_{x_{2}mt}(t)\cdot(x_{1}-x_{2})+(Z_{x_{1}mt}(t)-Z_{x_{2}mt}(t))(x_{1}-x_{2})-C|x_{1}-x_{2}|^{2}.\label{eq:4-20}
\end{equation}
The result will follow from the estimate:
\begin{equation}
(Z_{x_{1}mt}(t)-Z_{x_{2}mt}(t))(x_{1}-x_{2})\geq-C_{T}|x_{1}-x_{2}|^{2}.\label{eq:4-21}
\end{equation}
To obtain (\ref{eq:4-21}), observe that 
\begin{align*}
    &\quad (Z^{1}(t)-Z^{2}(t)) \cdot (x_{1}-x_{2})\nonumber\\
    &=\int_{t}^{T}\Bigg\{\mathbb{E} \bigg[\Big(l_{v}(X^{1}(s),u^{1}(s))-l_{v}(X^{2}(s),u^{2}(s))\Big) \cdot (u^{1}(s)-u^{2}(s))\nonumber\\
    &\qquad\qquad\quad +\bigg(l_{x}(X^{1}(s),u^{1}(s))-l_{x}(X^{2}(s),u^{2}(s)) +\dfrac{dF}{d\nu}((X_{\cdot mt}(s)\otimes m)^{\mathcal{B}_{t}^{s}})(X^{1}(s))\nonumber\\
    &\qquad\qquad\qquad\quad -\dfrac{dF}{d\nu}((X_{\cdot mt}(s)\otimes m)^{\mathcal{B}_{t}^{s}})(X^{2}(s))\bigg) \cdot(X^{1}(s)-X^{2}(s))\bigg]\Bigg\}ds\nonumber\\
    &\quad \,+\mathbb{E}\bigg(h_{x}(X^{1}(T))-h_{x}(X^{2}(T))+\dfrac{dF_{T}}{d\nu}((X_{\cdot mt}(T)\otimes m)^{\mathcal{B}_{t}^{T}})(X^{1}(T))\nonumber\\
    &\qquad\quad \,\,\, -\dfrac{dF_T}{d\nu}((X_{\cdot mt}(T)\otimes m)^{\mathcal{B}_{t}^{T}})(X^{2}(T))\cdot(X^{1}(T)-X^{2}(T))\bigg)\nonumber\\
    &\quad \,+\lambda \mathbb{E}\left(\int_{t}^{T}|u^{1}(s)-u^{2}(s)|^{2}ds\right)-(c'_{l}+c')\mathbb{E}\left(\int_{t}^{T}|X^{1}(s)-X^{2}(s)|^{2}ds\right)-(c'_{h}+c'_{T})\mathbb{E}\left(|X^{1}(T)-X^{2}(T)|^{2}\right).\nonumber
\end{align*}






Since 
\[
X^{1}(s)-X^{2}(s)=x_{1}-x_{2}+\int_{t}^{s}(u^{1}(\tau)-u^{2}(\tau))d\tau,
\]
it is straightforward to complete the calculations, and thanks to (\ref{eq:3-6})
we obtain (\ref{eq:4-21}) and then (\ref{eq:4-17}). So $\Psi(x,m,t)$
is Lipschitz in $x.$ Let us check that it is Lipschitz in $m.$ We
give the main results but skip the details. The key point is to establish
the following formula:
\begin{align}
    &\quad \Psi(x,m',t)-\Psi(x,m,t)\nonumber\\
    &= \mathbb{E}\Bigg(\int_{t}^{T}\int_{0}^{1}\int_{0}^{1}\Bigg[\theta l_{xx}\Big(\delta_{(\theta\mu)}X_{x\widehat{X}_{m}t}(s),\delta_{(\theta\mu)}u_{x\widehat{X}_{m}t}(s)\Big)\Big(X_{x\widehat{X}_{m'\,t}}(s)-X_{x\widehat{X}_{m}t}(s)\Big)\cdot\Big(X_{x\widehat{X}_{m'\,t}}(s)-X_{x\widehat{X}_{m}t}(s)\Big)\nonumber\\
    &\qquad\qquad +2\theta l_{xv}\Big(\delta_{(\theta\mu)} X_{x\widehat{X}_{m}t}(s), \delta_{(\theta\mu)} u_{x\widehat{X}_{m}t}(s) \Big)\Big(u_{x\widehat{X}_{m'\,t}}(s)-u_{x\widehat{X}_{m}t}(s)\Big)\cdot \Big(X_{x\widehat{X}_{m'\,t}}(s)-X_{x\widehat{X}_{m}t}(s)\Big)\nonumber\\
    &\qquad\qquad +\theta l_{vv}\Big(\delta_{(\theta\mu)} X_{x\widehat{X}_{m}t}(s), \delta_{(\theta\mu)}u_{x\widehat{X}_{m}t}(s)\Big)\Big(u_{x\widehat{X}_{m'\,t}}(s)-u_{x\widehat{X}_{m}t}(s)\Big)\cdot\Big(u_{x\widehat{X}_{m'\,t}}(s)-u_{x\widehat{X}_{m}t}(s)\Big)\nonumber\\
    &\qquad\qquad +\theta D^{2}\dfrac{dF}{d\nu}\Big((\mathcal{L}_{\widehat{X}_{m}t}(s))^{\mathcal{B}_{t}^{s}}\Big)\Big(\delta_{(\theta\mu)} X_{x\widehat{X}_{m}t}(s)\Big) \Big(X_{x\widehat{X}_{m'\,t}}(s)-X_{x\widehat{X}_{m}t}(s)\Big)\cdot\Big(X_{x\widehat{X}_{m'\,t}}(s)-X_{x\widehat{X}_{m}t}(s)\Big)\Bigg]d\theta d\mu ds\Bigg)\nonumber\\
    &\quad \, +\mathbb{E}\Bigg(\int_{0}^{1}\int_{0}^{1}\Bigg[\theta h_{xx}\Big(\delta_{(\theta\mu)} X_{x\widehat{X}_{m}t}(T)\Big)\Big(X_{x\widehat{X}_{m'\,t}}(T)-X_{x\widehat{X}_{m}t}(T)\Big)\cdot\Big(X_{x\widehat{X}_{m'\,t}}(T)-X_{x\widehat{X}_{m}t}(T)\Big)\nonumber\\
    &\qquad\qquad+\theta D^{2}\dfrac{dF_{T}}{d\nu}\Big((\mathcal{L}_{\widehat{X}_{m}t}(T))^{\mathcal{B}_{t}^{T}}\Big)\Big(\delta_{(\theta\mu)} X_{x\widehat{X}_{m}t}(T)\Big) \Big(X_{x\widehat{X}_{m'\,t}}(T)-X_{x\widehat{X}_{m}t}(T)\Big)\cdot\Big(X_{x\widehat{X}_{m'\,t}}(T)-X_{x\widehat{X}_{m}t}(T)\Big)\Bigg]d\theta d\mu\Bigg)\nonumber\\
&\quad \, +\mathbb{E}\left(\int_{t}^{T}\int_{0}^{1}\int_{0}^{1}\left[\mathbb{E}^{1\mathcal{B}_{t}^{s}} \left(D_{1}\dfrac{d^{2}F}{d\nu^{2}}\Big(\delta_{(\theta)} \mathcal{L}_{\widehat{X}_{m}t}(s)\Big)\Big(X_{x\widehat{X}_{m'\,t}}(s), \delta_{(\mu)}X_{\widehat{X}_{m}^{1}t}^{1}(s)\Big)\cdot\Big(X_{\widehat{X}_{m'}^{1}t}^{1}(s)-X_{\widehat{X}_{m}^{1}t}^{1}(s)\Big)\right)\right]d\theta d\mu ds\right)\nonumber\\
&\quad \, +\mathbb{E}\left(\int_{0}^{1}\int_{0}^{1}\left[\mathbb{E}^{1\mathcal{B}_{t}^{T}} \left(D_{1}\dfrac{d^{2}F_{T}}{d\nu^{2}}\Big(\delta_{(\theta)}\mathcal{L}_{\widehat{X}_{m}t}(T)\Big) \Big(X_{x\widehat{X}_{m'\,t}}(T), \delta_{(\mu)}X_{\widehat{X}_{m}^{1}t}^{1}(T)\Big) \cdot \Big(X_{\widehat{X}_{m'}^{1}t}^{1}(T)-X_{\widehat{X}_{m}^{1}t}^{1}(T)\Big)\right)\right]d\theta d\mu\right), \label{eq:4-22}
\end{align}
where
\begin{equation}
\begin{split}
\delta_{(\theta\mu)}X_{x\widehat{X}_{m}t}(s) &:= X_{x\widehat{X}_{m}t}(s) + \theta\mu\Big(X_{x\widehat{X}_{m'\,t}}(s) - X_{x\widehat{X}_{m}t}(s)\Big),\\
\delta_{(\theta\mu)}u_{x\widehat{X}_{m}t}(s) &:= u_{x\widehat{X}_{m}t}(s) + \theta\mu \Big(u_{x\widehat{X}_{m'\,t}}(s) - u_{x\widehat{X}_{m}t}(s)\Big),\\
\delta_{(\theta)}\mathcal{L}_{\widehat{X}_{m}t}(s) &:= (\mathcal{L}_{\widehat{X}_{m}t}(s))^{\mathcal{B}_{t}^{s}} + \theta \Big((\mathcal{L}_{\widehat{X}_{m'}t}(s))^{\mathcal{B}_{t}^{s}} - (\mathcal{L}_{\widehat{X}_{m}t}(s))^{\mathcal{B}_{t}^{s}}\Big),\\
\delta_{(\mu)}X_{\widehat{X}_{m}^{1}t}^{1}(s) &:= X_{\widehat{X}_{m}^{1}t}^{1}(s) + \mu\Big(X_{\widehat{X}_{m'}^{1}t}^{1}(s) - X_{\widehat{X}_{m}^{1}t}^{1}(s)\Big).
\end{split}
\end{equation}
Interchanging $m$ with $m'$ and adding up, we obtain, using the assumptions
of Section \ref{subsec:ASSUMPTIONS},
\begin{equation}\label{eq:4-23}
    \begin{aligned}
    0 &\geq \mathbb{E}\left(\int_{t}^{T}\left(\lambda \left|u_{x\,\widehat{X}_{m'}t}(s)-u_{x\widehat{X}_{m}t}(s)\right|^{2}-(c'_{l}+c') \left|X_{x\widehat{X}_{m'\,t}}(s)-X_{x\widehat{X}_{m}t}(s)\right|^{2}\right)ds\right)\\
    &\qquad -(c'_{h}+c'_{T})\mathbb{E}\left(\left|X_{x\widehat{X}_{m'\,t}}(T)-X_{x\widehat{X}_{m}t}(T)\right|^{2}\right)\\
    &\qquad + \mathbb{E}\Bigg(\int_{t}^{T}\int_{0}^{1}\int_{0}^{1}\int_{0}^{1}\mathbb{E}^{1\mathcal{B}_{t}^{s}} \bigg(DD_{1}\dfrac{d^{2}F}{d\nu^{2}} \Big(\delta_{(\theta)} \mathcal{L}_{\widehat{X}_{m}t}(s)\Big) \Big(\delta_{(\nu)} X_{x\widehat{X}_{m}t}(s),\delta_{(\mu)} X_{\widehat{X}_{m}^{1}t}^{1}(s)\Big) \\
    &\qquad \qquad \qquad \Big(X_{x\widehat{X}_{m'\,t}}(s)-X_{x\widehat{X}_{m}t}(s)\Big) \cdot\Big(X_{\widehat{X}_{m'}^{1}t}^{1}(s)-X_{\widehat{X}_{m}^{1}t}^{1}(s)\Big)\bigg)d\theta d\mu d\nu ds\Bigg)\\
&\qquad +\mathbb{E}\Bigg(\int_{0}^{1}\int_{0}^{1}\int_{0}^{1}\mathbb{E}^{1\mathcal{B}_{t}^{T}} \bigg(DD_{1}\dfrac{d^{2}F}{d\nu^{2}}\Big(\delta_{(\theta)}\mathcal{L}_{\widehat{X}_{m}t}(T)\Big) \Big(\delta_{(\nu)} X_{x\widehat{X}_{m}t}(T), \delta_{(\mu)}X_{\widehat{X}_{m}^{1}t}^{1}(T)\Big)\\
&\qquad\qquad\qquad \Big(X_{x\widehat{X}_{m'}t}(T)-X_{x\widehat{X}_{m}t}(T)\Big) \cdot \Big(X_{\widehat{X}_{m'}^{1}t}^{1}(T)-X_{\widehat{X}_{m}^{1}t}^{1}(T)\Big)\bigg)d\theta d\mu d\nu\Bigg).
\end{aligned}
\end{equation}
Using 
\begin{align*}
    \mathbb{E} \left(\int_{t}^{T} \left|X_{x\widehat{X}_{m'}t}(s)-X_{x\widehat{X}_{m}t}(s)\right|^{2}ds\right)&\leq\dfrac{(T-t)^{2}}{2}\mathbb{E}\left(\int_{t}^{T} \left|u_{x\widehat{X}_{m'}t}(s)-u_{x\widehat{X}_{m}t}(s)\right|^{2}ds\right), \\
    \mathbb{E} \left(\left|X_{x\widehat{X}_{m'}t}(T)-X_{x\widehat{X}_{m}t}(T)\right|^{2}\right)&\leq(T-t)\mathbb{E}\left(\int_{t}^{T} \left|u_{x\widehat{X}_{m'}t}(s)-u_{x\widehat{X}_{m}t}(s)\right|^{2}ds\right),
\end{align*}
 we deduce after simple estimates 
\begin{equation}
\mathbb{E}\left(\int_{t}^{T} \left|u_{x\widehat{X}_{m'}t}(s)-u_{x\widehat{X}_{m}t}(s)\right|^{2}ds\right)\leq C_{T}\mathbb{E} \left(\left|\widehat{X}_{m'}-\widehat{X}_{m}\right|^{2}\right)=C_{T}W_{2}^{2}(m,m'),\label{eq:4-24}
\end{equation}
\begin{equation}
\sup_{s\in(t,T)}\mathbb{E} \left(\left|X_{x\widehat{X}_{m'}t}(s)-X_{x\widehat{X}_{m}t}(s)\right|^{2}\right),\sup_{s\in(t,T)}\mathbb{E} \left(\left|u_{x\widehat{X}_{m'}t}(s)-u_{x\widehat{X}_{m}t}(s)\right|^{2}\right)\leq C_{T}W_{2}^{2}(m,m'), \text{ and}\label{eq:4-25}
\end{equation}
\begin{equation}
\sup_{s\in(t,T)}\mathbb{E} \left(\left|X_{x\widehat{X}_{m'}t}(s)\right|^{2}\right)\leq C_{T}(1+W_{2}^{2}(m,\delta)).\label{eq:4-250}
\end{equation}
Collecting estimates, we can assert that 
\begin{equation}
|\Psi(x,m',t)-\Psi(x,m,t)|\leq C_{T}W_{2}^{2}(m,m')+K_{T}(m)W_{2}(m,m'),\label{eq:4-251}
\end{equation}
which completes the proof. 
\end{proof}

\section{PROOFS FROM SECTION \ref{sec:2nd order diff}} \label{appendix C}

\subsection{PROOF OF PROPOSITION \ref{prop7-10}}

The control problem (\ref{eq:4-179}), (\ref{eq:4-180}) is
linear quadratic. One can check that it is strictly convex and has
a unique minimum. The optimal control is linear in $y$ and can be
written $\mathcal{U}_{xmt}(s)y$, where $\mathcal{U}_{xmt}(s)$ is
defined by (\ref{eq:4-175}), (\ref{eq:4-176}), (\ref{eq:4-177}).
Define
\[
u_{xymt}^{\epsilon}(s)=\dfrac{u_{x+\epsilon y,mt}(s)-u_{xmt}(s)}{\varepsilon},\,X_{xymt}^{\epsilon}(s)=\dfrac{X_{x+\epsilon y,mt}(s)-X_{xmt}(s)}{\varepsilon},
\]
\begin{equation}
Z_{xymt}^{\epsilon}(s)=\dfrac{Z_{x+\epsilon y,mt,t}(s)-Z_{xmt}(s)}{\epsilon}\ ,r_{xymt}^{j,\epsilon}(s)=\dfrac{r_{x+\epsilon y,mt}^{j}(s)-r_{xmt}^{j}(s)}{\epsilon},\label{eq:8-5}
\end{equation}
\[
\rho_{xymt}^{j,\epsilon}(s)=\dfrac{\rho_{x+\epsilon y,mt}^{j}(s)-\rho_{xmt}^{j}(s)}{\epsilon},
\]
then we have the relations:
\begin{equation}
X_{xymt}^{\epsilon}(s)=y+\int_{t}^{s}u_{xymt}^{\epsilon}(\tau)d\tau,\label{eq:8-6}
\end{equation}
\begin{align}
    &\int_{0}^{1} \bigg[l_{vx}\Big(X_{xmt}(s)+\theta\epsilon X_{xymt}^{\epsilon}(s),u_{xmt}(s)+\theta\epsilon u_{xymt}^{\epsilon}(s)\Big)X_{xymt}^{\epsilon}(s)\nonumber\\
    &\qquad +l_{vv}\Big(X_{xmt}(s)+\theta\epsilon X_{xymt}^{\epsilon}(s),u_{xmt}(s)+\theta\epsilon u_{xymt}^{\epsilon}(s)\Big)u_{xymt}^{\epsilon}(s)\bigg]d\theta+Z_{xymt}^{\epsilon}(s)=0,\label{eq:8-7}
\end{align}
\begin{align}
    -dZ_{xymt}^{\epsilon}(s))=&\int_{0}^{1} \bigg[l_{xx}\Big(X_{xmt}(s)+\theta\epsilon X_{xymt}^{\epsilon}(s),u_{xmt}(s)+\theta\epsilon u_{xymt}^{\epsilon}(s)\Big)X_{xymt}^{\epsilon}(s)\nonumber\\
    &\qquad +l_{xv}\Big(X_{xmt}(s)+\theta\epsilon X_{xymt}^{\epsilon}(s),u_{xmt}(s)+\theta\epsilon u_{xymt}^{\epsilon}(s)\Big)u_{xymt}^{\epsilon}(s) \bigg]d\theta\label{eq:8-70}\\
    &+\int_{0}^{1}D^{2}\dfrac{d}{d\nu}F \left((X_{\cdot mt}(s)\otimes m)^{\mathcal{B}_{t}^{s}}\right)\Big(X_{xmt}(s)+\theta\epsilon X_{xymt}^{\epsilon}(s)\Big)X_{xymt}^{\epsilon}(s)d\theta\nonumber\\
    &-\sum_{j}r_{xymt}^{j,\epsilon}(s)dw_{j}(s)-\sum_j\rho_{xymt}^{j,\epsilon}(s)db_{j}(s),\nonumber
\end{align}
\begin{equation}
    \begin{aligned}
    Z_{xymt}^{\epsilon}(T)=&\int_{0}^{1}h_{xx}\Big(X_{xmt}(T)+\theta\epsilon X_{xymt}^{\epsilon}(T)\Big)X_{xymt}^{\epsilon}(T)d\theta\\
    &+\int_{0}^{1}D^{2}\dfrac{d}{d\nu}F\left((X_{\cdot mt}(T)\otimes m)^{\mathcal{B}_{t}^{T}}\right)\Big(X_{xmt}(T)+\theta\epsilon X_{xymt}^{\epsilon}(T)\Big)X_{xymt}^{\epsilon}(T)d\theta.
\end{aligned}
\end{equation}

By techniques already used, we can check that $u_{xymt}^{\epsilon}(s),$
$X_{xymt}^{\epsilon}(s),Z_{xymt}^{\epsilon}(s)$ are bounded in $L_{\mathcal{F}_{t}}^{\infty}(t,T;\mathcal{H}_{m})$
and $r_{xymt}^{j,\epsilon}(s), \rho_{xymt}^{j,\epsilon}(s)$ are bounded
in $L_{\mathcal{F}_{t}}^{2}(t,T;\mathcal{H}_{m}).$ 

We consider then subsequences of $u_{xymt}^{\epsilon}(s),$ $X_{xymt}^{\epsilon}(s),Z_{xymt}^{\epsilon}(s)$$,r_{xymt}^{j,\epsilon}(s),\rho_{xymt}^{j,\epsilon}(s)$
which converge weakly in $L_{\mathcal{F}_{t}}^{2}(t,T;\mathcal{H}_{m})$
to $\mathcal{U}_{xymt}(s),\mathcal{X}_{xymt}(s),\mathcal{Z}_{xymt}(s),\mathcal{R}_{xymt}^{j}(s),\Theta_{xymt}^{j}(s).$
It follows easily that the limit is $\mathcal{U}_{xmt}(s)y$, $\mathcal{X}_{xmt}(s)y$, $\mathcal{Z}_{xmt}(s)y$, $\mathcal{R}_{xmt}^{j}(s)y$, $\Theta_{xmt}^{j}(s)y$.
It is possible to show that the convergence is strong using standard techniques. This completes
the proof.

\subsection{PROOF OF PROPOSITION \ref{prop7-100}}

By calculations already done, we can check the relation 
\begin{equation} \label{eq:8-60}
\begin{aligned}
    y\cdot\mathcal{Z}_{xmt}(t)y&=\mathbb{E}\Bigg(\int_{t}^{T}\Bigg\{ \mathcal{U}_{xmt}(s)y\cdot l_{vv}(X_{xmt}(s),u_{xmt}(s))\mathcal{U}_{xmt}(s)y\\
    &\qquad\qquad\quad +2\mathcal{U}_{xmt}(s)y \cdot l_{vx}(X_{xmt}(s),u_{xmt}(s))\mathcal{X}_{xmt}(s)y\\
    &\qquad\qquad\quad +\mathcal{X}_{xmt}(s)y \cdot \bigg(l_{xx}(X_{xmt}(s),u_{xmt}(s))\\
    &\qquad\qquad\quad +D^{2}\dfrac{d}{d\nu}F\left((X_{\cdot mt}(s)\otimes m)^{\mathcal{B}_{t}^{s}}\right)(X_{xmt}(s))\bigg)\mathcal{X}_{xmt}(s)y\Bigg\} ds\Bigg)\\
    &\quad \, +\mathbb{E}\left(\mathcal{X}_{xmt}(T)y \cdot \left(h_{xx}(X_{xmt}(T))+D^{2}\dfrac{d}{d\nu}F\left((X_{\cdot mt}(T)\otimes m)^{\mathcal{B}_{t}^{T}}\right)(X_{xmt}(T))\right)\mathcal{X}_{xmt}(T)y\right).
\end{aligned}  
\end{equation}



Also,
\begin{equation} \label{eq:8-61}
\begin{aligned}
    \mathcal{Z}_{xmt}(t)y &=\mathbb{E}\bigg[\int_{t}^{T} \bigg(\left(l_{xx}(X_{xmt}(s),u_{xmt}(s))+D^{2}\dfrac{d}{d\nu}F\left((X_{\cdot mt}(s)\otimes m)^{\mathcal{B}_{t}^{s}}\right)(X_{xmt}(s))\right)\mathcal{X}_{xmt}(s)y\\
    &\qquad\qquad\quad \, +l_{xv}(X_{xmt}(s),u_{xmt}(s))\mathcal{U}_{xmt}(s)y)\bigg)ds \\
    &\qquad \quad +\left(h_{xx}(X_{xmt}(T))+D^{2}\dfrac{d}{d\nu}F\left((X_{\cdot mt}(T)\otimes m)^{\mathcal{B}_{t}^{T}}\right)(X_{xmt}(T))\right)\mathcal{X}_{xmt}(T)y\bigg].
\end{aligned}
\end{equation}


 From the assumptions and estimates already done, we deduce $\mathbb{E}\left(\int_{t}^{T}|\mathcal{U}_{xmt}(s)y|^{2}ds\right) \leq C_{T}|y|^{2}.$
It easily follows $|\mathcal{Z}_{xmt}(t)y|\leq C_{T}$. This concludes
the proof. 

\subsection{PROOF OF PROPOSITION \ref{prop5-10}}

We connect the system \eqref{eq:5-2002}
to a control problem. The space of controls is $L_{\mathcal{F}_{X_{\cdot}\mathcal{X}_{\cdot}t}}^{2}(t,T;\mathcal{H}_{m})$
where $\mathcal{F}_{X_{\cdot}\mathcal{X}_{\cdot}t}$ is the filtration generated
by the $\sigma$-algebras $\mathcal{F}_{X_{\cdot}\mathcal{X}_{\cdot}t}^{s}$.
If $\mathcal{V}_{X_{\cdot}\mathcal{X}_{\cdot}t}(s)$ is a control, the state
is defined by 
\begin{equation}
\mathcal{X}_{X_{\cdot}\mathcal{X}_{\cdot}t}(s)=\mathcal{X}+\int_{t}^{s}\mathcal{V}_{X_{\cdot}\mathcal{X}_{\cdot}t}(\tau)d\tau,\label{eq:8-1}
\end{equation}
and the payoff is 
\begin{align}
    \mathcal{J}_{X_{\cdot}\mathcal{X}_{\cdot}t}(\mathcal{V}_{X_{\cdot}\mathcal{X}_{\cdot}t}(\cdot)) &=\dfrac{1}{2}\int_{t}^{T} \Bigg[\left\langle l_{xx}(X_{X_{\cdot}t}(s),u_{X_{\cdot}t}(s))\mathcal{X}_{X_{\cdot}\mathcal{X}_{\cdot}t}(s)+D_{X}^{2}\mathbb{E}\left(F\left((X_{X_{\cdot}t}(s)\otimes m)^{\mathcal{B}_{t}^{s}}\right)\right)(\mathcal{X}_{X_{\cdot}\mathcal{X}_{\cdot}t}(s)),\mathcal{X}_{X_{\cdot}\mathcal{X}_{\cdot}t}(s)\right\rangle \nonumber\\
    &\quad \, +2\left\langle l_{xv}(X_{X_{\cdot}t}(s),u_{X_{\cdot}t}(s))\mathcal{V}_{X_{\cdot}\mathcal{X}_{\cdot}t}(s),\mathcal{X}_{X_{\cdot}\mathcal{X}_{\cdot}t}(s)\right\rangle +\left\langle l_{vv}(X_{X_{\cdot}t}(s),u_{X_{\cdot}t}(s))\mathcal{V}_{X_{\cdot}\mathcal{X}_{\cdot}t}(s),\mathcal{\mathcal{V}}_{X_{\cdot}\mathcal{X}_{\cdot}t}(s)\right\rangle \Bigg]ds\nonumber\\
    &\quad \, +\dfrac{1}{2}\left\langle h_{xx}(X_{X_{\cdot}t}(T))\mathcal{X}_{X_{\cdot}\mathcal{X}_{\cdot}t}(T)+D_{X}^{2}\mathbb{E}\left(F_{T}\left((X_{X_{\cdot}t}(T)\otimes m)^{\mathcal{B}_{t}^{T}}\right)\right)(\mathcal{X}_{X_{\cdot}\mathcal{X}_{\cdot}t}(T)),\mathcal{X}_{X_{\cdot}\mathcal{X}_{\cdot}t}(T)\right\rangle.\label{eq:8-2}
\end{align}



Thanks to the assumption (\ref{eq:3-6}) this is a linear quadratic
convex problem, which has a unique optimal control. The system
\eqref{eq:5-2002} has a unique
solution and the optimal control is $\mathcal{U}_{X_{\cdot}\mathcal{X}_{\cdot}t}(s)$. The
optimal state is $\mathcal{\mathcal{X}}_{X_{\cdot}\mathcal{X}_{\cdot}t}(s)$.
Moreover, one has 
\begin{equation}\label{eq:8-3}
\begin{aligned}
    &\quad \inf_{\mathcal{V}_{X_{\cdot}\mathcal{X}_{\cdot}t}(\cdot)}\mathcal{J}_{X_{\cdot}\mathcal{X}_{\cdot}t}(\mathcal{V}_{X_{\cdot}\mathcal{X}_{\cdot}t}(\cdot))=\dfrac{1}{2}\langle \mathcal{Z}_{X_{\cdot}\mathcal{X}_{\cdot}t}(t),\mathcal{X}\rangle \\
    &=\dfrac{1}{2}\int_{t}^{T} \Bigg[\left\langle l_{xx}(X_{X_{\cdot}t}(s),u_{X_{\cdot}t}(s))\mathcal{X}_{X_{\cdot}\mathcal{X}_{\cdot}t}(s)+D_{X}^{2}\mathbb{E}\left(F\left((X_{X_{\cdot}t}(s)\otimes m)^{\mathcal{B}_{t}^{s}}\right)\right)(\mathcal{X}_{X_{\cdot}\mathcal{X}_{\cdot}t}(s)),\mathcal{X}_{X_{\cdot}\mathcal{X}_{\cdot}t}(s)\right\rangle \\
    &\qquad \, +2 \left\langle l_{xv}(X_{X_{\cdot}t}(s),u_{X_{\cdot}t}(s))\mathcal{U}_{X_{\cdot}\mathcal{X}_{\cdot}t}(\tau),\mathcal{X}_{X_{\cdot}\mathcal{X}_{\cdot}t}(s)\right\rangle + \left\langle l_{vv}(X_{X_{_{\cdot}}t}(s),u_{X_{\cdot}t}(s))\mathcal{U}_{X_{\cdot}\mathcal{X}_{\cdot}t}(s),\mathcal{\mathcal{U}}_{X_{\cdot}\mathcal{X}_{\cdot}t}(s)\right\rangle \Bigg]ds\\
    &\quad \, +\dfrac{1}{2} \left\langle h_{xx}(X_{X_{\cdot}t}(T))\mathcal{X}_{X_{\cdot}\mathcal{X}_{\cdot}t}(T)+D_{X}^{2}\mathbb{E}\left(F_{T}\left((X_{X_{\cdot}t}(T)\otimes m)^{\mathcal{B}_{t}^{T}}\right)\right)(\mathcal{X}_{X_{\cdot}\mathcal{X}_{\cdot}t}(T)),\mathcal{X}_{X\mathcal{X}t}(T)\right\rangle,
\end{aligned}
\end{equation}



where in (\ref{eq:8-3}) $\mathcal{U}_{X\mathcal{X}t}(s)$, $\mathcal{Y}_{X\mathcal{X}t}(s),\,\mathcal{Z}_{X\mathcal{X}t}(s)$
is the solution of \eqref{eq:5-2002}.
Thanks to (\ref{eq:3-6}) we can check, by estimates similar to those
already done, that
\begin{equation}
||\mathcal{U}_{X_{\cdot}\mathcal{X}_{\cdot}t}(s)||,\ ||\mathcal{X}_{X_{\cdot}\mathcal{X}_{\cdot}t}(s)||,\ ||\mathcal{Z}_{X_{\cdot}\mathcal{X}_{\cdot}t}(s)||\leq C_{T}||\mathcal{X}_{\cdot}||.\label{eq:8-4}
\end{equation}
 Now $\mathcal{X}\mapsto\mathcal{U}_{X_{\cdot}\mathcal{X}_{\cdot}t}(s),\mathcal{X}_{X_{\cdot}\mathcal{X}_{\cdot}t}(s),\mathcal{Z}_{X_{\cdot}\mathcal{X}_{\cdot}t}(s)$
are linear. Indeed, taking an initial condition $\alpha\mathcal{X}_{1}+\beta\mathcal{X}_{2}$,
one can use the $\sigma$-algebras $\mathcal{F}_{X_{\cdot}\mathcal{X}_{1\cdot}\mathcal{X}_{2\cdot}t}^{s}$
and the linearity follows easily. Therefore the maps $\mathcal{X}\mapsto\mathcal{U}_{X_{\cdot}\mathcal{X}_{\cdot}t}(s)$, $\mathcal{X}_{X_{\cdot}\mathcal{X}_{\cdot}t}(s)$,
$\mathcal{Z}_{X_{\cdot}\mathcal{X}_{\cdot}t}(s)$ belong to $\mathcal{L}(\mathcal{H}_{m},\mathcal{H}_{m}).$
Define
\[
u_{X_{\cdot}\mathcal{X}_{\cdot}t}^{\epsilon}(s):=\dfrac{u_{X+\epsilon\mathcal{X},t}(s)-u_{Xt}(s)}{\varepsilon},\,X_{X_{\cdot}\mathcal{X}_{\cdot}t}^{\epsilon}(s):=\dfrac{X_{X+\epsilon\mathcal{X},t}(s)-X_{Xt}(s)}{\varepsilon},
\]
\begin{equation}
Z_{X_{\cdot}\mathcal{X}_{\cdot}t}^{\epsilon}(s):=\dfrac{Z_{X+\epsilon\mathcal{X},t}(s)-Z_{Xt}(s)}{\epsilon}\ ,r_{X_{\cdot}\mathcal{X}_{\cdot}t}^{j,\epsilon}(s):=\dfrac{r_{X+\epsilon\mathcal{X},t}^{j}(s)-r_{Xt}^{j}(s)}{\epsilon},\label{eq:8-5-1}
\end{equation}
\[
\rho_{X_{\cdot}\mathcal{X}_{\cdot}t}^{j,\epsilon}(s):=\dfrac{\rho_{X+\epsilon\mathcal{X},t}^{j}(s)-\rho_{Xt}^{j}(s)}{\epsilon},
\]
then we have the relations: 
\begin{equation}
X_{X_{\cdot}\mathcal{X}_{\cdot}t}^{\epsilon}(s)=\mathcal{X}+\int_{t}^{s}u_{X_{\cdot}\mathcal{X}_{\cdot}t}^{\epsilon}(\tau)d\tau,\label{eq:8-6-1}
\end{equation}
\begin{align}
    &\int_{0}^{1} \bigg[l_{vx}\left(X_{X_{\cdot}t}(s)+\theta\epsilon X_{X_{\cdot}\mathcal{X}_{\cdot}t}^{\epsilon}(s),u_{X_{\cdot}t}(s)+\theta\epsilon u_{X_{\cdot}\mathcal{X}_{\cdot}t}^{\epsilon}(s) \right)X_{X_{\cdot}\mathcal{X}_{\cdot}t}^{\epsilon}(s)\nonumber\\
    &\qquad+l_{vv}\left(X_{X_{\cdot}t}(s)+\theta\epsilon X_{X_{\cdot}\mathcal{X}_{\cdot}t}^{\epsilon}(s),u_{X_{\cdot}t}(s)+\theta\epsilon u_{X_{\cdot}\mathcal{X}_{\cdot}t}^{\epsilon}(s)\right)u_{X_{\cdot}\mathcal{X}_{\cdot}t}^{\epsilon}(s)\bigg]d\theta+Z_{X_{\cdot}\mathcal{X}_{\cdot}t}^{\epsilon}(s)=0,\label{eq:8-7-1}
\end{align}
\begin{align}
    -dZ_{X_{\cdot}\mathcal{X}_{\cdot}t}^{\epsilon}(s) &=\int_{0}^{1} \bigg[l_{xx}(X_{X_{\cdot}t}(s)+\theta\epsilon X_{X_{\cdot}\mathcal{X}_{\cdot}t}^{\epsilon}(s),u_{X_{\cdot}t}(s)+\theta\epsilon u_{X_{\cdot}\mathcal{X}_{\cdot}t}^{\epsilon}(s))X_{X_{\cdot}\mathcal{X}_{\cdot}t}^{\epsilon}(s)\nonumber\\
    &\qquad\quad \, +l_{xv}(X_{X_{\cdot} t}(s)+\theta\epsilon X_{X_{\cdot}\mathcal{X}_{\cdot}t}^{\epsilon}(s),u_{X_{\cdot}t}(s)+\theta\epsilon u_{X_{\cdot}\mathcal{X}_{\cdot}t}^{\epsilon}(s))u_{X_{\cdot}\mathcal{X}_{\cdot}t}^{\epsilon}(s) \bigg]d\theta\label{eq:8-70-1}\\
    &\quad \,\, +\int_{0}^{1}D_{X}^{2}\mathbb{E}\left(F\left(((X_{X_{\cdot} t}(s)+\theta\epsilon X_{X_{\cdot}\mathcal{X}_{\cdot}t}^{\epsilon}(s))\otimes m)^{\mathcal{B}_{t}^{s}}\right)\right)(X_{X\mathcal{X}t}^{\epsilon}(s))d\theta \nonumber \\
    &\quad \,\, -\sum_{j}r_{X_{\cdot}\mathcal{X}_{\cdot}t}^{j,\epsilon}(s)dw_{j}(s)-\sum_{j}\rho_{X_{\cdot}\mathcal{X}_{\cdot}t}^{j,\epsilon}(s)db_{j}(s),\nonumber
\end{align}
\begin{align}
    Z_{X\mathcal{X}t}^{\epsilon}(T)=&\int_{0}^{1}h_{xx}(X_{X_{\cdot}t}(T)+\theta\epsilon X_{X_{\cdot}\mathcal{X}_{\cdot}t}^{\epsilon}(T))X_{X_{\cdot}\mathcal{X}_{\cdot}t}^{\epsilon}(T)d\theta\nonumber\\
    &+\int_{0}^{1}D_{X}^{2}\mathbb{E}\left(F_{T} \left(((X_{X_{\cdot}t}(T)+\theta\epsilon X_{X_{\cdot}\mathcal{X}_{\cdot}t}^{\epsilon}(T))\otimes m)^{\mathcal{B}_{t}^{T}}\right)\right)(X_{X\mathcal{X}t}^{\epsilon}(T))d\theta.\nonumber
\end{align}
By techniques already used, we can check that $u_{X_{\cdot}\mathcal{X}_{\cdot}t}^{\epsilon}(s),$
$X_{X_{\cdot}\mathcal{X}_{\cdot}t}^{\epsilon}(s),Z_{X_{\cdot}\mathcal{X}_{\cdot}t}^{\epsilon}(s)$
are bounded in $L_{\mathcal{F}_{X_{\cdot}\mathcal{X}_{\cdot}t}}^{\infty}(t,T;\mathcal{H}_{m})$
and $r_{X_{\cdot}\mathcal{X}_{\cdot}t}^{j,\epsilon}(s),\rho_{X_{\cdot}\mathcal{X}_{\cdot}t}^{j,\epsilon}(s)$
are bounded in $L_{\mathcal{F}_{X\mathcal{X}t}}^{2}(t,T;\mathcal{H}_{m}).$ 

We consider then subsequences of $u_{X_{\cdot}\mathcal{X}_{\cdot}t}^{\epsilon}(s),X_{X_{\cdot}\mathcal{X}_{\cdot}t}^{\epsilon}(s),$
$Z_{X_{\cdot}\mathcal{X}_{\cdot}t}^{\epsilon}(s),r_{X_{\cdot}\mathcal{X}_{\cdot}t}^{j,\epsilon}(s)$,
$\rho_{X_{\cdot}\mathcal{X}_{\cdot}t}^{j,\epsilon}(s)$ which converge weakly
in $L_{\mathcal{F}_{X_{\cdot}\mathcal{X}_{\cdot}t}}^{2}(t,T;\mathcal{H}_{m})$
to $\mathcal{U}_{X_{\cdot}\mathcal{X}_{\cdot}t}(s),\mathcal{X}_{X_{\cdot}\mathcal{X}_{\cdot}t}(s),\mathcal{Z}_{X_{\cdot}\mathcal{X}_{\cdot}t}(s),\mathcal{R}_{X_{\cdot}\mathcal{X}_{\cdot}t}^{j}(s),\Theta_{X_{\cdot}\mathcal{X}_{\cdot}t}^{j}(s).$
It follows immediately that 
\begin{equation}
\mathcal{X}_{X_{\cdot}\mathcal{X}_{\cdot}t}(s)=\mathcal{X}+\int_{t}^{s}\mathcal{U}_{X_{\cdot}\mathcal{X}_{\cdot}t}(\tau)d\tau.\label{eq:8-71}
\end{equation}
 Define 
\begin{align}
   L_{X_{\cdot}\mathcal{X}_{\cdot}t}^{\epsilon}(s) :=&\int_{0}^{1} \bigg[l_{vx} \Big(X_{X_{\cdot}t}(s)+\theta\epsilon X_{X_{\cdot}\mathcal{X}_{\cdot}t}^{\epsilon}(s),u_{X_{\cdot}t}(s)+\theta\epsilon u_{X_{\cdot}\mathcal{X}_{\cdot}t}^{\epsilon}(s)\Big)X_{X_{\cdot}\mathcal{X}_{\cdot}t}^{\epsilon}(s)\nonumber\\
    &\qquad +l_{vv}\Big(X_{X_{\cdot}t}(s)+\theta\epsilon X_{X_{\cdot}\mathcal{X}_{\cdot}t}^{\epsilon}(s),u_{X_{\cdot}t}(s)+\theta\epsilon u_{X_{\cdot}\mathcal{X}_{\cdot}t}^{\epsilon}(s)\Big)u_{X_{\cdot}\mathcal{X}_{\cdot}t}^{\epsilon}(s)\bigg]d\theta.\nonumber
\end{align}
We want to show that 
\begin{equation}
L_{X_{\cdot}\mathcal{X}_{\cdot}t}^{\epsilon}(s)\rightarrow L_{X_{\cdot}\mathcal{X}_{\cdot}t}(s)=l_{vx}(X_{X_{\cdot}t}(s),u_{X_{\cdot}t}(s))\mathcal{X}_{X_{\cdot}\mathcal{X}_{\cdot}t}(s)+l_{vv}(X_{X_{\cdot}t}(s),u_{X_{\cdot}t}(s))\mathcal{U}_{X_{\cdot}\mathcal{X}_{\cdot}t}(s)\label{eq:8-72}
\end{equation}
in $L_{\mathcal{F}_{X_{\cdot}\mathcal{X}_{\cdot}t}}^{2}(t,T;\mathcal{H}_{m})$
weakly as $\epsilon \rightarrow 0$. Let us take $\Gamma_{X_{\cdot}\mathcal{X}_{\cdot}t}(s)$ in $L_{\mathcal{F}_{X_{\cdot}\mathcal{X}_{\cdot}t}}^{2}(t,T;\mathcal{H}_{m}).$
We have: 
\begin{align*}
&\quad\int_{t}^{T} \langle L_{X_{\cdot}\mathcal{X}_{\cdot}t}^{\epsilon}(s),\Gamma_{X\mathcal{X}t}(s)\rangle ds \\
&= \int_{t}^{T} \left\langle X_{X_{\cdot}\mathcal{X}_{\cdot}t}^{\epsilon}(s),\int_{0}^{1}l_{xv} \Big(X_{Xt}(s)+\theta\epsilon X_{X\mathcal{X}t}^{\epsilon}(s),u_{Xt}(s)+\theta\epsilon u_{X_{\cdot}\mathcal{X}_{\cdot}t}^{\epsilon}(s)\Big)\Gamma_{X_{\cdot}\mathcal{X}_{\cdot}t}(s)d\theta \right\rangle ds\\
&\quad+\int_{t}^{T} \left\langle u_{X_{\cdot}\mathcal{X}_{\cdot}t}^{\epsilon}(s),\int_{0}^{1}l_{vv} \Big(X_{Xt}(s)+\theta\epsilon X_{X_{\cdot}\mathcal{X}_{\cdot}t}^{\epsilon}(s),u_{Xt}(s)+\theta\epsilon u_{X_{\cdot}\mathcal{X}_{\cdot}t}^{\epsilon}(s) \Big)\Gamma_{X_{\cdot}\mathcal{X}_{\cdot}t}(s)d\theta \right\rangle ds.
\end{align*}

Since $X_{Xt}(\cdot)+\theta\epsilon X_{X_{\cdot}\mathcal{X}_{\cdot}t}^{\epsilon}(\cdot)\rightarrow X_{Xt}(\cdot)$,$\ u_{Xt}(\cdot)+\theta\epsilon u_{X_{\cdot}\mathcal{X}_{\cdot}t}^{\epsilon}(\cdot) \rightarrow u_{Xt}(\cdot)$
in $L_{\mathcal{F}_{X\mathcal{X}t}}^{2}(t,T;\mathcal{H}_{m})$, and
$l_{xx}$ is continuous and bounded by assumptions (\ref{eq:4-1000}),
(\ref{eq:4-1002}), it follows by classical Lebesgue integral theory
that 
\[
\int_{0}^{1}l_{xv}\Big(X_{Xt}(s)+\theta\epsilon X_{X_{\cdot}\mathcal{X}_{\cdot}t}^{\epsilon}(s),u_{Xt}(s)+\theta\epsilon u_{X_{\cdot}\mathcal{X}_{\cdot}t}^{\epsilon}(s)\Big)\Gamma_{X_{\cdot}\mathcal{X}_{\cdot}t}(s)d\theta\rightarrow l_{xv}\Big(X_{Xt}(s),u_{Xt}(s)\Big)\Gamma_{X_{\cdot}\mathcal{X}_{\cdot}t}(s)\ \text{ in}\ L_{\mathcal{F}_{X\mathcal{X}t}}^{2}(t,T;\mathcal{H}_{m}),
\]
and a similar argument applies to $\int_{0}^{1}l_{vv}(X_{Xt}(s)+\theta\epsilon X_{X_{\cdot}\mathcal{X}_{\cdot}t}^{\epsilon}(s),u_{Xt}(s)+\theta\epsilon u_{X_{\cdot}\mathcal{X}_{\cdot}t}^{\epsilon}(s))\Gamma_{X_{\cdot}\mathcal{X}_{\cdot}t}(s)d\theta.$
Therefore 
\begin{align}
    \int_{t}^{T} \Big\langle L_{X_{\cdot}\mathcal{X}_{\cdot}t}^{\epsilon}(s),\Gamma_{X_{\cdot}\mathcal{X}_{\cdot}t}(s)\Big\rangle ds\rightarrow&\int_{t}^{T}\Big\langle \mathcal{X}_{X\mathcal{X}t}(s),l_{xv}(X_{Xt}(s),u_{Xt}(s))\Gamma_{X_{\cdot}\mathcal{X}_{\cdot}t}(s)\Big\rangle ds\nonumber\\
    &+\int_{t}^{T} \Big\langle \mathcal{U}_{X_{\cdot}\mathcal{X}_{\cdot}t}(s),l_{vv}(X_{Xt}(s),u_{Xt}(s))\Gamma_{X_{\cdot}\mathcal{X}_{\cdot}t}(s) \Big\rangle ds.\nonumber
\end{align}

Transposing the matrices $l_{xv}$ and $l_{vv}$ we obtain 
\[
\int_{t}^{T} \Big\langle L_{X_{\cdot}\mathcal{X}_{\cdot}t}^{\epsilon}(s),\Gamma_{X_{\cdot}\mathcal{X}_{\cdot}t}(s)\Big\rangle ds\rightarrow\int_{t}^{T}\Big\langle L_{X_{\cdot}\mathcal{X}_{\cdot}t}(s),\Gamma_{X_{\cdot}\mathcal{X}_{\cdot}t}(s)\Big \rangle ds,
\]
and since $\Gamma_{X\mathcal{X}t}(s)$ is arbitrary, the result (\ref{eq:8-72})
follows. So (\ref{eq:8-7-1}) yields:
\begin{equation}
l_{vx}(X_{Xt}(s),u_{Xt}(s))\mathcal{X}_{X\mathcal{X}t}(s)+l_{vv}(X_{Xt}(s),u_{Xt}(s))\mathcal{U}_{X_{\cdot}\mathcal{X}_{\cdot}t}(s)+\mathcal{Z}_{X_{\cdot}\mathcal{X}_{\cdot}t}(s)=0.\label{eq:8-73}
\end{equation}
A similar argument implies:
\begin{align}
&\quad \int_{0}^{1} \bigg[l_{xx}\Big(X_{Xt}(s)+\theta\epsilon X_{X\mathcal{X}t}^{\epsilon}(s),u_{Xt}(s)+\theta\epsilon u_{X_{\cdot}\mathcal{X}_{\cdot}t}^{\epsilon}(s)\Big)X_{X_{\cdot}\mathcal{X}_{\cdot}t}^{\epsilon}(s)\label{eq:8-74}\\
&\qquad\quad +l_{xv}\Big(X_{Xt}(s)+\theta\epsilon X_{X_{\cdot}\mathcal{X}_{\cdot}t}^{\epsilon}(s),u_{Xt}(s)+\theta\epsilon u_{X_{\cdot}\mathcal{X}_{\cdot}t}^{\epsilon}(s)\Big)u_{X_{\cdot}\mathcal{X}_{\cdot}t}^{\epsilon}(s) \bigg]d\theta\nonumber\\
&\rightarrow l_{xx}\Big(X_{Xt}(s),u_{Xt}(s)\Big)\mathcal{X}_{X_{\cdot}\mathcal{X}_{\cdot}t}(s)+l_{xv}\Big(X_{Xt}(s),u_{Xt}(s)\Big)\mathcal{U}_{X_{\cdot}\mathcal{X}_{\cdot}t}(s)\ \ \text{in }L_{\mathcal{F}_{X\mathcal{X}t}}^{2}(t,T;\mathcal{H}_{m})\ \text{weakly}.\nonumber
\end{align}

Now, the property (\ref{eq:3-157}) holds true, so we can also state
that 
\begin{align}
    &\quad \int_{0}^{1}D_{X}^{2}\mathbb{E} \left(F \left(((X_{Xt}(s)+\theta\epsilon X_{X_{\cdot}\mathcal{X}_{\cdot}t}^{\epsilon}(s))\otimes m)^{\mathcal{B}_{t}^{s}}\right)\right)(X_{X_{\cdot}\mathcal{X}_{\cdot}t}^{\epsilon}(s))d\theta\nonumber\\
    &\rightarrow D_{X}^{2}\mathbb{E}\left(F\left((X_{X_{\cdot}t}(s)\otimes m)^{\mathcal{B}_{t}^{s}}\right)\right)(\mathcal{X}_{X_{\cdot}\mathcal{X}_{\cdot}t}(s))\ \text{in }L_{\mathcal{F}_{X\mathcal{X}t}}^{2}(t,T;\mathcal{H}_{m})\ \text{weakly}.\label{eq:8-75}
\end{align}
Also,
\begin{align}
    &\quad \int_{0}^{1}h_{xx}(X_{Xt}(T)+\theta\epsilon X_{X_{\cdot}\mathcal{X}_{\cdot}t}^{\epsilon}(T))X_{X_{\cdot}\mathcal{X}_{\cdot}t}^{\epsilon}(T)d\theta+\int_{0}^{1}D_{X}^{2}\mathbb{E}\left(F_{T} \left(((X_{Xt}(T)+\theta\epsilon X_{X_{\cdot}\mathcal{X}_{\cdot}t}^{\epsilon}(T))\otimes m)^{\mathcal{B}_{t}^{T}}\right)\right)(X_{X_{\cdot}\mathcal{X}_{\cdot}t}^{\epsilon}(T))d\theta\nonumber\\
    &\rightarrow h_{xx}(X_{Xt}(T))\mathcal{X}_{X_{\cdot}\mathcal{X}_{\cdot}t}(T)+D_{X}^{2}\mathbb{E} \left(F_{T} \left((X_{X_{\cdot}t}(T)\otimes m)^{\mathcal{B}_{t}^{T}}\right)\right)(\mathcal{X}_{X_{\cdot}\mathcal{X}_{\cdot}t}(T)) \label{eq:8-76}
\end{align}
weakly in the subspace of \ensuremath{\mathcal{H}_{m}} of $\mathcal{F}_{X_{\cdot}\mathcal{X}_{\cdot}t}^{T}$-measurable variable.
From relation (\ref{eq:8-70-1}), by conditioning and letting $\epsilon\rightarrow0$
we obtain 
\begin{align}
    \mathcal{Z}_{X_{\cdot}\mathcal{X}_{\cdot}t}(s) &=\mathbb{E} \Bigg[\int_{s}^{T} \bigg(l_{xx}(X_{Xt}(\tau),u_{Xt}(\tau))\mathcal{X}_{X_{\cdot}\mathcal{X}_{\cdot}t}(\tau)+l_{xv}(X_{Xt}(\tau),u_{Xt}(\tau))\mathcal{U}_{X_{\cdot}\mathcal{X}_{\cdot}t}(\tau)\label{eq:8-77}\\
    &\qquad\qquad \quad +D_{X}^{2}\mathbb{E} \left(F \left((X_{X_{\cdot}t}(\tau)\otimes m)^{\mathcal{B}_{t}^{\tau}}\right)\right)(\mathcal{X}_{X_{\cdot}\mathcal{X}_{\cdot}t}(\tau))\bigg)d\tau\nonumber\\
    &\quad \, +h_{xx}(X_{Xt}(T))\mathcal{X}_{X\mathcal{X}t}(T)+D_{X}^{2}\mathbb{E} \left(F_{T}\left((X_{X_{\cdot}t}(T)\otimes m)^{\mathcal{B}_{t}^{T}}\right)\right)(\mathcal{X}_{X_{\cdot}\mathcal{X}_{\cdot}t}(T)) \bigg|\mathcal{F}_{X_{\cdot}\mathcal{X}_{\cdot}t}^{s} \Bigg].\nonumber
\end{align}

This implies that $\mathcal{Z}_{X_{\cdot}\mathcal{X}_{\cdot}t}(s)$ is an
It\^o process, and satisfies 
\begin{equation} \label{eq:8-78}
\left\{
\begin{aligned}
    -d\mathcal{Z}_{X_{\cdot}\mathcal{X}_{\cdot}t}(s) &= \bigg[l_{xx}(X_{Xt}(s),u_{Xt}(s))\mathcal{X}_{X_{\cdot}\mathcal{X}_{\cdot}t}(s)+l_{xv}(X_{Xt}(s),u_{Xt}(s))\mathcal{\mathcal{U}}_{X_{\cdot}\mathcal{X}_{\cdot}t}(s)\\
    &\qquad +D_{X}^{2}\mathbb{E} \left(F\left((X_{X_{\cdot}t}(s)\otimes m)^{\mathcal{B}_{t}^{s}}\right)\right)(\mathcal{X}_{X_{\cdot}\mathcal{X}_{\cdot}t}(s))\bigg]ds\\
    &\quad \, -\sum_{j=1}^{n}\mathcal{R}_{X_{\cdot}\mathcal{X}_{\cdot}t}^{j}(s)dw_{j}(s)-\sum_{j=1}^{n}\Theta_{X_{\cdot}\mathcal{X}_{\cdot}t}^{j}(s)db_{j}(s),\\
    \mathcal{Z}_{X_{\cdot}\mathcal{X}_{\cdot}t}(T)&=h_{xx}(X_{Xt}(T))\mathcal{X}_{X_{\cdot}\mathcal{X}_{\cdot}t}(T)+D_{X}^{2}\mathbb{E} \left(F_{T}\left((X_{X_{\cdot}t}(T)\otimes m)^{\mathcal{B}_{t}^{T}}\right)\right)(\mathcal{X}_{X_{\cdot}\mathcal{X}_{\cdot}t}(T)).
\end{aligned}\right.
\end{equation}


We can also check that $\mathcal{R}_{X_{\cdot}\mathcal{X}_{\cdot}t}^{j}(s),\Theta_{X_{\cdot}\mathcal{X}_{\cdot}t}^{j}(s)$
are the weak limits of $r_{X_{\cdot}\mathcal{X}_{\cdot}t}^{j,\epsilon}(s)$, $\rho_{X_{\cdot}\mathcal{X}_{\cdot}t}^{j,\epsilon}(s)$.

Since we have also 
\begin{align}
    \Big\langle \mathcal{X},Z_{X_{\cdot}\mathcal{X}_{\cdot}t}^{\epsilon}(t)\Big\rangle =&\Bigg\langle \mathcal{X},\int_{0}^{1} \bigg[l_{xx} \Big(X_{Xt}(s)+\theta\epsilon X_{X_{\cdot}\mathcal{X}_{\cdot}t}^{\epsilon}(s),u_{Xt}(s)+\theta\epsilon u_{X_{\cdot}\mathcal{X}_{\cdot}t}^{\epsilon}(s)\Big)X_{X_{\cdot}\mathcal{X}_{\cdot}t}^{\epsilon}(s)\nonumber\\
    &\qquad\qquad +l_{xv}\Big(X_{Xt}(s)+\theta\epsilon X_{X_{\cdot}\mathcal{X}_{\cdot}t}^{\epsilon}(s),u_{Xt}(s)+\theta\epsilon u_{X_{\cdot}\mathcal{X}_{\cdot}t}^{\epsilon}(s)\Big)u_{X_{\cdot}\mathcal{X}_{\cdot}t}^{\epsilon}(s) \bigg]d\theta\nonumber\\
    &\qquad \,\, +\int_{0}^{1}D_{X}^{2}\mathbb{E} \left(F\left(((X_{Xt}(s)+\theta\epsilon X_{X_{\cdot}\mathcal{X}_{\cdot}t}^{\epsilon}(s))\otimes m)^{\mathcal{B}_{t}^{s}}\right)\right)(X_{X_{\cdot}\mathcal{X}_{\cdot}t}^{\epsilon}(s))d\theta\Bigg\rangle, \nonumber
\end{align}

 we can state that 
\begin{equation}
\Big\langle \mathcal{X},Z_{X_{\cdot}\mathcal{X}_{\cdot}t}^{\epsilon}(t)\Big\rangle \rightarrow \Big\langle \mathcal{X},\mathcal{Z}_{X_{\cdot}\mathcal{X}_{\cdot}t}(t)\Big\rangle. \label{eq:8-80}
\end{equation}
Since $Z_{Xt}(t)=D_{X}V(X\otimes m,t),$ from the definition of the
second derivative (see (\ref{eq:2-227})), we obtain: 
\begin{equation}
D_{X}^{2}V(X\otimes m,t)(\mathcal{X})=\mathcal{Z}_{X_{\cdot}\mathcal{X}_{\cdot}t}(t).\label{eq:8-81}
\end{equation}
We then proceed to prove the continuity property (\ref{eq:5-2001}).
By definition,
\begin{numcases}{}
\mathcal{X}_{X_{k}\mathcal{X}_{k}t_{k}}(s)=\mathcal{X}_{k}+\int_{t_{k}}^{s}\mathcal{\mathcal{U}}_{X_{k}\mathcal{X}_{k}t_{k}}(\tau)d\tau,\ s>t_{k},\label{eq:8-10}\\
    -d\mathcal{\ Z}_{X_{k}\mathcal{X}_{k}t_{k}}(s)=\Bigg[l_{xx}\Big(X_{X_{k}t_{k}}(s),u_{X_{k}t_{k}}(s) \Big)\mathcal{X}_{X_{k}\mathcal{X}_{k}t_{k}}(s)+l_{xv}\Big(X_{X_{k}t_{k}}(s),u_{X_{k}t_{k}}(s)\Big)\mathcal{\mathcal{U}}_{X_{k}\mathcal{X}_{k}t_{k}}(s)\label{eq:8-82}\\
    \qquad\qquad\qquad\qquad +D_{X}^{2}\mathbb{E} \left(F\left((X_{X_{k}t_{k}}(s)\otimes m)^{\mathcal{B}_{t_{k}}^{s}}\right)\right)(\mathcal{X}_{X_{k}\mathcal{X}_{k}t_{k}}(s)) \Bigg]ds \nonumber\\
    \qquad\qquad\qquad\qquad -\sum_{j=1}^{n}\mathcal{R}_{X_{k}\mathcal{X}_{k}t_{k}}^{j}(s)dw_{j}(s)-\sum_{j=1}^{n}\Theta_{X_{k}\mathcal{X}_{k}t_{k}}^{j}(s)db_{j}(s),\nonumber\\
    \mathcal{Z}_{X_{k}\mathcal{X}_{k}t_{k}}(T)=h_{xx}(X_{X_{k}t_{k}}(T))\mathcal{X}_{X_{k}\mathcal{X}_{k}t_{k}}(T)+D_{X}^{2}\mathbb{E} \left(F_{T} \left((X_{X_{k}t_{k}}(T)\otimes m)^{\mathcal{B}_{t_{k}}^{T}}\right)\right)(\mathcal{X}_{X_{k}\mathcal{X}_{k}t_{k}}(T)),\nonumber\\
    l_{vx}\Big(X_{X_{k}t_{k}}(s),u_{X_{k}t_{k}}(s)\Big)\mathcal{X}_{X_{k}\mathcal{X}_{k}t_{k}}(s)+l_{vv}\Big(X_{X_{k}t_{k}}(s),u_{X_{k}t_{k}}(s)\Big)\mathcal{U}_{X_{k}\mathcal{X}_{k}t_{k}}(s)+\mathcal{Z}_{X_{k}\mathcal{X}_{k}t_{k}}(s)=0,\label{eq:8-83}
\end{numcases}





and 
\begin{numcases}{}
X_{X_{k}t_{k}}(s)=X_{k}+\int_{t_{k}}^{s}u_{X_{k}t_{k}}(\tau)d\tau+\sigma(w(s)-w(t_{k}))+\beta(b(s)-b(t_{k})),\label{eq:8-11}\\
    -dZ_{X_{k}t_{k}}(s)= \left(l_{x}(X_{X_{k}t_{k}}(s),u_{X_{k}t_{k}}(s))+D_{X}\mathbb{E} \left(F\left((X_{X_{k}t_{k}}(s)\otimes m)^{\mathcal{B}_{t_{k}}^{s}}\right)\right)\right)ds\nonumber\\
    \qquad\qquad\qquad \,\,\, -\sum_{j=1}^{n}r_{X_{k}t_{k}}^{j}(s)dw_{j}(s)-\sum_{j=1}^{n}\rho_{X_{k}t_{k}}^{j}(s)db_{j}(s),\label{eq:8-84}\\
    Z_{X_{k}t_{k}}(T)=h_{x}(X_{X_{k}t_{k}}(T))+D_{X}\mathbb{E} \left(F_{T}\left((Y_{X_{k}t_{k}}(T)\otimes m)^{\mathcal{B}_{t_{k}}^{T}}\right)\right),\nonumber\\
l_{v}\Big(X_{X_{k}t_{k}}(s),u_{X_{k}t_{k}}(s)\Big)+Z_{X_{k}t_{k}}(s)=0.\label{eq:8-85}
\end{numcases}




We need to prove that 
\begin{equation}
\mathcal{Z}_{X_{k}\mathcal{X}_{k}t_{k}}(t_{k})\rightarrow\mathcal{Z}_{X\mathcal{X}t}(t).\label{eq:8-86}
\end{equation}
We fix $s>t.$ We can assume that $s>t_{k}.$ Since, from (\ref{eq:8-4}), 
\begin{equation}
||\mathcal{Z}_{X_{k}\mathcal{X}_{k}t_{k}}(t_{k})-\mathcal{Z}_{X_{k}\mathcal{\mathcal{X}}t_{k}}(t_{k})||\leq C_{T}||\mathcal{X}_{k}-\mathcal{X}||,\label{eq:8-87}
\end{equation}
it is sufficient to prove that 
\begin{equation}
\mathcal{Z}_{X_{k}\mathcal{\mathcal{X}}t_{k}}(t_{k})\rightarrow\mathcal{Z}_{X\mathcal{X}t}(t).\label{eq:8-88}
\end{equation}
So we have to consider the system: 
\begin{numcases}{}
    \mathcal{Y}_{X_{k}\mathcal{\mathcal{X}}t_{k}}(s)=\mathcal{\mathcal{X}}+\int_{t_{k}}^{s}\mathcal{U}{}_{X\mathcal{\mathcal{X}}t_{k}}(\tau)d\tau,\ s>t_{k},\label{eq:8-89}\\
    -d\mathcal{\ Z}_{X_{k}\mathcal{\mathcal{X}}t_{k}}(s)=\Bigg[l_{xx}\Big(X_{X_{k}t_{k}}(s),u_{X_{k}t_{k}}(s)\Big)\mathcal{X}_{X_{k}\mathcal{\mathcal{X}}t_{k}}(s)+l_{xv}\Big(X_{X_{k}t_{k}}(s),u_{X_{k}t_{k}}(s)\Big)\mathcal{\mathcal{U}}_{X_{k}\mathcal{\mathcal{X}}t_{k}}(s)\nonumber\\
    \qquad\qquad\qquad\qquad +D_{X}^{2}\mathbb{E} \left(F\left((X_{X_{k}t_{k}}(s)\otimes m)^{B_{t_{k}}^{s}}\right)\right)(\mathcal{X}_{X_{k}\mathcal{\mathcal{X}}t_{k}}(s))\Bigg]ds \nonumber\\
    \qquad\qquad\qquad\qquad -\sum_{j=1}^{n}\mathcal{R}_{X_{k}\mathcal{\mathcal{X}}t_{k}}^{j}(s)dw_{j}(s)
    -\sum_{j=1}^{n}\Theta_{X_{k}\mathcal{X}_{k}t_{k}}^{j}(s)db_{j}(s),\label{eq:8-90}\\
    \mathcal{Z}_{X_{k}\mathcal{\mathcal{X}}t_{k}}(T)=h_{xx}(X_{X_{k}t_{k}}(T))\mathcal{X}_{X_{k}\mathcal{\mathcal{X}}t_{k}}(T)+D_{X}^{2}F_{T}\left((X_{X_{k}t_{k}}(T)\otimes m)^{\mathcal{B}_{t_{k}}^{T}}\right)(\mathcal{X}_{X_{k}\mathcal{\mathcal{X}}t_{k}}(T)),\nonumber\\
    l_{vx}\Big(X_{X_{k}t_{k}}(s),u_{X_{k}t_{k}}(s)\Big)\mathcal{X}_{X_{k}\mathcal{\mathcal{X}}t_{k}}(s)+l_{vv}\Big(X_{X_{k}t_{k}}(s),u_{X_{k}t_{k}}(s)\Big)\mathcal{U}_{X_{k}\mathcal{\mathcal{X}}t_{k}}(s)+\mathcal{Z}_{X_{k}\mathcal{\mathcal{X}}t_{k}}(s)=0.\label{eq:8-91}
\end{numcases}




From the proof of Proposition \ref{prop4-5}, we obtain 
\begin{align}
    ||X_{X_{k}t_{k}}(s)-X_{Xt}(s)||&\leq||X_{X_{k}t_{k}}(s)-X_{Xt_{k}}(s)||+||X_{Xt_{k}}(s)-X_{Xt}(s)||\nonumber\\
    &\leq C_{T}||X_{k}-X||+C_{T}||X||(t_{k}-t)+C_{T}(t_{k}-t)^{\frac{1}{2}}\label{eq:8-110}
\end{align}
 and similar estimates for $||Z_{X_{k}t_{k}}(s)-Z_{Xt}(s)||$ and
$||u_{X_{k}t_{k}}(s)-u_{Xt}(s)||$. We next introduce the $\sigma$-algebras 
\begin{equation}
\mathcal{\widetilde{F}}_{X\mathcal{X}t}^{s}=\underset{\{j|t_{j}<s\}}{\cup}\mathcal{F}_{X_{j}\mathcal{X}t_{j}}^{s}\cup\mathcal{F}_{t}^{t_{j}}.\label{eq:8-111}
\end{equation}
Note that $X$ is $\mathcal{\widetilde{F}}_{X\mathcal{X}t}^{s}$-measurable,
so $\mathcal{\widetilde{F}}_{X\mathcal{X}t}^{s}$ is an extension
of $\mathcal{F}_{X\mathcal{X}t}^{s}.$ To define the processes $\mathcal{X}_{X_{k}\mathcal{\mathcal{X}}t_{k}}(s)$, $\mathcal{U}_{X_{k}\mathcal{\mathcal{X}}t_{k}}(s)$, $\mathcal{\ Z}_{X_{k}\mathcal{\mathcal{X}}t_{k}}(s)$ for
$t<s<t_{k}$, we set 
\begin{equation}
\mathcal{X}_{X_{k}\mathcal{X}t_{k}}(s)=\mathcal{X},\ \mathcal{U}_{X_{k}\mathcal{\mathcal{X}}t_{k}}(s)=0,\ \mathcal{Z}_{X_{k}\mathcal{X}t_{k}}(s)=\mathbb{E}[\mathcal{Z}_{X_{k}\mathcal{X}t_{k}}(t_{k})|\widetilde{\mathcal{F}}_{X\mathcal{X}t}^{s}],\ t<s<t_{k}. \label{eq:8-15}
\end{equation}
To simplify notation, we shall denote $\mathcal{X}^{k}(s)=\mathcal{X}_{X_{k}\mathcal{X}t_{k}}(s)$
, $\mathcal{U}^{k}(s)=$ $\mathcal{U}_{X_{k}\mathcal{X}t_{k}}(s)$$,\ \mathcal{Z}^{k}(s)=\mathcal{Z}_{X_{k}\mathcal{X}t_{k}}(s)$.
We first see that $\mathcal{X}^{k}(s)$, $\mathcal{U}^{k}(s)$ and
$\mathcal{Z}^{k}(s)$ remain bounded in $L_{\widetilde{\mathcal{F}}_{X\mathcal{X}t}}^{\infty}(t,T;\mathcal{H}_{m}).$
We pick a subsequence of $\mathcal{X}^{k}(s),\mathcal{U}^{k}(s),\mathcal{Z}^{k}(s)$,
which converges weakly to $\mathcal{\widetilde{X}}(\cdot),\mathcal{\widetilde{U}}(\cdot),$$\widetilde{\mathcal{Z}}(\cdot)$
in $L_{\widetilde{\mathcal{F}}_{X\mathcal{X}t}}^{2}(t,T;\mathcal{H}_{m}).$
Then 
\[
\mathcal{\widetilde{X}}(s)\ =\mathcal{X}+\int_{t}^{s}\mathcal{\widetilde{U}}(\tau)d\tau.
\]
As for the proof of (\ref{eq:8-72})-(\ref{eq:8-76}) we have 
\begin{align*}
    &\quad \, l_{xx}\Big(X_{X_{k}t_{k}}(s),u_{X_{k}t_{k}}(s)\Big)\mathcal{X}_{X_{k}\mathcal{\mathcal{X}}t_{k}}(s)+l_{xv}\Big(X_{X_{k}t_{k}}(s),u_{X_{k}t_{k}}(s)\Big)\mathcal{\mathcal{U}}_{X_{k}\mathcal{\mathcal{X}}t_{k}}(s)\\
    &\quad \, +D_{X}^{2}\mathbb{E} \left(F\left((X_{X_{k}t_{k}}(s)\otimes m)^{\mathcal{B}_{t_{k}}^{s}}\right)\right)(\mathcal{X}_{X_{k}\mathcal{\mathcal{X}}t_{k}}(s))\nonumber\\
    &\rightarrow l_{xx}\Big(X_{Xt}(s),u_{Xt}(s)\Big)\mathcal{\widetilde{X}}(s)+l_{xv}\Big(X_{Xt}(s),u_{Xt}(s)\Big)\mathcal{\widetilde{U}}(s)+D_{X}^{2}\mathbb{E} \left(F\left((X_{Xt}(s)\otimes m)^{\mathcal{B}_{t}^{s}}\right)\right)(\mathcal{\widetilde{X}}(s))\nonumber
\end{align*}

weakly in $L_{\widetilde{\mathcal{F}}_{X\mathcal{X}t}}^{2}(t,T;\mathcal{H}_{m})$, and likewise for similar terms. We have used here the property 
\[
(X_{X_{k}t_{k}}(s)\otimes m)^{\mathcal{B}_{t_{k}}^{s}}=(X_{X_{k}t_{k}}(s)\otimes m)^{\mathcal{B}_{t}^{s}}
\]
since $\mathcal{B}_{t}^{s}=\mathcal{B}_{t_{k}}^{s}\cup\mathcal{B}_{t}^{t_{k}}$
and $X_{X_{k}t_{k}}(s)$ is independent of $\mathcal{B}_{t}^{t_{k}}.$ 

Necessarily $\mathcal{\widetilde{X}}(s)=\mathcal{X}_{X\mathcal{X}t}(s),\ \mathcal{\widetilde{U}}(s)=\mathcal{\mathcal{U}}_{X\mathcal{X}t}(s),\ \widetilde{\mathcal{Z}}(s)=\mathcal{Z}_{X\mathcal{X}t}(s).$
We need next to prove the strong convergence. We first check the
relation that
\begin{equation*} \hspace{-1cm}
\begin{aligned}
    I_{k}&=\int_{t_{k}}^{T}\Bigg[\Big\langle l_{xx}\Big(X_{X_{k}t_{k}}(s),u_{X_{k}t_{k}}(s)\Big)\mathcal{X}_{X_{k}\mathcal{X}t_{k}}(s),\mathcal{X}_{X_{k}\mathcal{X}t_{k}}(s)\ \Big\rangle +2\Big\langle l_{xv}\Big(X_{X_{k}t_{k}}(s),u_{X_{k}t_{k}}(s)\Big)\mathcal{U}_{X_{k}\mathcal{X}t_{k}}(s),\mathcal{X}_{X_{k}\mathcal{X}t_{k}}(s)\ \Big\rangle \nonumber\\
    &\quad +\Big\langle l_{vv}\Big(X_{X_{k}t_{k}}(s),u_{X_{k}t_{k}}(s)\Big)\mathcal{U}_{X_{k}\mathcal{X}t_{k}}(s),\mathcal{U}_{X_{k}\mathcal{X}t_{k}}(s)\ \Big\rangle+\Big\langle D_{X}^{2}\mathbb{E}\left(F\left((X_{X_{k}t_{k}}(s)\otimes m)^{\mathcal{B}_{t_{k}}^{s}}\right)\right)(\mathcal{X}_{X_{k}\mathcal{X}t_{k}}(s)),\mathcal{X}_{X_{k}\mathcal{X}t_{k}}(s)\Big\rangle\Bigg]ds\nonumber\\
    &\quad +\bigg\langle h_{xx}(X_{X_{k}t_{k}}(T))\mathcal{X}_{X_{k}\mathcal{X}t_{k}}(T)+D_{X}^{2}\mathbb{E} \left(F_{T}\left((X_{X_{k}t_{k}}(T)\otimes m)^{\mathcal{B}_{t_{k}}^{T}}\right)\right)(\mathcal{X}_{X_{k}\mathcal{X}t_{k}}(T)),\mathcal{X}_{X_{k}\mathcal{X}t_{k}}(T)\bigg\rangle\nonumber\\
    &=\Big\langle \mathcal{Z}_{X_{k}\mathcal{X}t_{k}}(t_{k}),\mathcal{X}\Big\rangle \nonumber
\end{aligned}
\end{equation*}
converges to


\begin{equation*} \hspace{-1.1cm}
\begin{aligned}
     I&=\Big\langle \mathcal{Z}_{X\mathcal{X}t}(t),\mathcal{X}\Big\rangle \\
     &=\int_{t}^{T}\Bigg[\bigg\langle l_{xx}\Big(X_{Xt}(s),u_{Xt}(s)\Big)\mathcal{X}_{X\mathcal{X}t}(s),\mathcal{X}_{X\mathcal{X}t}(s)\ \bigg\rangle +2\bigg\langle l_{xv}\Big(X_{Xt}(s),u_{Xt}(s)\Big)\mathcal{U}_{X\mathcal{X}t}(s),\mathcal{X}_{X\mathcal{X}t}(s)\ \bigg\rangle \nonumber\\
     &\qquad +\bigg\langle l_{vv}\Big(X_{Xt}(s),u_{Xt}(s)\Big)\mathcal{U}_{X\mathcal{X}t}(s),\mathcal{U}_{X\mathcal{X}t}(s)\ \bigg\rangle +\bigg\langle D_{X}^{2}\mathbb{E} \left(F\left((X_{Xt}(s)\otimes m)^{\mathcal{B}_{t}^{s}}\right)\right)(\mathcal{X}_{X\mathcal{X}t}(s)),\mathcal{X}_{X\mathcal{X}t}(s)\bigg\rangle \Bigg]ds\nonumber\\
     &\qquad +\bigg\langle h_{xx}(X_{Xt}(T))\mathcal{X}_{X\mathcal{X}t}(T)+D_{X}^{2}\mathbb{E} \left(F_{T}\left((X_{Xt}(T)\otimes m)^{\mathcal{B}_{t}^{T}}\right)\right)(\mathcal{X}_{X\mathcal{X}t}(T)),\mathcal{X}_{X\mathcal{X}t}(T)\bigg\rangle. \nonumber
\end{aligned}
\end{equation*}


From this convergence and weak convergence, we obtain immediately 
\begin{align}
    J_{k}=&\int_{t_{k}}^{T}\Bigg[\bigg\langle l_{xx}\Big(X_{X_{k}t_{k}}(s),u_{X_{k}t_{k}}(s)\Big)(\mathcal{X}_{X_{k}\mathcal{X}t_{k}}(s)-\mathcal{X}_{X\mathcal{X}t}(s)),\mathcal{X}_{X_{k}\mathcal{X}t_{k}}(s)-\mathcal{X}_{X\mathcal{X}t}(s)\ \bigg\rangle \nonumber\\
    &\qquad +2\bigg\langle l_{xv}\Big(X_{X_{k}t_{k}}(s),u_{X_{k}t_{k}}(s)\Big)(\mathcal{U}_{X_{k}\mathcal{X}t_{k}}(s)-\mathcal{U}_{X\mathcal{X}t}(s)),\mathcal{X}_{X_{k}\mathcal{X}t_{k}}(s)-\mathcal{X}_{X\mathcal{X}t}(s)\ \bigg\rangle \nonumber\\
    &\qquad +\bigg\langle l_{vv}\Big(X_{Xt}(s),u_{Xt}(s)\Big)(\mathcal{U}_{X_{k}\mathcal{X}t_{k}}(s)-\mathcal{U}_{X\mathcal{X}t}(s)),\mathcal{U}_{X_{k}\mathcal{X}t_{k}}(s)-\mathcal{U}_{X\mathcal{X}t}(s)\ \bigg\rangle \nonumber\\
    &\qquad +\bigg\langle D_{X}^{2}\mathbb{E} \left(F\left((X_{X_{k}t_{k}}(s)\otimes m)^{\mathcal{B}_{t}^{s}}\right)\right)(\mathcal{X}_{X_{k}\mathcal{X}t_{k}}(s)-\mathcal{X}_{X\mathcal{X}t}(s)),\mathcal{X}{}_{X_{k}\mathcal{X}t_{k}}(s)-\mathcal{X}_{X\mathcal{X}t}(s)\bigg\rangle\Bigg]ds \\
    &+\bigg\langle h_{xx}(X_{X_{k}t_{k}}(T))(\mathcal{X}_{X_{k}\mathcal{X}t_{k}}(T)-\mathcal{X}_{X\mathcal{X}t}(T))\nonumber\\
    &\qquad +D_{X}^{2}\mathbb{E} \left(F_{T}\left((X_{X_{k}t_{k}}(T)\otimes m)^{\mathcal{B}_{t}^{T}}\right)\right) \cdot (\mathcal{X}_{X_{k}\mathcal{X}t_{k}}(T)-\mathcal{X}_{X\mathcal{X}t}(T)),\mathcal{X}_{X_{k}\mathcal{X}t_{k}}(T)-\mathcal{X}_{X\mathcal{X}t}(T)\bigg\rangle \nonumber\\
    \rightarrow &\ 0.\nonumber
\end{align}




But 
\begin{equation}
\begin{aligned}
    J_{k}\geq&\, \lambda\int_{t_{k}}^{T} \left\lVert \mathcal{U}_{X_{k}\mathcal{X}t_{k}}(s)-\mathcal{U}_{X\mathcal{X}t}(s)\right\rVert^{2}ds-(c'_{l}+c')\int_{t_{k}}^{T}||\mathcal{X}_{X_{k}\mathcal{X}t_{k}}(s)-\mathcal{X}_{X\mathcal{X}t}(s)||^{2}ds\\
    &-(c'_{h}+c'_{T})||\mathcal{X}_{X_{k}\mathcal{X}t_{k}}(T)-\mathcal{X}_{X\mathcal{X}t}(T)||^{2},
\end{aligned}
\end{equation}


and 
\begin{equation}
\begin{aligned}
    ||\mathcal{X}_{X_{k}\mathcal{X}t_{k}}(T)-\mathcal{X}_{X\mathcal{X}t}(T)||^{2}\leq& \,(1+\epsilon)T\int_{t_{k}}^{T}||\mathcal{U}_{X_{k}\mathcal{X}t_{k}}(s)-\mathcal{U}_{X\mathcal{X}t}(s)||^{2}ds\\
    &+\left(1+\dfrac{1}{\epsilon}\right)(t_{k}-t)^{2}\sup_{s}||\mathcal{U}_{X\mathcal{X}t}(s)||^{2},\\
    \int_{t_{k}}^{T}||\mathcal{X}_{X_{k}\mathcal{X}t_{k}}(s)-\mathcal{X}_{X\mathcal{X}t}(s)||^{2}ds\leq&\,(1+\epsilon)\dfrac{T^{2}}{2}\int_{t_{k}}^{T}||\mathcal{U}_{X_{k}\mathcal{X}t_{k}}(s)-\mathcal{U}_{X\mathcal{X}t}(s)||^{2}ds\\
    &+\left(1+\dfrac{1}{\epsilon}\right)(t_{k}-t)^{2}\sup_{s}||\mathcal{U}_{X\mathcal{X}t}(s)||^{2}.
\end{aligned}
\end{equation}
From the assumption (\ref{eq:3-6}), since $t_{k}\downarrow t$ and
$J_{k}\rightarrow0$, we obtain $\int_{t_{k}}^{T}||\mathcal{U}_{X_{k}\mathcal{X}t_{k}}(s)-\mathcal{U}_{X\mathcal{X}t}(s)||^{2}ds\rightarrow0.$
It is then easy to get (\ref{eq:8-86}). This concludes the proof
of Proposition \ref{prop5-10}.  

\subsection{PROOF OF THEOREM \ref{theo5-10}}

From the optimality principle we can write 
\begin{align}
    0=&\, \dfrac{1}{\epsilon}\int_{s}^{s+\epsilon} \left[\mathbb{E} \left(\int_{\mathbb{R}^{n}}l\Big(X_{X_{\cdot}t}(\tau),u_{X_{\cdot}t}(\tau)\Big)dm(x)\right)+\mathbb{E}\left(F\left((X_{X_{\cdot}t}(\tau)\otimes m)^{\mathcal{B}_{t}^{\tau}}\right)\right)\right]d\tau\nonumber\\
    &+\dfrac{1}{\epsilon} \bigg[\mathbb{E} \left(V\left((X_{X_{\cdot}t}(s+\epsilon)\otimes m)^{\mathcal{B}_{t}^{s+\epsilon}},t+\epsilon\right)\right)-\mathbb{E}\left(V\left((X_{Xt}(s)\otimes m,s)^{\mathcal{B}_{t}^{s}},s\right)\right)\bigg].\nonumber
\end{align}
From the continuity of functions $s\mapsto X_{X_{\cdot}t}(s),u_{X_{\cdot}t}(s)$,
the first term converges to $\mathbb{E}\left(\int_{\mathbb{R}^{n}}l(X_{X_{\cdot}t}(s),u_{X_{\cdot}t}(s))dm(x)\right)+\mathbb{E}\left(F((X_{X_{\cdot}t}(s)\otimes m)^{\mathcal{B}_{t}^{s}})\right)$.
Therefore we have 
\begin{equation}
\mathbb{E}\left(\int_{\mathbb{R}^{n}}l\Big(X_{X_{\cdot}t}(s),u_{X_{\cdot}t}(s)\Big)dm(x)\right)+\mathbb{E}\left(F\left((X_{X_{\cdot}t}(s)\otimes m)^{\mathcal{B}_{t}^{s}}\right)\right)+\dfrac{d}{ds}\mathbb{E}\left(V\left((X_{X_{\cdot}t}(s)\otimes m)^{\mathcal{B}_{t}^{s}},s\right)\right)=0.\label{eq:9-1}
\end{equation}
But we have proven that the function $X\mapsto \mathbb{E} \left(V((X_{X_{\cdot}t}(s)\otimes m)^{\mathcal{B}_{t}^{s}},s)\right)$
satisfies all the conditions for the applicability of Theorem \ref{theo3-10}.
So 
\begin{equation}\label{eq:9-3}
    \begin{aligned}
    \dfrac{d}{ds}\mathbb{E}\left(V\left((X_{X_{\cdot}t}(s)\otimes m)^{\mathcal{B}_{t}^{s}},s\right)\right)=& \,\dfrac{\partial}{\partial s}\mathbb{E}\left(V\left((X_{X_{\cdot}t}(s)\otimes m)^{\mathcal{B}_{t}^{s}},s\right)\right)\\
    &+\bigg\langle D_{X}\mathbb{E}\left(V\left((X_{X_{\cdot}t}(s)\otimes m)^{\mathcal{B}_{t}^{s}},s\right)\right),u_{X_{\cdot}t}(s)\bigg\rangle \\
    &+\dfrac{1}{2}\sum_{j=1}^{n}\bigg\langle D_{X}^{2}\mathbb{E}\left(V\left((X_{X_{\cdot}t}(s)\otimes m)^{\mathcal{B}_{t}^{s}},s\right)\right)(\sigma N^j_{s}),\sigma N^j_{s}\bigg\rangle \\
    &+\dfrac{\beta^{2}}{2}\sum_{j=1}^{n} \bigg\langle D_{X}^{2}\mathbb{E}\left(V\left((X_{X_{\cdot}t}(s)\otimes m)^{\mathcal{B}_{t}^{s}},s\right)\right)(e^{j}),e^{j}\bigg\rangle,\, \text{a.e.}\ s\in(0,T),
\end{aligned}
\end{equation}

where $N^{j}_s$ are scalar standard Gaussian independent of $\mathcal{F}_{Xt}^{s}.$ Recall
that $D_{X}\mathbb{E}\left(V((X_{X_{\cdot}t}(s)\otimes m)^{\mathcal{B}_{t}^{s}},s)\right)=Z_{X_{\cdot}t}(s)$
and 
\[
\Big\langle Z_{X_{\cdot}t}(s),u_{X_{\cdot}t}(s) \Big\rangle =\mathbb{E}\left(\int_{\mathbb{R}^{n}}Z_{X_{\cdot}t}(s) \cdot u_{X_{\cdot}t}(s)dm(x)\right).
\]
Since 
\[
l_{v}(X_{\cdot, t}(s),u_{X_{\cdot}t}(s))+Z_{X_{\cdot}t}(s)=0,
\]
we can write 
\begin{align}
    \mathbb{E}\left(\int_{\mathbb{R}^{n}}l\Big(X_{X_{\cdot}t}(s),u_{X_{\cdot}t}(s)\Big)dm(x)\right)+\mathbb{E}\left(\int_{\mathbb{R}^{n}}Z_{X_{\cdot}t}(s) \cdot u_{X_{\cdot}t}(s)dm(x)\right)=\mathbb{E}\left(\int_{\mathbb{R}^{n}}H\Big(X_{X_{\cdot}t}(s),Z_{X_{\cdot}t}(s)\Big)dm(x)\right).\label{eq:9-2}
\end{align}

 Therefore (\ref{eq:9-1}) and (\ref{eq:9-3}) yields 
\begin{equation}\label{eq:9-4}
    \begin{aligned}
    &\quad \dfrac{\partial}{\partial s}\mathbb{E}\left(V\left((X_{X_{\cdot}t}(s)\otimes m)^{\mathcal{B}_{t}^{s}},s\right)\right)+\mathbb{E}\left(\int_{\mathbb{R}^{n}}H\Big(X_{X_{\cdot}t}(s),Z_{X_{\cdot}t}(s)\Big)dm(x)\right)+\mathbb{E}\left(V\left((X_{X_{\cdot}t}(s)\otimes m)^{\mathcal{B}_{t}^{s}},s\right)\right)\\
    &+\dfrac{1}{2} \bigg\langle D_{X}^{2}\mathbb{E}\left(V\left((X_{X_{\cdot}t}(s)\otimes m)^{\mathcal{B}_{t}^{s}},s\right)\right)(\sigma N_{s}),\sigma N_{s} \bigg\rangle +\dfrac{\beta^{2}}{2}\sum_{j=1}^{n} \bigg\langle D_{X}^{2}\mathbb{E}\left(V\left((X_{X_{\cdot}t}(s)\otimes m)^{\mathcal{B}_{t}^{s}},s\right)\right)(e^{j}),e^{j}\bigg\rangle =0,\text{a.e.}.
\end{aligned}
\end{equation}

 Also, 
\begin{equation}
\mathbb{E}\left(V\left((X_{X_{\cdot}t}(T)\otimes m)^{\mathcal{B}_{t}^{T}},T\right)\right)=\mathbb{E}\left(\int_{\mathbb{R}^{n}} h(X_{X_{\cdot}t}(T))dm(x)\right)+\mathbb{E}\left(F_{T}\left((X_{X_{\cdot}t}(T)\otimes m,T)^{\mathcal{B}_{t}^{T}}\right)\right).\label{eq:9-5}
\end{equation}
Taking $s=t$ in (\ref{eq:9-4}) and $t=T$ in (\ref{eq:9-5}), and
recalling that $Z_{X_{\cdot}t}(t)=D_{X}V(X_{\cdot}\otimes m,t),$ we obtain
(\ref{eq:5-2008}). Now if a functional $V(X\otimes m,t)$ satisfies
the regularity properties of the value function, then it satisfies
also (\ref{eq:9-4}) and (\ref{eq:9-5}). Then (\ref{eq:9-1}) holds
and integrating with respect to $s$, between $t$ and $T,$ we obtain
\begin{equation}
V(X_{\cdot}\otimes m,t)=J_{X_{\cdot}t}(u_{X_{\cdot}t}(\cdot)).\label{eq:9-6}
\end{equation}
Since the right-hand side is the value function and uniquely defined,
the solution is necessarily unique. Note that $X_{\cdot}$ must be independent
of $\mathcal{F}_{t}.$ But given $X_{\cdot}$, we can always construct
the Wiener process, so that this condition is satisfied. This concludes
the proof.  

\section{PROOF OF PROPOSITION \ref{prop8-1} }

Considering Proposition \ref{prop5-10}, we can enlarge the space
of controls as follows: Let $\mathcal{B}_{x}$ be a family of $\sigma$-algebras independent of the filtration $\mathcal{F}_{t}$ and $X_{\cdot}$, which is the initial condition of the system \eqref{eq:3-7}
be $\mathcal{B}_{\cdot}$-measurable. Let also $\mathcal{X}_{\cdot}$ be also
$\mathcal{B}_{\cdot}$-measurable. We consider the control problem 
\begin{equation}
\mathcal{X}_{\mathcal{B}_{\cdot}}(s)=\mathcal{X}_{\cdot}+\int_{t}^{s}\mathcal{V}_{\mathcal{B}_{\cdot}}(\tau)d\tau,\label{eq:9-7}
\end{equation}
\begin{equation} \label{eq:9-8}
\begin{aligned}
    \mathcal{J}_{X_{\cdot}\mathcal{B}_{\cdot}t}(\mathcal{V}_{\mathcal{B}_{\cdot}}(\cdot))=&\,\dfrac{1}{2}\int_{t}^{T} \Bigg[\bigg\langle l_{xx}\Big(X_{X_{\cdot}t}(s),u_{X_{\cdot}t}(s)\Big)\mathcal{X}_{\mathcal{B}_{\cdot}}(s)+D_{X}^{2}\mathbb{E}\left(F\left((X_{X_{\cdot}t}(s)\otimes m)^{\mathcal{B}_{t}^{s}}\right)\right)(\mathcal{X}_{\mathcal{B}_{\cdot}}(s)),\mathcal{X}_{\mathcal{B}_{\cdot}}(s)\bigg\rangle \\
    &\qquad +2\bigg\langle l_{xv}\Big(X_{X_{\cdot}t}(s),u_{X_{\cdot}t}(s)\Big)\mathcal{V}_{\mathcal{B}_{\cdot}}(s),\mathcal{X}_{\mathcal{B}_{\cdot}}(s)\bigg\rangle +\bigg\langle l_{vv}\Big(X_{X_{_{\cdot}}t}(s),u_{X_{\cdot}t}(s)\Big)\mathcal{V}_{\mathcal{B}_{\cdot}}(s),\mathcal{V}_{\mathcal{B}_{\cdot}}(s)\bigg\rangle \Bigg]ds\\
    &+\dfrac{1}{2}\bigg\langle h_{xx}(X_{X_{\cdot}t}(T))\mathcal{X}_{\mathcal{B}_{\cdot}}(T)+D_{X}^{2}\mathbb{E}\left(F_{T}\left((X_{X_{\cdot}t}(T)\otimes m)^{\mathcal{B}_{t}^{T}}\right)\right)(\mathcal{X}_{\mathcal{B}_{\cdot}}(T)),\mathcal{X}_{\mathcal{B}_{\cdot}}(T)\bigg\rangle.
\end{aligned}
\end{equation}


Then we have 
\begin{align}
    \inf_{\mathcal{V}_{\mathcal{B}_{\cdot}}(\cdot)}\mathcal{J}_{X_{\cdot}\mathcal{B}_{\cdot}t}(\mathcal{V}_{\mathcal{B}_{\cdot}}(\cdot))=\dfrac{1}{2} \Big\langle \mathcal{Z}_{X_{\cdot}\mathcal{X}_{\cdot}t}(t),\mathcal{X}_{\cdot} \Big\rangle &=\dfrac{1}{2}\Big\langle D_{X}^{2}V(X_{\cdot}\otimes m,t)(\mathcal{X}_{\cdot}),\mathcal{X}_{\cdot}\Big\rangle = \mathcal{J}_{X_{\cdot}\mathcal{B}_{\cdot}t}(\mathcal{U}_{X_{\cdot}\mathcal{X}_{\cdot}t}(\cdot)). \label{eq:9-9}
\end{align}

From this formula and the uniqueness of the point of minimum, we can
give a formula for $\langle D_{X} \langle D_{X}^{2}V(X_{\cdot}\otimes m,t)(\mathcal{X}_{\cdot}),\mathcal{X}_{\cdot}\rangle ,\mathcal{Z}_{\cdot} \rangle $
where $\mathcal{Z}_{\cdot}$ is also $\mathcal{B}_{\cdot}$-measurable and
independent of $\mathcal{F}_{t}.$ This formula can be applied to
$\mathcal{X}_{\cdot}=\sigma N_{t}$ and $e^{j}$, respectively. This formula
will be valid thanks to the assumptions (\ref{eq:8-600}), (\ref{eq:8-601}), (\ref{eq:8-602}), (\ref{eq:8-603}). This provides the justification needed to obtain the master equation
(\ref{eq:8-500}). Denoting by $D_{X_{\cdot}}\mathcal{J}_{X_{\cdot}\mathcal{B}_{\cdot}t}(\mathcal{V}_{\mathcal{B}_{\cdot}}(\cdot))$
the G\textroundcap{a}teaux differential of $\mathcal{J}_{X_{\cdot}\mathcal{B}_{\cdot}t}(\mathcal{V}_{\mathcal{B}_{\cdot}}(\cdot))$
with respect to $X_{\cdot}$, when $\mathcal{V}_{\mathcal{B}_{\cdot}}(\cdot)$
is fixed, then we have 
\begin{equation}
\dfrac{1}{2}\bigg\langle D_{X}\Big\langle D_{X}^{2}V(X_{\cdot}\otimes m,t)(\mathcal{X}_{\cdot}),\mathcal{X}_{\cdot}\Big\rangle ,\mathcal{Z}_{\cdot}\bigg\rangle =D_{X_{\cdot}}\mathcal{J}_{X_{\cdot}\mathcal{B}_{\cdot}t}(\mathcal{U}_{X_{\cdot}\mathcal{X}_{\cdot}t}(\cdot)).\label{eq:9-10}
\end{equation}
We get the very long formula: 
\begin{equation} \label{eq:9-11}
\begin{aligned}
    &\quad\ \dfrac{1}{2} \bigg\langle D_{X} \Big\langle D_{X}^{2}V(X_{\cdot}\otimes m,t)(\mathcal{X}_{\cdot}),\mathcal{X}_{\cdot}\Big\rangle ,\mathcal{Z}_{\cdot}\bigg\rangle \\
    &=\dfrac{1}{2}\int_{t}^{T}\left[\bigg\langle \ \left(l_{xxx}\Big(X_{X_{\cdot}t}(s),u_{X_{\cdot}t}(s)\Big)\mathcal{X}_{X_{\cdot}\mathcal{Z}_{\cdot}t}(s)\right)\,\mathcal{X}_{X_{\cdot}\mathcal{X}_{\cdot}t}(s),\mathcal{X}_{X_{\cdot}\mathcal{X}_{\cdot}t}(s)\,\bigg\rangle \right.\\
    &\qquad\qquad+\dfrac{1}{2}\bigg\langle \  \left(D_{X}^{3}\mathbb{E}\left(F\left((X_{X_{\cdot}t}(s)\otimes m)^{\mathcal{B}_{t}^{s}}\right)\right)(\mathcal{X}_{X_{\cdot}\mathcal{Z}_{\cdot}t}(s))\right)\,(\mathcal{X}_{X_{\cdot}\mathcal{X}_{\cdot}t}(s)),\mathcal{X}_{X_{\cdot}\mathcal{X}_{\cdot}t}(s)\,\bigg\rangle \\
    &\qquad\qquad+\dfrac{1}{2}\bigg\langle \ \Big(l_{xxv}(X_{X_{\cdot}t}(s),u_{X_{\cdot}t}(s))\mathcal{U}_{X_{\cdot}\mathcal{Z}_{\cdot}t}(s)\Big)\,\mathcal{X}_{X_{\cdot}\mathcal{X}_{\cdot}t}(s),\mathcal{X}_{X_{\cdot}\mathcal{X}_{\cdot}t}(s)\,\bigg\rangle \\
    &\qquad\qquad+\bigg\langle \ \Big(l_{xvx}(X_{X_{\cdot}t}(s),u_{X_{\cdot}t}(s))\mathcal{X}_{X_{\cdot}\mathcal{Z}_{\cdot}t}(s)\Big)\,\mathcal{U}_{X_{\cdot}\mathcal{X}_{\cdot}t}(s),\mathcal{X}_{X_{\cdot}\mathcal{X}_{\cdot}t}(s)\,\bigg\rangle \\
    &\qquad\qquad+\bigg\langle \ \Big(l_{xvv}(X_{X_{\cdot}t}(s),u_{X_{\cdot}t}(s))\mathcal{U}_{X_{\cdot}\mathcal{Z}_{\cdot}t}(s)\Big)\,\mathcal{U}_{X_{\cdot}\mathcal{X}_{\cdot}t}(s),\mathcal{X}_{X_{\cdot}\mathcal{X}_{\cdot}t}(s)\,\bigg\rangle \\
    &\qquad\qquad+\dfrac{1}{2}\bigg\langle \ \Big(l_{vvx}(X_{X_{\cdot}t}(s),u_{X_{\cdot}t}(s))\mathcal{X}_{X_{\cdot}\mathcal{Z}_{\cdot}t}(s)\Big)\,\mathcal{U}_{X_{\cdot}\mathcal{X}_{\cdot}t}(s),\mathcal{U}_{X_{\cdot}\mathcal{X}_{\cdot}t}(s)\,\bigg\rangle \\
    &\qquad\qquad\left.+\dfrac{1}{2}\bigg\langle \ \Big(l_{vvv}(X_{X_{\cdot}t}(s),u_{X_{\cdot}t}(s))\mathcal{U}_{X_{\cdot}\mathcal{Z}_{\cdot}t}(s)\Big)\,\mathcal{U}_{X_{\cdot}\mathcal{X}_{\cdot}t}(s),\mathcal{U}_{X_{\cdot}\mathcal{X}_{\cdot}t}(s)\,\bigg\rangle \right]ds\\
    &\qquad+\dfrac{1}{2}\bigg\langle \ \Big(h_{xxx}(X_{X_{\cdot}t}(T))\mathcal{X}_{X_{\cdot}\mathcal{Z}_{\cdot}t}(T)\Big)\,\mathcal{X}_{X_{\cdot}\mathcal{X}_{\cdot}t}(T),\mathcal{X}_{X_{\cdot}\mathcal{X}_{\cdot}t}(T)\,\bigg\rangle\\
    &\qquad+\dfrac{1}{2}\bigg\langle \ \left(D_{X}^{3}\mathbb{E}\left(F_{T}\left((X_{X_{\cdot}t}(T)\otimes m)^{\mathcal{B}_{t}^{T}}\right)\right)(\mathcal{X}_{X_{\cdot}\mathcal{Z}_{\cdot}t}(T))\right)\,(\mathcal{X}_{X_{\cdot}\mathcal{X}_{\cdot}t}(T)),\mathcal{X}_{X_{\cdot}\mathcal{X}_{\cdot}t}(T)\,\bigg\rangle. 
\end{aligned} 
\end{equation}

This formula is not easy to read. One must keep in mind that $D_{X}^{3}\mathbb{E}\left(F((X_{X_{\cdot}t}(s)\otimes m)^{\mathcal{B}_{t}^{s}})\right)$
$\in\mathcal{L}(\mathcal{H}_{m};\mathcal{L}(\mathcal{H}_{m};\mathcal{H}_{m}))$, hence
$\left(D_{X}^{3}\mathbb{E}\left(F((X_{X_{\cdot}t}(s)\otimes m)^{\mathcal{B}_{t}^{s}})\right)(\mathcal{X}_{X_{\cdot}\mathcal{Z}_{\cdot}t}(s))\right)\in\mathcal{L}(\mathcal{H}_{m};\mathcal{H}_{m}).$
The interpretation of the other terms is similar. More specifically, we can express 
{\small \begin{align}
    &\quad \bigg\langle D_{X}^{3}\mathbb{E}\left(F\left((Z_{\cdot}\otimes m)^{\mathcal{B}}\right)\right)(\mathcal{Z}_{\cdot})(Y_{\cdot}),Y_{\cdot}\bigg\rangle = \bigg\langle D_{X} \Big\langle D_{X}^{2}\mathbb{E}\left(F\left((Z_{\cdot}\otimes m)^{\mathcal{B}}\right)\right)(Y_{\cdot}),Y_{\cdot}\Big\rangle ,\mathcal{Z}_{\cdot}\bigg\rangle \nonumber\\
    &=\mathbb{E}\left(\int_{\mathbb{R}^{n}}\left(D^{3}\dfrac{dF}{d\nu} \left((Z_{\cdot}\otimes m)^{\mathcal{B}}\right)(Z_{x})\mathcal{Z}_{x}\right)Y_{x} \cdot Y_{x}dm(x)\right)\nonumber\\
    &\quad+\mathbb{E}\left(\ \mathbb{E}^{\mathcal{B}}\left(\int_{\mathbb{R}^{n}}\mathbb{E}^{1\mathcal{B}_{\cdot}}\left(\int_{\mathbb{R}^{n}} \left(D_{1}^{2}\dfrac{d^{2}F}{d\nu^{2}}\left((Z_{\cdot}\otimes m)^{\mathcal{B}}\right) (Z_{x},Z_{x^{1}}^{1})\mathcal{Z}_{x^{1}}^{1}\right)Y_{x} \cdot Y_{x}dm(x^{1})\right)dm(x)\right)\right)\nonumber\\
    &\quad+\mathbb{E}\left(\ \mathbb{E}^{\mathcal{B}}\left(\int_{\mathbb{R}^{n}}\mathbb{E}^{1\mathcal{B}_{\cdot}}\left(\int_{\mathbb{R}^{n}}\left(D^{2}D_{1}\dfrac{d^{2}F}{d\nu^{2}}\left((Z_{\cdot}\otimes m)^{\mathcal{B}}\right) (Z_{x},Z_{x^{1}}^{1})\mathcal{Z}_{x}\right)Y_{x^{1}}^{1} \cdot Y_{x}dm(x^{1})\right)dm(x)\right)\right)\nonumber\\
&\quad+\mathbb{E}\left(\ \mathbb{E}^{\mathcal{B}}\left(\int_{\mathbb{R}^{n}}\mathbb{E}^{1\mathcal{B}_{\cdot}}\left(\int_{\mathbb{R}^{n}}\left(DD_{1}^{2}\dfrac{d^{2}F}{d\nu^{2}}\left((Z_{\cdot}\otimes m)^{\mathcal{B}}\right) (Z_{x},Z_{x^{1}}^{1})\mathcal{Z}_{x^{1}}^{1}\right)Y_{x^{1}}^{1} \cdot Y_{x}dm(x^{1})\right)dm(x)\right)\right)\nonumber\\
&\quad+\mathbb{E}\left(\ \mathbb{E}^{\mathcal{B}}\left(\int_{\mathbb{R}^{n}}\mathbb{E}^{1\mathcal{B}_{\cdot}}\left(\int_{\mathbb{R}^{n}}\mathbb{E}^{2\mathcal{B}_{\cdot}}\left(\int_{\mathbb{R}^{n}}\left(DD_{1}D_{2}\dfrac{d^{3}F}{d\nu^{3}}\left((Z_{\cdot}\otimes m)^{\mathcal{B}}\right) (Z_{x},Z_{x^{1}}^{1},Z_{x^{2}}^{2})\mathcal{Z}_{x^{2}}^{2}\right)Y_{x^{1}}^{1} \cdot Y_{x}dm(x^{2})\right)dm(x^{1})\right)dm(x)\right)\right). \label{eq:9-12} 
\end{align}}%





Thanks to (\ref{eq:9-12}) and assumptions (\ref{eq:8-600}), (\ref{eq:8-601}), (\ref{eq:8-602}), (\ref{eq:8-603}),
we can see that formula (\ref{eq:9-11}) is valid. This completes
the proof.

\end{document}